\documentclass[11pt]{article}
\usepackage[utf8]{inputenc}   
\usepackage[left=2.0cm,right=2.0cm,top=2.2cm,bottom=2.2cm]{geometry}
\usepackage[T1]{fontenc}
\usepackage[english]{babel} 
\usepackage{graphicx}
\usepackage{amsmath,amsfonts,amssymb,amsthm}  
\usepackage{xspace}
\usepackage{geometry} 
\usepackage[overload]{empheq}
\usepackage{dsfont} 
\usepackage{stmaryrd}
\usepackage{tikz} 
\usepackage{mathrsfs}
\usepackage{enumerate}
\usepackage{subfig}
\usepackage{float}
\usepackage[breaklinks=true,bookmarks=false]{hyperref}
\usepackage{enumitem}

\usepackage{authblk}

\usepackage[f]{esvect}

\usepackage{tikz}
\usetikzlibrary{calc,arrows}

\newcommand{\vect}[2]{\vv{#1}_{\!#2}}
\usepackage{accents}

\newcommand\restr[2]{{
  \left.\kern-\nulldelimiterspace 
  #1
  \littletaller
  \right|_{#2}
  }}

\newcommand{\littletaller}{\mathchoice{\vphantom{\big|}}{}{}{}}

\usepackage{hyperref}
\hypersetup{
colorlinks = true,
urlcolor = blue,
linkcolor = blue,
citecolor = blue,
}

\newcommand{\1}{\mathbbm{1}}
\renewcommand{\P}{\mathbb{P}}
\newcommand{\R}{\mathbb{R}}
\newcommand{\E}{\mathbb{E}}
\newcommand{\B}{\mathrm{B}}

\newcommand{\indep}{\perp \!\!\! \perp}

\usepackage{amsmath,amssymb,euscript ,yfonts,psfrag,latexsym,dsfont,graphicx,bbm,color,amstext,wasysym,epsfig,subfig,parskip,textcomp}

\newtheorem{theorem}{Theorem}
\newtheorem*{informalmain}{Informal main results}
\newtheorem{definition}[theorem]{Definition} 
\newtheorem{proposition}[theorem]{Proposition}
\newtheorem{lemma}[theorem]{Lemma}

\newtheorem{conjecture}[theorem]{Conjecture}
\newtheorem{remark}[theorem]{Remark}

\newtheorem*{assumHbeta}{Properties $(\mathcal{P}_\beta)$}

\newtheorem{maintheorem}{Theorem}

\usepackage{mathrsfs}
\usetikzlibrary{arrows}

\usepackage{natbib}
\bibpunct{(}{)}{;}{a}{,}{,}
\usepackage{hyperref}
\hypersetup{
colorlinks = true,
urlcolor = blue,
linkcolor = blue,
citecolor = blue,
}

\newtheorem{assumption}{Assumption}

\DeclareMathOperator*{\argmin}{arg\,min}

\makeatletter
\def\thm@space@setup{%
    \thm@preskip=4pt plus 2pt minus 2pt 
    \thm@postskip=0pt plus 2pt minus 2pt 
}
\makeatother

\title{
Generalization bounds for score-based generative models:
\\
a synthetic proof
}

\author[1,2]{Arthur Stéphanovitch\thanks{\href{mailto:arthur.stephanovitch@ens.fr}{arthur.stephanovitch@ens.fr} 
--- 
\href{https://sites.google.com/view/arthur-stephanovitch/}{https://sites.google.com/view/arthur-stephanovitch/}}}
\author[1]{Eddie Aamari\thanks{\href{mailto:eddie.aamari@ens.fr}{eddie.aamari@ens.fr} 
---
\href{https://www.math.ens.psl.eu/~eaamari/}{https://www.math.ens.psl.eu/${\sim}$eaamari/}}}
\author[3]{Clément Levrard\thanks{\href{mailto:clement.levrard@univ-rennes1.fr}{clement.levrard@univ-rennes1.fr} 
--- 
\href{https://www.normalesup.org/~levrard/}{https://www.normalesup.org/${\sim}$levrard/}}}

\affil[1]{Département de Mathématiques et Applications --- École Normale Supérieure, U. PSL, CNRS 
\newline Paris, France
} 
\affil[2]{Laboratoire de Probabilités, Statistiques et Modélisation --- U. Paris Cité, Sorbonne U., CNRS 
\newline Paris, France
}
\affil[3]{Institut de Recherche Mathématique de Rennes --- U. de Rennes, CNRS 
\newline Rennes, France
}

\date{}

\begin{document}

\maketitle

\begin{abstract}
\noindent 
We establish minimax convergence rates for score-based generative models (SGMs) under the $1$-Wasserstein distance.
Assuming the target density $p^\star$ lies in a nonparametric $\beta$-smooth H\"older class with either compact support or subGaussian tails on $\mathbb{R}^d$, we prove that neural network-based score estimators trained via denoising score matching yield generative models achieving rate $n^{-(\beta+1)/(2\beta+d)}$ up to polylogarithmic factors. Our unified analysis handles arbitrary smoothness $\beta > 0$, supports both deterministic and stochastic samplers, and leverages shape constraints on $p^\star$ to induce regularity of the score. The resulting proofs are more concise, and grounded in generic stability of diffusions and standard approximation theory.
\end{abstract}


\section{Overview
}
\subsection{Generative modeling}

Given data $\mathbb{X}^{(n)} := \{X^{(1)},\ldots,X^{(n)}\}$ sampled from an unknown distribution $p^\star(x)\mathrm{d}x$ on $\mathbb{R}^d$, \emph{generative modeling} is about how to produce some fake $n+1^{\text{st}}$ sample $\widehat{X}^{(n+1)}$ that has approximately distribution $p^\star$.
To do so, a versatile idea consists in drawing an independent uninformative \emph{seed point} $\widehat{X}_0$ from an easy-to-sample distribution $\mu_0(x)\mathrm{d}x$ --- usually uniform or Gaussian---, and then apply to it an informative algorithm $\widehat{\mathtt{A}}$ built from $\mathbb{X}^{(n)}$, and designed specifically so that the pushforward distribution 
$\widehat{p} := \widehat{\mathtt{A}}_{\#}\mu_0$ of $\widehat{X}^{(n+1)} := \widehat{\mathtt{A}}(\widehat{X}_0)$ has approximately distribution $p^\star$.

\subsubsection{Score-based marginal time reversal(s) for diffusions}

To date, \emph{score-based generative models} (SGM) are among the most prominent generation methods. 
They include \textit{Denoising Diffusion Probabilistic Models} (DDPM) and   \textit{Denoising Diffusion Implicit Models} (DDIM), also called \emph{Probability Flow ODE}~\citep{song2020score}. 
These methods design $\widehat{\mathtt{A}}$ via a stochastic differential equation (SDE) formalism, and achieve state-of-the-art performances for all types of data except text~\citep{kong2020diffwave,dhariwal2021diffusion,hoogeboom2022equivariant}.
In these methods, the core idea consists in designing a deterministic or stochastic forward dynamic $(\vect{X}{t})_t$, called \emph{forward process}, with initial distribution $\vect{X}{0} \sim p$ and terminal distribution $\vect{X}{\overline{T}} \approx \vect{p}{\infty}$.
Then, if the marginal distributions $(\vect{p}{t})_{0 \leq t \leq \overline{T}}$ was fully known, one may be able to run a \emph{backward process} $(X_t)_{0 \leq t \leq \overline{T}}$ such that starting from $X_0 \approx \vect{p}{\infty}$, $\widehat{\mathtt{A}}(X_0) := X_{\overline{T}}$ is approximately distributed as $\vect{p}{0} = p$.

More precisely, given a user-defined \emph{backward diffusion schedule} $(b_t)_{0 \leq t \leq \overline{T}}$, the Fokker-Planck equation driving the dynamic at the level of probability densities allows to get the following stochastic dynamic, which involves the so-called \emph{score function} $\nabla \log \vect{p}{t}$ of the forward dynamic.
Throughout this article, notation $(B_t)_t$ stands for a $d$-dimensional isotropic Brownian motion independent from the data, as well as the initial conditions of the SDEs to come.
We refer the reader to~\cite{le2016brownian} for the necessary background on Itô calculus and SDEs.

\begin{theorem}[Marginal reversal(s) of an SDE,~\cite{anderson1982reverse}]
\label{thm:backward-SDE}
If the solution to $\mathrm{d} \vect{X}{t} = \vect{a}{t}(\vect{X}{t}) \mathrm{d} t + \sqrt{2} \vect{b}{t} \mathrm{d} B_t$ has distribution $\vect{X}{t} \sim \vect{p}{t}(x) \mathrm{d} x$, then the solution to
\begin{align*}
\begin{cases}
\mathrm{d} X_t
=
a_t(X_t) \mathrm{d} t 
+
\sqrt{2} b_t \mathrm{d} B_t
\\
X_0 \sim \vect{p}{\overline{T}}(x) \mathrm{d} x
\end{cases}
\end{align*}
with 
$a_t := -\vect{a}{\overline{T}-t} + (\vect{b}{\overline{T}-t}^2 + b_t^2) \nabla \log \vect{p}{\overline{T}-t}$
satisfies
$
X_t
\sim 
\vect{X}{\overline{T}-t}
.$
\end{theorem}

The above result highlights the fact that reversing an Itô stochastic process \emph{marginally} leaves lots of degrees of freedom in the \emph{backward diffusion coefficient} $b_t$.
For instance, the founding idea of DDIM consists in choosing $b_t := 0$, while DDPM takes instead $b_t :=\vect{b}{\overline{T}-t}$.
Many other variants/combinations co-exist in practice, for instance with profiles $b_t$ alternating between the two previous choices~\citep{karras2022elucidating}.

\subsubsection{Ornstein-Uhlenbeck forward,  backward, and approximate backward}
\label{sec:OU-generation}

This work will only consider so-called \emph{Variance-Preserving SDE} for the forward process, meaning that it corresponds to a time-changed \emph{Ornstein-Uhlenbeck} process.
Conversely, to cover various popular sampling strategies, it leaves the backward diffusion term $b$ depend freely on time $t$. 

\paragraph{True forward}
Given a \emph{diffusion coefficient} $\sigma > 0$, write $(\vect{X}{t})_t$ for the solution to the Ornstein-Uhlenbeck process
\begin{align}
\label{eq:forward-true}
\begin{cases}
\mathrm{d} \vect{X}{t}
=
- \vect{X}{t} \mathrm{d} t 
+
\sqrt{2}\sigma \mathrm{d} B_t
\\
\vect{X}{0} \sim p^\star(x) \mathrm{d} x
,
\end{cases}
\end{align}
where $p^\star$ is the (unknown) data distribution.
One easily shows that the above SDE has a unique solution given by
\begin{align}
\label{eq:law-at-time-t}
    \vect{X}{t} 
    &=
    \vect{X}{0} e^{-t} + \int_{0}^t \sqrt{2\sigma^2} e^{s-t} \mathrm{d} B_s
    .
\end{align}
If $\vect{X}{0} \indep Z$ with $Z \sim \mathcal{N}(0,I_{d \times d})$, marginals $\vect{X}{t}$ can then be represented in distribution as
\begin{align}\label{align:lawofxt}
    \vect{X}{t} \sim e^{-t}\vect{X}{0} + \sigma_t Z
    ,
\end{align}
where 
\begin{align}
\label{eq:sigma_t}
 \sigma_t
 &:= 
\sqrt{
    1 - e^{-2t}
}
\sigma
.
\end{align}
In what follows, we will denote $\vect{X}{t} \sim \vect{p}{t}(x) \mathrm{d} x$,
so that
$\vect{p}{0} = p^\star$ and $\vect{p}{\infty}(x) \mathrm{d}x := \mathcal{N}(0,\sigma^2 I_{d\times d})$.

\begin{remark}[On seemingly more general Ornstein-Uhlenbck processes]
\label{rem:general-OU}
Authors and practitioners sometimes consider dynamics of the form $\vect{Y}{0} \sim p^\star(x) \mathrm{d} x$ and
$
\mathrm{d} \vect{Y}{t}
=
- \lambda \alpha_t \vect{Y}{t} \mathrm{d} t 
+
\sqrt{2 \alpha_t }\sigma \mathrm{d} B_t
,
$
for general \emph{drift coefficient} $\lambda > 0$,  and deterministic \emph{noise schedule} $\alpha : \R_{+} \to \R_{>0}$.
For $\lambda = 1$, one easily sees that $\vect{Y}{t} \sim \vect{X}{u_t}$ with $u_t := \int_{0}^t \alpha_s \mathrm{d} s$. 
Therefore, all the results below extend to such a general framework as soon as $u_t \xrightarrow[t \to \infty]{} \infty$.
In fact, statements are essentially identical, differing only by replacing $t$ with $u_t$ and $\sigma^2$ with $\sigma^2/\lambda$.
\end{remark}

\paragraph{True backward}
Let $\overline{T} >0$ be a finite \emph{time horizon} to be specified later.
For short, let us denote the true \emph{score function} by
\begin{equation}\label{eq:scfunc}
    s^\star(t,x):= \nabla \log \vect{p}{t}(x).
\end{equation}

Given a deterministic \emph{backward diffusion coefficient} $b : [0,\overline{T}] \to \R_{+}$, write $(X_t)_t$ for the solution to
\begin{align}
\label{eq:backward-true}
\begin{cases}
\mathrm{d} X_t
=
a_t(X_t) \mathrm{d} t 
+
\sqrt{2} b_t \mathrm{d} B_t
\\
X_0 \sim \vect{p}{\overline{T}}(x) \mathrm{d} x
,
\end{cases}
\end{align}
where $a_t(x) := x + \bigl( \sigma^2 + b_t^2 \bigr) s^\star(\overline{T}-t,x)$.
From Theorem~\ref{thm:backward-SDE}, we have $X_t \sim \vect{X}{\overline{T}-t}$ for all $t \in [0,\overline{T}]$.
For this reason, we will write $p_t := \vect{p}{\overline{T}-t}$, so that $X_t \sim p_t(x) \mathrm{d}x$.

\paragraph*{Approximate backward}
For unknown distribution $p^\star = \vect{p}{0}$, the idea of SGMs is to run a plugin SDE based on~\eqref{eq:backward-true}.
To avoid zones where $s^\star(t,\cdot)$ blows up (see Section~\ref{sec:scoreapp}), let us fix $\underbar{T} \leq \overline{T} $ be a positive \emph{early stopping time} to be specified later.
Given some approximate function $\widehat{s} : [\underbar{T},\overline{T}] \times \R^d \to \R^d$ meant to approximate $s^\star$, consider the surrogate SDE
\begin{align}
\label{eq:final-generator}
\begin{cases}
\mathrm{d} \widehat{X}_t
=
\widehat{a}_t(\widehat{X}_t) \mathrm{d} t 
+
\sqrt{2} b_t \mathrm{d} B_t
\\
\widehat{X}_0 \sim \vect{p}{\infty}(x) \mathrm{d} x
,
\end{cases}
\end{align}
where $\widehat{a}_t(x) :=  x + \bigl( \sigma^2 + b_t^2 \bigr) \widehat{s}(\overline{T}-t,x)$.
After running the SDE until time $\overline{T}-\underline{T}$, the new \emph{fake} sample that is output by the method is 
$$
\widehat{X}^{(n+1)} := \widehat{X}_{\overline{T}-\underline{T}}
.
$$
In what follows, we will write $\widehat{X}_t \sim \widehat{p}_{t}(x) \mathrm{d} x$.

At this point, note that when approximating the distribution of $\vect{X}{0}$ with that of $\widehat{X}_{\overline{T}-\underline{T}}$, error in distribution arises from the three quite different plug-in tricks. 
Indeed, keeping in mind that we would like to have $\widehat{X}_t \approx \vect{X}{t}$ in distribution:
\begin{itemize}[leftmargin=*]
    \item 
    For $\overline{T}< \infty$, their respective starting distributions $\vect{p}{\infty}$ and $\vect{p}{\overline{T}}$ do not exactly coincide, hence inducing bias.

    \item 
    For $\widehat{s} \neq s^\star$, their respective drifts differ, hence leading to different dynamics.
    
    \item
     For $\underline{T}>0$, the early stopping make our process incomplete, hence adding another bias term.
\end{itemize}
These error terms will be made precise in Section~\ref{sec:stability-of-sdes}.
The first and third error terms can be dealt with pretty easily through the explicit expression~\eqref{eq:law-at-time-t}, and do not rely on any particular training.
It is the second error term which requires an actual pre-training phase from data $\mathbb{X}^{(n)}$, which we now address.

\subsubsection{Score estimation}

As will be clear in Section~\ref{sec:stability-of-sdes}, a natural quantity controlling the closeness of marginals of the target process~\eqref{eq:forward-true} and its approximant~\eqref{eq:final-generator} is the so-called time-integrated \emph{Fisher loss}, defined for any given candidate $s: [\underline{T},\overline{T}] \times \R^d \to \R^d$ as
\begin{align*}
\mathcal{E}_{\underline{T},\overline{T}}(s)
:=
\int_{\underline{T}}^{\overline{T}}
\mathbb{E}
[
\|s^\star(t,\vect{X}{t}) - s(t,\vect{X}{t})\|^2 
]
\, \mathrm{d} t
.
\end{align*}
Despite its native expression involving target $s^\star(t,x)$ explicitly, it can be rewritten as an average  involving $s$ only, with respect to an easily samplable distribution (conditionally on data).

\begin{theorem}[Denoising score trick ---~\cite{vincent2011connection}]
\label{thm:Vincent}
With notation of Section~\ref{sec:OU-generation}, let $t \in [0,\overline{T}]$ be fixed.
Let $X \sim p^\star(x)\mathrm{d}x$ and $Z \sim \mathcal{N}(0,I_{d\times d})$ be independent, and write $\vect{X}{t} := e^{- t}X + \sigma_t Z$ (recall~\eqref{eq:sigma_t}).
Then there exists $C_{p^\star,\sigma_t,d}$ depending only on $p^\star$, $\sigma_t$, and $d$, such that for all integrable candidate score function $s : [0,\overline{T}] \times \R^d \to \R^d$,
\begin{align*}
\mathbb{E}
[
\|s^\star(t,\vect{X}{t}) - s(t,\vect{X}{t})\|^2 
]
=
\mathbb{E} 
[
\|s(t, \vect{X}{t} ) + Z / \sigma_t\|^2 
]
+ 
C_{p^\star,\sigma_t,d}.
\end{align*}
\end{theorem}

In the above context, data $\mathbb{X}^{(n)} = \{X^{(1)}, \dots, X^{(n)}\}$ allows to perform a preliminary fit of $\widehat{s}(t,x)$ to estimate the true score function $s^\star(t,x)$.
This score estimation step is preliminary to the sampling step~\eqref{eq:final-generator}.
In light of Theorem~\ref{thm:Vincent},
it is then natural to minimize an empirical version of the \emph{score-matching loss}. 
Hence, defining for $x \in \mathbb{R}^d$ the loss function $\gamma_{[\underline{T},\overline{T}]}(s,x)$ by
\begin{align}
\label{eq:contrast}
\gamma_{[\underline{T},\overline{T}]}(s,x)
:= \int_{\underline{T}}^{\overline{T}} \mathbb{E}_{Z\sim \mathcal{N}(0,I_{d \times d})} \|s(t, e^{-t} x + \sigma_t Z) + Z / \sigma_t\|^2 \mathrm{d}t,
\end{align}
the score estimator based on data $\mathbb{X}^{(n)}$ is  defined as a minimizer of an empirical loss function, that is
\begin{equation}\label{eq:scoreestimatorheuristic}
    \widehat{s}\in \argmin_{s\in \mathcal{S}} \frac{1}{n} \sum_{i=1}^n
    \gamma_{[\underline{T},\overline{T}]}(s,X^{(i)})
,
\end{equation}
where $\mathcal{S}$ is a user-defined class of candidate functions $s:[\underline{T},\overline{T}] \times \R^d \to \R^d$, which will be taken as neural networks eventually.

\begin{remark}[On numerical considerations]
Throughout, we will assume to have access to an exact integrator with respect to measure $g_d(\mathrm{d} z) \mathrm{d}t$, where $g_d = \mathcal{N}(0,I_{d \times d})$.  Therefore, the values of $\gamma_{[\underline{T},\overline{T}]}(s,X^{(i)})$ are fully available given evaluations of $s$, and hence $\widehat{s}$ is indeed an estimator in the usual statistical sense.
However, in practice, errors due to additional discretizations of time and Gaussian Monte Carlo simulations need to be taken into account.
Such a point also applies to the optimization method used to approximate the  minimum of the highly non-convex minimization problem~\eqref{eq:scoreestimatorheuristic}.
All our results extend, in a quantitative way, to the case where~\eqref{eq:scoreestimatorheuristic} is solved up to small enough tolerance.
As this work focuses on the statistical aspects of SGMs, we shall hence consider an idealized framework where~\eqref{eq:scoreestimatorheuristic} is solved exactly given a function class~$\mathcal{S}$. 
See~\cite{zhang2023fast} for more on practical numerical considerations.
\end{remark}

\subsection{Convergence rates for SGMs}
\subsubsection{Prior works}
\label{sec:priro-works}

In recent years, several works have established minimax convergence rates for score-based generative models (SGMs), under various distributional assumptions, evaluation metrics, and with different score estimation strategies.
Let us summarize those most relevant to our context, breaking them into SDE sampling ($b_t > 0$ in~\eqref{eq:final-generator}) and ODE sampling ($b_t = 0$ in~\eqref{eq:final-generator}).

\begin{itemize}[leftmargin=*]
    \item For stochastic sampling methods:

\begin{itemize}[leftmargin=*]
    \item~\cite{oko2023diffusion} is the first article showing that SGMs achieve near-minimax rates for density estimation.
    The authors consider densities belonging to the Besov space $\mathcal{B}^\beta_{p,q}$ defined on $[0,1]^d$. These densities are assumed to be bounded below and $\mathcal{C}^\infty$ near the boundary. 
    This paper works both in $\mathrm{W}_1$ and $\mathrm{TV}$ distances.
    Despite a flaw in one of their core concentration technical result (Lemma C.4), its fixability (see Proposition~\ref{prop:tradeoffbv}) ensures that this work remains seminal.

    \item In~\cite{Zhang24},
the authors study subGaussian densities belonging to a Hölder space $\mathcal{H}^\beta$ on $\mathbb{R}^d$ with $\beta \leq 2$. They show that SGMs achieve minimax rates for estimating the density in $\mathrm{TV}$ distance, using a kernel method for estimating the score.
They are the first to work without any lower bound assumption on $p^\star$.

\item 
In~\cite{dou2024optimal}, the authors consider densities bounded below in the Hölder space $\mathcal{H}^\beta$ on $[-1,1]$, for general $\beta > 0$. They derive minimax rates for score estimation, which are then translated into minimax rates for density estimation, again using a kernel estimator. In particular, the rates are obtained for the distance $\mathrm{TV}$ without a logarithmic factor.

\item 
In~\cite{kwon2025nonparametric}, the authors consider densities bounded below in $\mathcal{H}^\beta$ on $[-1,1]^d$, for general $\beta > 0$ and that are factorizable. 
This means that $p^\star(x)=\prod_{I\in \mathcal{I}}g_{I}(x_{I})$ for a subset of variable indices $\mathcal{I}\subset 2^{\{1,...,d\}}$ and functions $g_{I}:\mathbb{R}^{|I|}\rightarrow \mathbb{R}$. 
There, SGMs provably achieve the minimax rate in TV distance and that this rate only depends on $d^{'}=\max_{I\in \mathcal{I}} |I|$.

\end{itemize}
\item 
For the deterministic sampling methods:
\begin{itemize}[leftmargin=*]

    \item~\cite{fukumizu2024flow} deals with flow matching, a slightly different framework than classical SGMs. 
    They provide nearly minimax rates for the flow matching version of the probability flow ODE for $\mathrm{W}_r$ distance with $r \in [1,2]$.
    They consider densities bounded below, belonging to a Besov space $\mathcal{B}^\beta$ on $[0,1]^d$, that are even smoother near the boundary. Furthermore, they work under a Lipschitz-type assumption on the chosen flow, amounting to assume that $s^\star(t,\cdot)$ is $(t^{-1})$-Lipschitz when phrased into SGMs. 
    Unfortunately, to date, it should be noted that a  flaw --- in the inner integration bounds on top of p.14 ---, seems to affect the overall correctness of the proof. The recent work of~\cite{kunkel2025minimax} circumvents this issue by using velocity fields with an integrable Lipschitz norm.

     \item~\cite{cai2025minimax} studies subGaussian densities belonging to $\mathcal{H}^\beta$ on $\mathbb{R}^d$ with $\beta \leq 2$ and establishes the first minimax optimality result for Probability Flow ODE. 
     The authors use a kernel-based score estimator, yielding a strong estimation of the score \emph{and} its Jacobian.
    The result is almost minimax in $\mathrm{TV}$ distance.

\end{itemize}
\end{itemize}

It is worth noting that numerous works on the quantitative convergence of diffusion models assume some Lipschitz regularity of the score function~\citep{chen2022sampling,pmlr-v202-chen23q,wibisono2024optimal,lee2023convergence}. Furthermore, this assumption appeared particularly useful for establishing convergence of ODE methods, as it is commonly employed in most related studies~\citep{chen2023probability,benton2023error,huang2024convergence}. In fact, approximating the Jacobian of the score is necessary when working in total variation distance~\cite{li2024sharp}.

The works above lay a strong foundation for understanding the theoretical performance of SGMs across different settings. In this paper, we present a synthetic proof of the minimax optimality of score-based models in the $\mathrm{W}_1$ distance, under both stochastic and deterministic sampling schemes. Our argument builds on regularity properties of the true score established in~\cite{stephanovitch2024smooth}, which allows for a significantly more concise proof.

For completeness, we also note a series of related works that address distributions supported on manifolds~\citep{pmlr-v238-tang24a,azangulov2024convergence,yakovlev2025generalization}. Since our goal here is to unify and simplify existing results and proofs, we concentrate on the simpler setting of distributions with densities. Nonetheless, in Section~\ref{sec:perspective}, we discuss how our approach could potentially extend to distributions supported on lower-dimensional submanifolds.

\subsubsection{Main contribution}
\label{sec:main-contribution}

Throughout, we will work under two complementary types of assumptions on the targeted density $p^\star:\R^d \to \R_+$ with $d\geq 3$. Namely, a classical Hölder regularity constraint, and a shape constraint. 
When combined, they ensure quantitative regularity of the score function, and hence make the overall pipeline easily analyzable from generic stability of SDEs, as well as concentration and approximation arguments~  (Section~\ref{sec:two-step-analysis}).

Before diving into the technical details, we provide an informal version of our main results. 
In what follows, the $1$-Wasserstein distance between two probability distributions $P$ and $Q$ on $\R^d$ is denoted by
\begin{align*}
    \mathrm{W}_1(P,Q)
    :=
    \sup_{f \in \mathrm{Lip}_1(\R^d,\R)}
    \left\{
    \int_{\R^d} f(x) P(\mathrm{d} x)
    -
    \int_{\R^d} f(x) Q(\mathrm{d} x)
    \right\}
    ,
\end{align*}
and we shall identify distributions $P(\mathrm{d}x) = p(x)\mathrm{d}x$ with their Lebesgue density
$p$ for sake of brevity. \begin{informalmain}
    Given $\beta > 0$, let $\mathcal{A}^\beta$ denote the model consisting of all the subGaussian $\beta$-Hölder probability densities $p^\star : \Omega \to \R_{+}$
    defined:
    \begin{itemize}[leftmargin=*]
        \item 
        either on a compact convex set $\Omega$, with a tempered decay at the boundary;
        \item 
        or on $\Omega = \mathbb{R}^d$, with a tempered decay at infinity.
        \end{itemize}

    Let $\mathbb{X}^{(n)} := \{ X^{(1)},\ldots,X^{(n)}\}$ be drawn i.i.d. from some unknown $p^\star \in \mathcal{A}^\beta$.
\begin{itemize}[leftmargin=*]
    \item \textit{(Upper bound)}
    Let $\widehat{s}$ be the estimated score obtained by score matching~\eqref{eq:scoreestimatorheuristic} over a well-chosen neural network class $\mathcal{S}$ with a $\tanh$ activation function.
    Then for well-chosen $0 < \underline{T} < \overline{T} < \infty$, the distribution $\widehat{p}_{\mathrm{SGM}}$ of the early-stopped backward diffusion $ \widehat{X}_{\overline{T}-\underline{T}}$ from~\eqref{eq:final-generator} satisfies
    \begin{align}
    \tag{Theorem~\ref{thm:cv_rates}}
 \mathbb{E}
 \left[\mathrm{W}_1(p^\star,\widehat{p}_{\mathrm{SGM}})\right] 
 \underset{{\mathrm{polylog}(n)}}{\lesssim}
 n^{-\frac{\beta+1}{2\beta + d}}
 ,
\end{align}
where the expectation is taken with respect to $\mathbb{X}^{(n)}$.

\item
\textit{(Lower bound)}
This rate is minimax optimal over model $\mathcal{A}^\beta$ up to $\mathrm{polylog}(n)$ factors, in the sense that
\begin{align}
    \tag{Theorem~\ref{thm:borneinf}}
    \inf_{\widehat{p}_n} \sup_{p^\star \in \mathcal{A}^\beta}\ \mathbb{E}[\mathrm{W}_1(p^\star,\widehat{p}_n)] 
     \gtrsim 
     n^{-\frac{\beta+1}{2\beta + d}}.
\end{align}

\end{itemize}

\end{informalmain}

\subsubsection{Comments}

\paragraph{Strength of the result}
As mentioned in Section~\ref{sec:priro-works},
Theorem~\ref{thm:cv_rates} is not the first instance of a minimax analysis of SGMs.
However, our overall approach benefits from the following features, including several that offer fresh perspectives.
\begin{itemize}[leftmargin=*]

    \item \textbf{Unification.}
    We present a unified proof that simultaneously covers arbitrary levels of smoothness as well as both SDE and ODE sampling schemes—that is, for general $\beta > 0$ and potentially time-varying diffusion coefficients $b_t \geq 0$ in~\eqref{eq:final-generator}.
    Since the regime where $\beta > 2$ and $b_t = 0$ had not been previously analyzed in combination, our result also incidentally extends the existing literature.
    
    \item \textbf{Generality of the statistical model.}
    We deal with both compactly-supported distributions \emph{and} distributions with unbounded support in the same framework (see Assumption~\ref{assum:1}). 
    In particular, the case of full-supported densities with $\beta > 2$ appears to be new to the literature.

    \item \textbf{Proof conciseness.}
    A notable advantage of our result is that the regularity of the score function is not required explicitly. 
    Instead, this regularity is naturally induced by the regularity of $p^\star$ itself (Theorem~\ref{coro:regularityp}).
    This results in a notably more concise proof, mostly coming from simplifications of approximation theory arguments (Proposition~\ref{prop:aproxsuzuki}).

    \item \textbf{Methodology and neural networks.}
    While many prior works rely on score estimators based on smoothing kernels, we analyze SGMs trained via denoising score matching (i.e empirical risk minimization of the form~\eqref{eq:scoreestimatorheuristic}). 
    Moreover, it is constructed using a class of neural networks, making it more aligned with the practical implementations of SGMs.

    \item \textbf{Transport metric.} 
    We derive bounds in the Wasserstein-1 distance $\mathrm{W}_1$, thereby capturing the geometric structure of the underlying space—an aspect not addressed by information-theoretic metrics such as the Kullback-Leibler divergence or total variation distance.
    
    \item \textbf{Weak score approximation.}
    Unlike for the total variation~\cite{li2024sharp}, we show that approximating the Jacobian of the score is not necessary when working in $\mathrm{W}_1$ distance as a weaker one-sided Lipschitness constraint is sufficient.

    \item \textbf{Nearly tight rates.}
    In addition to the above contributions, we provide convergence rates that, unlike most prior work, are minimax optimal up to $\mathrm{polylog}(n)$ factors only.

\end{itemize}

\paragraph{Limits of the result}
To highlight potential directions for future research, let us also acknowledge a few limitations of our approach.

\begin{itemize}[leftmargin=*]
    \item 
    \textbf{One-sided Lipschitz approximants.}
    The neural networks used for score approximation are constrained to be spatially one-sided Lipschitz, a condition that promotes the stability of the backward process $(\widehat{X}_t)_t$. Although enforcing this constraint in practice can be challenging, several works have proposed effective methods for imposing (even stronger versions of)  it~\citep{ducotterd2024improving,gouk2021regularisation,fazlyab2019efficient,pauli2021training}.
    
\item 
    \textbf{Time-dependency of the neural network architecture.}
    As in previous minimax estimators based on neural networks~\citep{oko2023diffusion,fukumizu2024flow}, our estimator employs different networks across time intervals to exploit the increasing regularity of the score function over time (see Section~\ref{sec:scoreapp}).
    In earlier works, the total number of networks used was of order $O(\log n)$.
    Here, we demonstrate that by allocating more neural networks at later time steps (i.e., for larger forward times $t$), we can achieve minimax rates up to logarithmic factors, rather than merely near-optimal rates.
    However, this refinement requires a total of $O(n^{\frac{2}{2\beta + d}})$ neural networks.
    Restricting the number of networks to $O(\log n)$ would still yield almost minimax rates, following a similar proof strategy.

\end{itemize}
\subsection{Organisation of the paper}

The remainder of the paper is organized as follows. Section~\ref{sec:statframe} introduces the statistical model and details the regularity properties it induces on the score function. Section~\ref{sec:two-step-analysis} presents a two-step analysis of score-based generative models based on the stability of the backward dynamic, and generalization bounds for score matching. 
In Section~\ref{sec:score-regularity-implies-score-estimation}, we detail how regularity properties of the score enables to construct estimators via neural networks, and to establish precise approximation guarantees. 
Section~\ref{sec:lemma:minimaxforsmoothscore} combines the analytical and statistical ingredients to derive the minimax convergence rate. 
Finally, Section~\ref{sec:concluandpersp} discusses the optimality of this rate and outlines directions for future work.

\section{Statistical framework}\label{sec:statframe}
In this section, we introduce the statistical model under consideration in the minimax result~(Section~\ref{sec:model-setup}). Then, we describe the regularity properties that this model induces on the score function~(Section~\ref{sec:model-properties}).
First, let us start by setting out the notations that will be used in the sequel.
\subsection{General notation}
Throughout, $d\geq 3$ stands for the ambient dimension and $p^\star:\R^d \to \R_+$ for the unknown target density. 
Symbols $\langle \cdot , \cdot \rangle$ and $\| \cdot \|$ will denote the canonical scalar product and Euclidean norm.
The closed ball with radius $r \geq $ and center $x \in \R^d$ is denoted by $\B(x,r)$.
The largest eigenvalue of a symmetric matrix $A \in \R^{d \times d}$ will be denoted by 
\begin{align*}
    \lambda_{\max}(A)
    :=
    \max_{\|v\|=1}
    \langle A v , v \rangle
    .
\end{align*}
For $\eta \geq 0$, define
$
\lfloor \eta \rfloor := \max\{k \in \mathbb{N}_0 \mid k \leq \eta\}$. 
Let $\mathcal{X} \subset \R^d$, $\mathcal{Y} \subset \R^p$ be subsets of Euclidean spaces,
and let $(f_1, \ldots, f_p) =: f : \mathcal{X} \to \mathcal{Y}$ belong to $\mathcal{C}^{\lfloor \eta \rfloor}(\mathcal{X}, \mathcal{Y})$, meaning that  it is $\lfloor \eta \rfloor$-times continuously differentiable.
For all multi-index $\nu = (\nu_1, \ldots, \nu_d) \in \mathbb{N}_0^d$ with $|\nu| := \nu_1 + \cdots + \nu_d \leq \lfloor \eta \rfloor$, denote the corresponding partial derivative of each component $f_j$ by
$
\partial^\nu f_j := \frac{\partial^{|\nu|} f_j}{\partial x_1^{\nu_1} \cdots \partial x_d^{\nu_d}}.
$
Writing
$
\|f_j\|_{\alpha} := \sup_{x \neq y} \frac{|f_j(x) - f_j(y)|}{\|x - y\|^{\alpha}}
$
,
we define
\begin{align*}
    \| f \|_{\mathcal{H}^\eta}
    :=
    \max_{1 \leq j \leq d} \Bigl\{ \sum_{0\leq |\nu| \leq \lfloor \eta \rfloor} \|\partial^\nu f_j\|_{L^\infty(\mathcal{X})} + \sum_{|\nu| = \lfloor \eta \rfloor} \|\partial^\nu f_j\|_{\eta - \lfloor \eta \rfloor} \Bigr\}.
\end{align*}
Then, for $K > 0$, define the Hölder ball of regularity $\eta$ and radius $K$ by
\[
\mathcal{H}^\eta_K(\mathcal{X}, \mathcal{Y}) := \left\{ f \in \mathcal{C}^{\lfloor \eta \rfloor}(\mathcal{X}, \mathcal{Y}) \ \middle| \ \| f \|_{\mathcal{H}^\eta} \leq K \right\}.
\]
We use $\nabla^k$ and $\partial_t^k$ to denote the $k$-th order differential operators with respect to the space variable $x \in \mathbb{R}^d$ and the time variable $t \in \mathbb{R}_+$, respectively.

For a measure $\mu$ on $\mathbb{R}^d$, we write $L^2(\mu)$ for the set of functions $f$ such that $\int f(x)^2\mu(\mathrm{d} x)<\infty$. Given $\kappa \geq 0$, a $\R^d$-valued random variable $Y$ will be said to be \emph{$\kappa$-subGaussian} 
if for all $0 < \varepsilon < 1$,
\begin{align}
    \label{eq:subgaussian}
    \mathbb{P} \left ( \|Y\| \geq \sqrt{\kappa\log(\varepsilon^{-1})} \right ) \leq \varepsilon.
\end{align}
Finally, $C,C_1,C_2, \ldots$ will denote generic constants independent of $n$ and $p^\star$, which may vary from line to line.

\subsection{Model setup}\label{sec:model-setup}
Given $\beta>0$ and $K\geq 1$, we now introduce the assumptions on the target data distribution $p^\star$.
\begin{assumption}\label{assum:1}
The probability density function $p^\star:\R^d\rightarrow \R_+$ writes as
$p^\star(x)=\exp(-u(x)+a(x))$, with $u:\R^d \to \R \cup \{\infty\}$ and $a:\R^d \to \R$.
Furthermore:
\begin{enumerate}
    \item $a \in \mathcal{H}^{\beta}_K\bigl(\mathrm{Support}(p^\star)\bigr)$ and $u\in \mathcal{C}^{1+ \max\{1,\lfloor \beta \rfloor\}}\bigl(\mathrm{Support}(p^\star)\bigr)$;
    \item $\nabla^2 u \succeq K^{-1} I_{d \times d}$;
    \item
    $\|\argmin u\| \leq K$;
    \item 
    For all $x,y \in \mathbb{R}^d$ such that $\|x-y\|\leq (K\{1 \vee |u(x)|\})^{-K}$ and $k\in \{0,...,\lfloor \beta \rfloor+1\}$, we have $\|\nabla^k u(y)\|\leq K(1+|u(x)|^K)$
    .
\end{enumerate}
\end{assumption}

Informally, Assumption~\ref{assum:1} covers \( \beta \)-Hölder densities that are bounded below on a given compact subset, and that exhibit a controlled decay at the boundary of their (necessarily convex) support.
Given $\beta>0$, $K\geq 1$, we define our statistical model as follows:
\begin{equation}\label{eq:finalstatisticalmodel}
    \mathcal{A}_K^\beta
    :=
    \big\{p^\star:\mathbb{R}^d\rightarrow \mathbb{R}\ \big|\ p^\star \text{ satisfies Assumption~\ref{assum:1} with } \beta \text{ and } K\big\}.
\end{equation}

Assumptions~\ref{assum:1}.1 and~\ref{assum:1}.2 ensure that $p^\star$ is a deformation of a log-concave measure $\exp(-u)$ by a $\beta$-smooth perturbation $a$. 
These conditions ensure that $p^\star$, although not necessarily log-concave itself, is both subGaussian and \( \beta \)-Hölder. 

Assumption~\ref{assum:1}.3 guarantees that the bulk of the probability mass of \( p \) is concentrated near the origin, which in turn ensures that under the forward process~\eqref{eq:forward-true}, mass transported towards the \emph{centered} gaussian does not need to travel long distances.

Assumption~\ref{assum:1}.4 prevents $u$ to blow up arbitrarily fast, or equivalently, $p^\star$ to decay too abruptly towards zero.
Although somewhat unconventional, this assumption is in fact quite mild and versatile. 
It is designed so as to accommodate both compactly supported and full-support distributions $p^\star$.
Given a compact convex subset \( D \subset \R^d \), it is fulfilled for potentials of the form \( u(x) = \mathrm{dist}(x, D^c)^{-R} \) with $R>0$ and \( \mathrm{dist}(\cdot, D^c) \) standing for the distance to the complement of $D$. 
It is sufficiently mild to also encompass distributions with full support $\R^d$, such as those corresponding to any strictly convex polynomial \( u \).

To showcase the versatility of Assumption~\ref{assum:1}, we note that it accommodates distributions $p^\star$ corresponding to perturbations of mixtures of non-degenerate Gaussians with modes lying within a prescribed ball.
In this case, the perturbation \( a \) can absorb the non-convex behavior of \( p \) inside the ball, while the convexity of \( u \) governs the Gaussian behavior outside.

\begin{proposition}[Variations of Gaussian mixtures belong to the model]
\label{prop:mixtures-in-model}
Let $L\in \mathbb{N}_{>0}$ and \( A \geq 1 \). Consider independent Gaussian random variables $Y_1,...,Y_L$ such that for all \( l \in \{1, \ldots, L\} \), 
\begin{itemize}[leftmargin=*]
    \item 
\( Y_l \sim \mathcal{N}(\mu_l, \Sigma_l) \),
    \item 
\( \mu_l \in \B(0,A) \), and
    \item 
    \( A^{-1} I_{d \times d}\preceq\Sigma_l^{-1} \preceq A I_{d \times d} \).
\end{itemize}
Denote by \( q \) the density of the mixture random variable \( \sum_{l=1}^L Y_l\mathds{1}_{Z=l} \), where $Z$ is an independent multinomial random variable of parameters $\{\alpha_1,...,\alpha_L\}$.

Then, there exists \( C_A > 0 \) such that for all function \( a \in \mathcal{H}^{\beta}_{C_A} \), $q e^{a}$ satisfies Assumption~\ref{assum:1}
with parameters $\beta$ and $K_{A}$ provided that $\int qe^a = 1$.
\end{proposition}
The proof of Proposition~\ref{prop:mixtures-in-model} can be found in Section~\ref{sec:prop:mixtures-in-model}.

\subsection{Model properties}
\label{sec:model-properties}
Under Assumption~\ref{assum:1}, the score function $s^\star(t,x) = \nabla \log \vect{p}{t}(x)$ naturally inherits regularity from both the data distribution $\vect{p}{0} = p^\star$ along the forward (diffusive) dynamic, but also from that of its limiting distribution $\vect{p}{\infty} = \mathcal{N}(0,\sigma^2 I_{d \times d})$.

\paragraph{Boundedness of the score} We begin by establishing a general bound on the norm of the score function.

\begin{proposition}[Boundedness of the score]
\label{prop:estimatenormscore}
Suppose that $p^\star$ satisfies Assumptions~\ref{assum:1}.1,\ref{assum:1}.2 and~\ref{assum:1}.3. Then
for all $t>0$ and $x\in \mathbb{R}^d$, we have
\begin{align}\label{eq:estimatenormscore}
\|s^\star(t,x)\|
&\leq C(1+t^{-1})(1+\|x\|)
.
\tag{A}
\end{align}
\end{proposition}
The proof of Proposition~\ref{prop:estimatenormscore} is provided in Section~\ref{sec:prop:estimatenormscore}. 
This estimate is sharp for small time values $t \to 0^+$. For instance, when \( p \) is compactly supported, the backward process~\eqref{eq:backward-true} must transport any point \( x \in \R^d \) into the support of \( p \) within time \( t \), hence justifying the inverse dependence on \( t \). 
Importantly, this bound allows us to control the growth of the score as \( \|x\| \to \infty \), where no accurate approximation properties of a neural network can be expected.

\paragraph*{One-sided Lipschitzness of the score}
Additionally, the largest eigenvalue of the Jacobian of the score turns out to be uniformly integrable over time. 
In fact, in our context, the following strong quantitative upper bound holds true.

\begin{proposition}[One-sided Lipschiztness of the score --- {\cite[Theorem~1]{stephanovitch2025regularity}}]
\label{prop:lipregu}
Suppose that $p^\star$ satisfies Assumptions~\ref{assum:1}.1 and~\ref{assum:1}.2.
Then for all $t>0$ the score function $s^\star$~\eqref{eq:scfunc} associated to equation~\eqref{eq:forward-true} satisfies 
\begin{equation}\label{eq:onesidedlipscoreassum}
\sup_{x\in \mathbb{R}^d} \lambda_{\max}\Bigl(  \nabla s^\star(t,x)+\frac{1}{\sigma^2}I_{d \times d}\Bigr)\leq Ce^{-2t}(1+t^{-(1-\frac{\beta\wedge 1}{2})}).
\tag{B}
\end{equation}
\end{proposition}
For large $t$, such a result is to be expected in the regime where $\nabla s^\star(t,x) \simeq \nabla s^\star(\infty,x) = -\frac{1}{\sigma^2}I_{d \times d}$. Proposition~\ref{prop:lipregu} provides an explicit non-asymptotic exponential decay uniformly over space.
As will become clear in Theorem~\ref{thm:stabilitysde}, the integrability of the right-hand side is a key property for ensuring the stability of the backward equation for arbitrary (bounded) diffusion coefficient $(b_t)_t$.
Roughly speaking, it enables the use of a Grönwall-type argument to control the Wasserstein distance between the law of the backward process and that of its approximation.

\paragraph*{Higher-order score regularity}

Leveraging the controlled decay of \( p^\star \) given by Assumption~\ref{assum:1}.4, the \( \beta \)-Hölder regularity of \( p^\star \) carries over and enhances the high-order regularity of the score function for $t>0$ on sets of large mass.
More precisely, on high-probability subsets with respect to distribution \( \vect{p}{t} \), the score function is \( \gamma \)-Hölder continuous for all \( \gamma \geq 0 \), with associated norm bounds depending explicitly on time and the regularity of \( p^\star \).

\begin{theorem}[{\cite[Corollary~3]{stephanovitch2025regularity}}]
\label{coro:regularityp}
Suppose that $p^\star$ satisfies Assumption~\ref{assum:1}.
Then for all $t> 0$ and $\varepsilon\in (0,1/4)$, the score function $s^\star$~\eqref{eq:scfunc} associated to equation~\eqref{eq:forward-true} satisfies the following.
There exists a convex subset $A_t^\varepsilon \subset \B(0,C\log(\varepsilon^{-1})^{C_2})$ with $\int_{A_t^\varepsilon} \vect{p}{t}(z) \mathrm{d}z \geq 1-\varepsilon$ such that for all $x\in A_t^\varepsilon$ and $\gamma\geq 0$,
\begin{equation}\label{eq:reguscoretime}
\Bigl\Vert s^\star(\cdot,x)+\frac{1}{\sigma^2}x\Bigr\Vert_{ \mathcal{H}^{\gamma}([t,\infty))}
\leq
C_\gamma\log(\varepsilon^{-1})^{C_2(1+\gamma)}e^{-t}\left(1+t^{-\big((\frac{1}{2}+\gamma-\frac{\beta}{2})\vee 0\big)}\right)
,
\tag{C}
\end{equation}
and for all $k\in \mathbb{N}_{\geq 0}$, 
\begin{equation}\label{eq:reguscorespace}
\Bigl\Vert
\partial_t^k\left( s^\star(t,\cdot)+\frac{1}{\sigma^2}I_{d \times d}\right)\Bigr\Vert_{\mathcal{H}^{\gamma}(A_t^\varepsilon)}\leq C_{k,\gamma}\log(\varepsilon^{-1})^{C_2(1+k+\gamma)}e^{-t}\left(1+t^{-\big((\frac{1}{2}+k+\frac{\gamma-\beta}{2})\vee 0\big)}\right).
\tag{D}
\end{equation}
\end{theorem}
Such space-time regularity estimates on the score allow us to apply classical results from approximation theory, and to construct neural network architectures of moderate size that provably approximate it well~(see Proposition~\ref{prop:aproxsuzuki}).
Compared to, for instance,~\cite[Theorem~3.1, Lemma~3.6]{oko2023diffusion}, this perspective significantly simplifies the study of the bias-variance trade-off in score matching, and, consequently, the minimax analysis of SGMs as a whole. 

\subsection{An emergent series of score regularity properties}
\label{sec:emergent-regularity-packed}

The quantitative smoothness properties of the score $s^\star(t,x)$ established in the previous section form the foundation for our statistical analysis of score-based generative models.
To fully leverage them in the minimax convergence analysis, it will in fact be crucial to obtain a slightly refined description of the high-probability sets $A_t^\varepsilon$ on which these regularity properties hold.
Indeed, for probabilistic control, the family $(A_t^\varepsilon)_{\varepsilon}$ must behave monotonically in $\varepsilon$, while for approximation purposes, the family $(A_t^\varepsilon)_t$ should remain sufficiently structured to permit uniform approximation of $s^\star(t,x)$ by neural networks.

Let us introduce a compact series of desirable \emph{properties} of the score that shall be used later on. 

\begin{assumHbeta}\phantomsection\label{assump:scoreregu}
The score function $s^\star$  satisfies properties~\eqref{eq:estimatenormscore},~\eqref{eq:onesidedlipscoreassum},~\eqref{eq:reguscoretime} and~\eqref{eq:reguscorespace} with parameter $\beta$. Furthermore, there exists $C_\star>0$ such that for all  $\varepsilon\in (0,1/4)$,  the sets $A_t^\varepsilon$ in~\eqref{eq:reguscoretime} and~\eqref{eq:reguscorespace} satisfy:
\begin{enumerate}
    \item  there exist convex subsets $A_0^\varepsilon , A_\infty^\varepsilon \subset \B(0,C\log(\varepsilon^{-1})^{C_2})$ such that $\mathbb{P}_{X \sim p^\star} (A_\infty^\varepsilon)\geq 1-\varepsilon$ and \begin{align*}
A_t^\varepsilon:= 
\begin{cases}
A_0^\varepsilon 
& \text{if } t\leq C_\star^{-1}\log(\varepsilon^{-1})^{-C_\star},
\\
A_\infty^\varepsilon 
  & \text{if }  t> C_\star^{-1}\log(\varepsilon^{-1})^{-C_\star}.
    \end{cases}
\end{align*}

\item for all $t>0$,
$$A_t^{\varepsilon} \subset A_t^{\varepsilon^2} \quad \mbox{and} \quad \inf_{x\in  A_t^\varepsilon,\ y\in \partial A_t^{\varepsilon^2}} \|x-y\|\geq  C^{-1}\log(\varepsilon^{-1})^{-C_2}.$$
\end{enumerate}
\end{assumHbeta}

Crucially, for an Ornstein-Uhlenbeck forward process with initial distributions $p^\star = \vect{p}{0}$ belonging to the statistical model $\mathcal{A}_K^\beta$, the score $s^\star$ do satisfy all the properties uniformly. 
This is precisely the statement of the following lemma.

\begin{lemma}[$p^\star \in \mathcal{A}^\beta \Rightarrow s^\star \in \mathcal{P}_\beta$]\label{lem:AbetaimpliesHbeta}
If $p^\star \in \mathcal{A}^\beta_K$, then the score function $s^\star(t,x) = \nabla \log \vect{p}{t}(x)$ from~\eqref{eq:scfunc} satisfies Properties~\hyperref[{assump:scoreregu}]{$(\mathcal{P}_\beta)$}
with constants $C,C_\star$ depending only on $d$, $\sigma$, $\beta$ and $K$.
\end{lemma}
The proof of Lemma~\ref{lem:AbetaimpliesHbeta} can be found in Section~\ref{sec:lem:AbetaimpliesHbeta}, in which the precise forms of sets $A_0^\varepsilon$, $A_\infty^\varepsilon$ are given.

Most importantly, we will show in Lemma~\ref{lemma:minimaxforsmoothscore} that for the Ornstein-Uhlenbeck forward process, if the score function satisfies Properties~\hyperref[{assump:scoreregu}]{$(\mathcal{P}_\beta)$}, then score-based generative models achieve minimax convergence rates for density estimation.
This fact highlights the fundamental role of these regularity properties in guaranteeing the optimal statistical performance of SGMs.

Besides the specific framework of this article (i.e. Ornstein-Uhlenbeck forward process with initial condition in $\mathcal{A}^\beta$), we advocate that this series of score properties constitute a base coreset naturally emerging from reasonable forward diffusion schemes starting from $\beta$-smooth distributions $p^\star$. 
We also argue that this coreset is essential for studying SGMs overall.
Indeed:
\begin{enumerate}[leftmargin=*]
    \item[$\beta$.] As detailed in Section~\ref{sec:model-properties}, regularity bounds~\eqref{eq:estimatenormscore},~\eqref{eq:onesidedlipscoreassum},~\eqref{eq:reguscoretime} and~\eqref{eq:reguscorespace} naturally arise from processes \emph{essentially} convolving a $\beta$-smooth initial distribution $p^\star = \vect{p}{0}$ with a gaussian of amplitude $O(\sqrt{t})$.
    These smoothness bounds make the use of standard approximation theory results.
    \item[1.]
    Because $\vect{p}{t}$ interpolates between the two distributions $\vect{p}{0}$ and $\vect{p}{\infty}$, Property~\hyperref[{assump:scoreregu}]{$(\mathcal{P}_\beta)$}.1 encodes the fact that only two localized extremal high-probability sets $A_0^\varepsilon$ and $A_\infty^\varepsilon$ are necessary to assess the regularity of the score in~\eqref{eq:reguscoretime} and~\eqref{eq:reguscorespace}. 
    Here, the threshold $C_\star^{-1}\log(\varepsilon^{-1})^{-C_\star}$ corresponds to the time after which $s^\star(t,x)$ is essentially as smooth as $s^\star(\infty,x)$.

    \item[2.]
    Property~\hyperref[{assump:scoreregu}]{$(\mathcal{P}_\beta)$}.2 captures the fact that for fixed $t$, the high-probability sets $A_t^\varepsilon$ expand as $\varepsilon$ decreases.
    This typically will be fulfilled if $\vect{p}{0}$ and $\vect{p}{\infty}$ do not decay too fast at the boundary of their supports.
    Such a quantitative lower bound on the distance between $A_t^\varepsilon$ and the boundary of $A_t^{\varepsilon^2}$ ensures minimal growth of $(A_t^\varepsilon)_\varepsilon$. This is key for constructing globally one-sided Lipschitz neural networks (see Proposition~\ref{prop:aproxsuzuki}) and for probabilistic control of the trajectories $\vect{X}{t}$ staying within regularity zones of $s^\star(t,\cdot)$, crucial for deriving generalization bounds (see Proposition~\ref{prop:tradeoffbv}).

\end{enumerate}
For these reasons, modulo a time reparametrization of the forward process, we expect that Properties~\hyperref[{assump:scoreregu}]{$(\mathcal{P}_\beta)$} accurately capture the behavior of score functions for a broad class of generative diffusions initialized from $\beta$-smooth distributions. We therefore take them as interesting in their own right.

\section{A generic two-step analysis of generative diffusions}
\label{sec:two-step-analysis}

In this section, we will outline the key steps leading to guarantees for SGMs.
We intentionally avoid committing to a specific statistical model to maintain generality.
The analysis involves two complimentary parts: the stability of the backward equations (Section~\ref{sec:stability-of-sdes}), and the generalization accuracy of score matching for score estimation (Section~\ref{sec:oracle-for-score-matching}).

\subsection{Quantitative stability of (stochastic) differential equations}
\label{sec:stability-of-sdes}

Recalling the methodology of SGMs (Section~\ref{sec:OU-generation}), we notice that the (stochastic) differential equations governing dynamics of the true backward~\eqref{eq:backward-true} and the approximate backward~\eqref{eq:final-generator} differ only in their initial distribution and their drifts, but share the same diffusion coefficient. 
When measuring discrepancy in Wasserstein distance $\mathrm{W}_1$, this leads us to examine the marginal stability of stochastic differential equation solutions under variations in the initial distribution and drift term.
Unlike~\cite{oko2023diffusion,fukumizu2024flow,chen2023probability,gentiloni2025beyond}, which operate within an $L^2$ framework adapted to $\mathrm{W}_2$ and Kullback-Leibler, our results yield $L^1$-type stability, which arises more naturally when dealing with $\mathrm{W}_1$ distance.

In the first place, we bound the error arising from the mismatch between initial distributions $p_0 = \vect{p}{\overline{T}}$ for the true backward and $\widehat{p}_0 = \vect{p}{\infty}$ for the approximate backward.

\begin{lemma}[Stability with respect to the initial condition]
    \label{lem:SDEstability-varying-drift}
    Let $\underline{\tau} \leq \overline{\tau}$, and $p_{\underline{\tau}},\widetilde{p}_{\underline{\tau}}$ be two $L$-subGaussian probability densities on $\R^d$.
    Let $a : [\underline{\tau},\overline{\tau}] \times \R^d \to \R^d$ be a vector field such that for all $t$, $\|a_t\|_{L^\infty(\R^d)}\leq V$.
    Let $r : [\underline{\tau},\overline{\tau}] \to \R_{+}$ be a space-homogeneous diffusion coefficient such that $\|r\|_{L^2([\underline{\tau},\overline{\tau}])} \leq M$.
    Consider processes $(X_t)_{\underline{\tau} \leq t \leq \overline{\tau}}$ and $(\widetilde{X}_t)_{\underline{\tau} \leq t \leq \overline{\tau}}$ defined on $[\underline{\tau},\overline{\tau}]$, solutions to
    \begin{align*}
 \mathrm{d}X_t &= a_t(X_t) \mathrm{d}t + r_t \mathrm{d}B_t, \quad X_{\underline{\tau}} \sim p_{\underline{\tau}},
    \\
    d\widetilde{X}_t &= a_t(\widetilde{X}_t)\mathrm{d}t+r_t \mathrm{d}B_t, \quad \widetilde{X}_{\underline{\tau}} \sim \widetilde{p}_{\underline{\tau}}.
\end{align*}
Write $X_t \sim p_t$ and $\widetilde{X}_t \sim \widetilde{p}_t$. 
Then for all $0 <\varepsilon < 1$,
\begin{align*}
    \mathrm{W}_1(p_{\overline{\tau}},\widetilde{p}_{\overline{\tau}})
    &\leq
    \bigl(
    (\overline{\tau}-\underline{\tau})V+\sqrt{d}M
    +
    \sqrt{L\log(\varepsilon^{-1})}
    \bigr)
    \|p_{\underline{\tau}}-\widetilde{p}_{\underline{\tau}}\|_{L^1(\R^d)}
    +
    L\varepsilon/2
    .
\end{align*}
\end{lemma}
The proof of Lemma~\ref{lem:SDEstability-varying-drift} can be found in Section~\ref{sec:lem:SDEstability-varying-drift}.
In contrast to other stability bounds that require smoothness of the drift terms~\cite{backhoff2022adapted}, the assumptions on $a$ and $r$ here are particularly mild, amounting merely to boundedness of order zero.
On the other hand, the result relies on a subGaussian assumption for the initial distributions, which enables a quantitative link between the information-theoretic discrepancy $\|p_{\underline{\tau}} - \widetilde{p}_{\underline{\tau}}\|_{L^1}$ and the geometry-aware Wasserstein distance $\mathrm{W}_1$.

In the second place, we bound the drift error $a-\widehat{a} = \bigl( \sigma^2 + b^2 \bigr) (s^\star-\widehat{s})$ arising from the mismatch between the true score $s^\star$ and its estimated counterpart $\widehat{s}$ in~\eqref{eq:backward-true} and~\eqref{eq:final-generator} respectively.
We first deal with the most general case with no constraint on the diffusion coefficient.

\begin{theorem}[Stability with respect to the drift --- general case]
\label{thm:stabilitysde}
    Let $\underline{\tau} \leq \overline{\tau}$, and $p_{\underline{\tau}}$ be a probability density on $\R^d$.
    Let $a,\bar{a} : [\underline{\tau},\overline{\tau}] \times \R^d \to \R^d$ be vector fields. Assume that for all $t$, $\bar{a}_t$ is locally Lipschitz and $\lambda_{\max}(\nabla \bar{a}_t) \leq \bar{L}_t < \infty$.
    Let $r : [\underline{\tau},\overline{\tau}] \to \R_{+}$ be a space-homogeneous diffusion coefficient.
    Consider processes $(X_t)_{\underline{\tau} \leq t \leq \overline{\tau}}$ and $(\bar{X}_t)_{\underline{\tau} \leq t \leq \overline{\tau}}$ defined on $[\underline{\tau},\overline{\tau}]$, solutions to
\begin{align*}
 \mathrm{d}X_t &= a_t(X_t) \mathrm{d}t+r_t \mathrm{d}B_t, \quad X_{\underline{\tau}} \sim p_{\underline{\tau}},
    \\
    d\bar{X}_t &= \bar{a}_t(\bar{X}_t) \mathrm{d}t+r_t \mathrm{d}B_t, \quad \bar{X}_{\underline{\tau}} \sim p_{\underline{\tau}}.
\end{align*}
Write $X_t \sim p_t$ and $\bar{X}_t \sim \bar{p}_t$. Then,
\begin{align*}
    \mathrm{W}_1(p_{\overline{\tau}},\bar{p}_{\overline{\tau}})
    &\leq
    \int_{\underline{\tau}}^{\overline{\tau}}
    \int_{\R^d}
    e^{\int_t^{\overline{\tau}} \bar{L}_u \mathrm{d} u}
    \|a_t(x) - \bar{a}_t(x)\|
p_t(x)\mathrm{d}x \mathrm{d} t.
\end{align*}
\end{theorem}

The proof of Theorem~\ref{thm:stabilitysde} can be found in Section~\ref{sec:thm:stabilitysde}.
Note the asymmetry between $a,p$ and $\bar{a},\bar{p}$ in the assumptions and statement of the result.
Namely, the bound involves error integration against $p_t$ (interpreted as the true marginal), with an additional exponential weight depending on the largest eigenvalue of the Jacobian of $\bar{a}_t$ (interpreted as an approximate drift).
Although this integrated exponential weighting term might seem harmful, recall that Proposition~\ref{prop:lipregu} ensures that the true score function $s^\star$ satisfies
$$
\int_{0}^{\infty}
\sup_{x \in \R^d} 
\lambda_{\max} (\nabla s^\star(u,x)+\frac{1}{\sigma^2}I_{d \times d})\mathrm{d}u\leq C < \infty 
.
$$
Therefore, one may reasonably expect to construct an approximation of $s^\star$ that preserves this quantitative one-sided Lipschitz property without introducing additional statistical limitations.

\subsection{Quantitative oracle bounds for score matching}
\label{sec:oracle-for-score-matching}

To complement the analytical stability results of the previous section, we now turn to the statistical generalization properties of the denoising score matching estimator~\eqref{eq:scoreestimatorheuristic}. As is common in non-asymptotic analyses of empirical risk minimization, the result takes the form of a general oracle inequality in expectation.

With transparent notation, we denote by $\restr{\mathcal{S}}{[\underline{\tau},\overline{\tau}]}$ the set of time-restrictions $s: [\underline{\tau},\overline{\tau}] \times \R^d \to \R^d$ of functions $s \in \mathcal{S}$ possibly defined over a larger domain.
The quantity $\mathcal{N}(\mathfrak{S},\|\cdot \|,\delta)$ denotes the covering number at scale $\delta>0$ of the set $\mathfrak{S}$ with respect to the norm $\|\cdot\|$.

\begin{proposition}[Oracle inequality for denoising score matching]
\label{prop:tradeoffbv}
Assume that $p^\star$ is $K$-subGaussian and that $s^\star$ satisfies Properties~\hyperref[{assump:scoreregu}]{$(\mathcal{P}_\beta)$}.
Let $X^{(1)}, \ldots, X^{(n)}$ be a i.i.d. sample from $p^\star(x) \mathrm{d} x$, and assume that $\sigma >0$. 
Let $\widehat{s}$ be an empirical risk minimizer of the denoising score matching loss~\eqref{eq:scoreestimatorheuristic}
\begin{align*}
    \widehat{s}\in \argmin_{s\in \mathcal{S}} \frac{1}{n} \sum_{i=1}^n
    \gamma_{[\underline{\tau},\overline{\tau}]}(s,X^{(i)}),
\end{align*}
for $n^{-1} \leq \underline{\tau} < \overline{\tau}$, over a class $\mathcal{S}$ of functions $s:\R_+ \times \R^d \to \R^d$ such that $\sup_{t\in [\underline{\tau},\overline{\tau}]} \|s(t,\cdot)\|_\infty \leq V$.
Then,
\begin{align*}
    \mathbb{E}\Bigg[
     \int_{\underline{\tau}}^{\overline{\tau}}\int_{\R^d} 
     \| \widehat{s}(t,x) - s^\star(t,x) \|^2 
     \vect{p}{t}(x)  \mathrm{d}x \mathrm{d}t
    \Bigg]
    &\leq 
    3 \inf_{s \in \mathcal{S}} \int_{\underline{\tau}}^{\overline{\tau}} \int_{\R^d} \|s(t,x) - s^\star(t,x)\|^2 \vect{p}{t} (x)  \mathrm{d}x \mathrm{d}t  \\
    &
    \hspace{1em}
    + 
    \frac{c_n}{n} \log\Big( \mathcal{N}\big(\restr{\mathcal{S}}{[\underline{\tau},\overline{\tau}]}, \|\cdot\|_{\infty}, n^{-1}\bigr)\vee n \Bigr) ,
\end{align*}
where $c_n := C (\overline{\tau}-\underline{\tau})\log(n)^C g(\underline{\tau},\sigma,V)$ and  $g(\underline{\tau},\sigma,V) :=  \left ( 1 + V^2 + (1+\sigma+\sigma^{-1})^4 (1+\underline{\tau}^{-1}) + \frac{1}{\sigma^2(1-e^{-2 \underline{\tau}})} \right )$.
\end{proposition}

The proof of Proposition~\ref{prop:tradeoffbv} can be found in Section~\ref{sec:prop:tradeoffbv}, where is also established a high probability deviation bound. 
This result is similar to~\cite[Theorem~4.3]{oko2023diffusion},
while circumventing the critical issue pointed out
by~\cite{yakovlev2025generalization}.
Its proof builds upon standard Bernstein-type concentration arguments.
Since the contrast function $\gamma(s,x)$ is unbounded, we use a clipping strategy. 
The subGaussianity of the underlying distribution $p^\star$ enables us to control the effect of this clipping, here limited to $\mathrm{polylog}(n)$ factors only.

Recall that the Proposition~\ref{prop:tradeoffbv} applies to an Orstein-Uhlenbeck forward process~\eqref{eq:forward-true}.
Indeed, this forward process allows for the denoising score trick (Theorem~\ref{thm:Vincent}), which grounds the empirical risk minimization strategy.
Hence, this result can then be adapted to other choices of forward processes, provided that a similar score trick holds. 

\section{From score regularity to guarantees for score matching
}
\label{sec:score-regularity-implies-score-estimation}

With all the analytical and statistical groundwork now in place, this section establishes that the subGaussianity of $p^\star$ combined with the smoothness Properties~\hyperref[{assump:scoreregu}]{$(\mathcal{P}_\beta)$} of the score function suffice for denoising score matching to be provably successful over a suitably chosen class of neural networks. 
We defer the precise construction of the neural networks to Section~\ref{sec:scoreapp}.

The following lemma shows that if the score satisfies Properties~\hyperref[{assump:scoreregu}]{$(\mathcal{P}_\beta)$}, then score-based generative models can achieve minimax convergence rates up to logarithmic factors.

\begin{lemma}[Score smoothness implies estimation rates for score matching]
\label{lemma:minimaxforsmoothscore}
Let $X^{(1)}, \ldots, X^{(n)}$ be a i.i.d. sample with density $p^\star:\mathbb{R}^d\rightarrow \mathbb{R}_+$.
Write $\widehat{p}$ for the distribution of $\widehat{X}_{\overline{T}-\underline{T}}$, obtained from the backward dynamic~\eqref{eq:final-generator} with $\|b\|_{L^\infty([0,\overline{T}-\underline{T}])} \leq C$, $C^{-1} \leq \sigma \leq C$, and with a score estimated via denoising score matching~\eqref{eq:scoreestimatorheuristic} over class $\mathcal{S}$ on time interval $[\underline{T},\overline{T}]$.

Assume that $p^\star(x) \mathrm{d}x$ is $K$-subGaussian~\eqref{eq:subgaussian}, and that the score function $s^\star$~\eqref{eq:scfunc} associated to $p^\star$ satisfies 
Properties~\hyperref[{assump:scoreregu}]{$(\mathcal{P}_\beta)$} with $\beta >0$.
Then there exists 
$(C_2n)^{-\frac{2(\beta+1)}{2\beta+d}} \leq  \underline{T} \leq \overline{T} \leq C_2 \log n$, and a class of neural networks $\mathcal{S}$ with $\tanh$ activation function (see Section~\ref{sec:scoreapp}), such that
    \begin{align*}
     \mathbb{E}[\mathrm{W}_1(p^\star,\widehat{p})] \leq C_3  \log(n)^{C_3} n^{-\frac{\beta+1}{2\beta+d}}
     .
\end{align*}
\end{lemma}
We showcase the proof of Lemma~\ref{lemma:minimaxforsmoothscore} in Section~\ref{sec:lemma:minimaxforsmoothscore} to highlight its simplicity.
This result illustrates how score regularity can yield minimax convergence rates of SGMs, no matter the value of the backward diffusion coefficient $(b_t)_t$.
In short, the one-sided Lipschitz regularity~\eqref{eq:onesidedlipscoreassum} allows to bound the Wasserstein distance by a \emph{boundedly}-weighted \( L^1 \) distance between the true score and its estimator~(Theorem~\ref{thm:stabilitysde}).
On the other hand, the high order smoothness of the score~\eqref{eq:reguscoretime} and~\eqref{eq:reguscorespace} allows to use classical neural networks approximation theory to bound this latter distance through a bias-variance trade-off (Proposition~\ref{prop:tradeoffbv}).

Let us recall that from Lemma~\ref{lem:AbetaimpliesHbeta}, the score $s^\star(t,x)$ of a distribution $p^\star \in \mathcal{A}_K^\beta$ satisfies Properties~\hyperref[{assump:scoreregu}]{$(\mathcal{P}_\beta)$}.  Lemma~\ref{lemma:minimaxforsmoothscore} hence directly directly yields quantitative generalization bounds over $\mathcal{A}_K^\beta$ (see Theorem~\ref{thm:cv_rates}).

\subsection{Time-dependent neural network structure}
\label{sec:scoreapp}
To prove Lemma~\ref{lemma:minimaxforsmoothscore}, let us first introduce the class $\mathcal{S}$ of neural networks used for constructing an effective score estimator.
The class of neural networks is designed to be spatially one-sided Lipschitz to ensure the finiteness of the exponential term appearing in Theorem~\ref{thm:stabilitysde}.
It relies on the following building block.

\begin{definition}[Time-dependent neural network class]
\label{defi:nn}
We denote by $\Psi(L, W,B,V,V^{'})_d^{k+1}$ the class of $\tanh$ neural networks $\phi:\mathbb{R}^{k+1}\rightarrow \mathbb{R}^d$ with:
\begin{itemize}[leftmargin=*]
\item at most $L$ hidden layers, each of maximum width $W$, and having weights bounded by $B$.
    \item $\|\phi\|_\infty\leq V$ .
    \item $\lambda_{\max}(\nabla \phi(t,x)) \leq V^{'}$ for all $x\in \mathbb{R}^{k}$ and $t\in \mathbb{R}$.
\end{itemize}
\end{definition}

\paragraph*{Time discretization.}
The construction builds upon a partition of the time interval \([\underline{T}, \overline{T}]\) into $m$ coarse sub-intervals \([\tau_k, \tau_{k+1}]\), which are themselves further subdivided into $\Upsilon_k$ finer sub-intervals~\([\tau_{k,j}, \tau_{k,j+1}]\), yielding
\begin{align*}
[\underline{T}, \overline{T}]
=
\bigcup_{k=0}^{m-1}
[\tau_k, \tau_{k+1}]
,
\text{~~and~~}
[\tau_k, \tau_{k+1}]
=
\bigcup_{j=0}^{\Upsilon_k-1}
[\tau_{k,j}, \tau_{k,j+1}]
,
\end{align*}
with $\tau_k = \tau_{k,0}$ and $\tau_{k,\Upsilon_k} = \tau_{k+1}$.
For each finer interval \([\tau_{k,j}, \tau_{k,j+1}]\), we employ a separate neural network of the form of Definition~\ref{defi:nn} (see~\eqref{eq:classofnn} below).

We set $\underline{T} := (Cn)^{-\frac{2(\beta+1)}{2\beta+d}}$, $\overline{T} := C \log n$, and $m := \log_2(\overline{T}/\underline{T}) \in \mathbb{N}$ without loss of generality. Note by now that $m \leq C \log n$.
We consider the subdivision $(\tau_k)_{0 \leq k \leq m}$ of $[\underline{T},\overline{T}]$ with extremal points $\underline{T} =: \tau_0 < \tau_m := \overline{T}$, and geometric progression $\tau_{k+1}/\tau_k = 2$, meaning that $\tau_k := 2^{k} \underline{T} = 2^{k-m} \overline{T}$.

For all $k \in \{0,\ldots,m-1\}$, we consider the subdivision $(\tau_{k,j})_{0 \leq j \leq \Upsilon_k}$ of $[\tau_k,\tau_{k+1}]$ with extremal points $\tau_{k} =: \tau_{k,0} < \tau_{k,\Upsilon_k} := \tau_{k+1}$, and arithmetic progression $\tau_{k,j+1}-\tau_{k,j} = \tau_k/\Upsilon_k$, meaning that
\begin{equation}\label{eq:subdiscretizationtime}
\tau_{k,j} := \left(1 + j/\Upsilon_k \right) \tau_k
\text{,~~where~~}
\Upsilon_k=\lceil (\tau_kn^{\frac{2}{2\beta+d}})^{\frac{d-2}{d}}\rceil
.
\end{equation} 

\paragraph*{Global architecture.}
Using notation of Definition~\ref{defi:nn}, the chosen class of candidate scores is 
\begin{equation}\label{eq:classofnn}
\mathcal{S}
:=
\left\{-
\frac{1}{\sigma^2} I_{d\times d} +\sum_{k=0}^{m}\sum_{j=0}^{\Upsilon_k-1} S_{k,j}(t,\cdot)\mathds{1}_{t\in[\tau_{k,j},\tau_{k,j+1})}\Big|\ S_{k,j}\in \Psi(L_{k},\mathrm{W}_{k},B_k,V_{k},V_{k}^{'})^{d+1}_d\right
\},
\end{equation}
with time-dependent fully connected architectures of sizes
\begin{equation}\label{eq:networkshape}L_k=C, \quad  W_k
=
C\log(n)^{C_2} \tau_k^{-1}n^{\frac{d-2}{2\beta+d}}, \quad B_k=C n^{C_2}
,
\end{equation}
and functional constraints
\begin{equation}\label{eq:networkshape2} V_k= C \log(n)^{C_2}(\tau_k\wedge 1)^{-1/2}, 
\quad  
V_k^{'}=Ce^{-\tau_k}(\tau_k\wedge 1)^{-1+\frac{\beta\wedge 1}{d}}.
\end{equation}
The sup-norm scaling $V_k$ is chosen to cover that of the true score over $[\tau_k, \tau_{k+1}]$ (see Theorem~\ref{coro:regularityp}), with the one-sided Lipschitz scaling $V_k'$ handled similarly (see Proposition~\ref{prop:lipregu}). 
We incorporate the term \(-\frac{1}{\sigma^2} I_{d \times d}\) into the network class to incorporate the known limit of the score as \(t \to \infty\). 
This design allows the network to naturally follow the exponential decay of the true score, leading to the guarantee
$
\lambda_{\max}(
\frac{1}{\sigma^2} I_{d \times d}
+
\nabla \widehat{s}(t,x) ) \lesssim e^{-t}$ when $t$ is large.

\subsection{Leveraging score regularity for neural networks approximation}\label{sec:nerualnetowrk}

The smoothness of the score function enables efficient approximation by neural networks. In particular, we make use of the following classical result from neural network approximation theory.
\begin{proposition}[{Fixed-time approximation of $\gamma$-Hölder functions ---~\cite[Theorem~5.1]{de2021approximation}}]
\label{prop:approxrelu}
For all $\varepsilon \in (0,1)$ and $\gamma,A,R>0$, there exists $C_\gamma,C_{\gamma,2}>0$ such that the neural networks class $$\mathbb{K}:=\Psi\Big(2,C_\gamma (A\varepsilon^{-1})^d,C_\gamma R\varepsilon^{-C_\gamma},C_\gamma R,\infty\Big)_d^{d}$$ satisfies for all $\theta\in (0,1)$
$$\log \bigl(\mathcal{N}(\mathbb{K},\|\cdot\|_\infty,\theta)\bigr)\leq C_{\gamma,2} (A\varepsilon^{-1})^d \log\bigl(R(\varepsilon\theta)^{-1}\bigr),$$

and for all $f\in \mathcal{H}^\gamma_R(\B(0,A),\mathbb{R}^d)$, there exists a neural network $N_f\in \mathbb{K}$ satisfying for all $k\in \{0,\ldots,\lfloor \gamma \rfloor\}$
\begin{align*}
    \sup_{x\in \B(0,A)} \|\nabla^kf(x)-\nabla^k N_f(x)\|\leq C R \varepsilon^{\gamma-k}.
\end{align*}
\end{proposition}

From this result applied to fixed-time scores functions $s^\star(\tau_k,\cdot)$, and local Taylor-like neural network constructions in the time variable, we deduce the following approximation result.

\begin{proposition}[Time-dependent score estimation]
\label{prop:aproxsuzuki}
Let $p^\star$ a probability measure on $\mathbb{R}^d$ such that its score function $s^\star$ satisfies Properties~\hyperref[{assump:scoreregu}]{$(\mathcal{P}_\beta)$}.
Then for all $k\in \{0,1,\ldots,m-1\}$ and $j\in\{0,\ldots,\Upsilon_k-1\}$, there exists $\bar{s}\in \Psi(L_k, W_k,B_k, V_k,V_k^{'})^{d+1}_d$  such that
       $$\sup_{t\in [\tau_{k,j},\tau_{k,j+1}]}\sup_{x\in A_{\tau_k}^{1/n}}\|s^\star(t,x)-\bigl(\bar{s}(t,x)-\frac{1}{\sigma^2}x\bigr)\|\leq C \log(n)^{C_2} \frac{1}{\tau_k}n^{-\frac{\beta+1}{2\beta+d}},$$ 
where $\tau_{k,j}$ defined in~\eqref{eq:subdiscretizationtime}
 and $\Psi(L_k, W_k,B_k, V_k,V_k^{'})^{d+1}_d$ defined in Definition~\ref{defi:nn}. Furthermore, we have
$$\log \bigl(\mathcal{N}(\Psi(L_k, W_k,B_k, V_k,V_k^{'})^{d+1}_d,\|\cdot\|_\infty,n^{-1})\bigr)\leq C \log(n)^{C_2}
\frac{ n^{-\frac{2}{2\beta+d}}}{\tau_k}n^{\frac{d}{2\beta+d}}.
$$
\end{proposition}

The proof of Proposition~\ref{prop:aproxsuzuki} can be found in Section~\ref{sec:coro:prop:aproxsuzuki}. This result illustrates how the adaptive regularity of the score function translates directly into efficient neural network approximations. In particular, it ensures that one can construct neural networks that are both accurate and one-sided Lipschitz, thereby controlling the exponential stability term appearing in Theorem~\ref{thm:stabilitysde}.

\subsection{Score matching rates from score regularity}
\label{sec:lemma:minimaxforsmoothscore}
Having defined the neural network class $\mathcal{S}$ for the score estimator in~\eqref{eq:classofnn}, we can now proceed to the proof of Lemma~\ref{lemma:minimaxforsmoothscore}. 
The proof follows three main steps:
\begin{enumerate}

    \item Using the stability results of Section~\ref{sec:two-step-analysis}, we show that the the problem reduces to controlling the $L^1$-distance between the true score and its estimator.
    \item Applying the oracle bound of Section~\ref{sec:oracle-for-score-matching}, we then bound this $L^1$-distance in terms of a trade-off between the best approximation error achievable by the class \( \mathcal{S} \) and its complexity on each sub-interval \( [\tau_{k,j}, \tau_{k,j+1}] \).
    \item Finally, using the universal approximation theory of neural networks from Section~\ref{sec:nerualnetowrk}, we show that this trade-off leads to the desired minimax convergence rate.
\end{enumerate}

\begin{proof}[Proof of Lemma~\ref{lemma:minimaxforsmoothscore}] 
As pointed out in Section~\ref{sec:OU-generation}, the distribution $\widehat{p} := \widehat{p}_{\overline{T} - \underline{T}} \sim \widehat{X}_{\overline{T} - \underline{T}}$ accumulates three distinct errors
\begin{itemize}[leftmargin=*]
    \item \textbf{Initial condition.} 
    The true backward process is initialized at $p_0 = \vect{p}{\overline{T}}$, while the approximated backward is initiated at $\vect{p}{\infty} = \mathcal{N}(0,\sigma^2 I_{d \times d})$.
    \item \textbf{Score error.} 
    The true backward process uses the true score $s^\star(t,\cdot)$, while the approximated backward process uses the estimator $\widehat{s}(t,\cdot)$.
    \item \textbf{Early stopping.}
    The true backward process should be run for time $\overline{T}$ to recover distribution $p_{\overline{T}} = \vect{p}{0} = p^\star$, while the approximate backward is stopped earlier, at time $\overline{T}-\underline{T}$.
\end{itemize}

\textbf{Early stopping.}
Let us first bound the error arising from the early stopping.
We have
\begin{align}\label{align:firstboundw1}
    \mathrm{W}_1(\widehat{p}, p^\star) & \leq \mathrm{W}_1(\widehat{p}, \vect{p}{\underline{T}})+ \mathrm{W}_1(\vect{p}{\underline{T}}, p^\star)
    .
\end{align}
Considering the independent coupling $(\vect{X}{0}, Z)\sim p^\star \otimes \mathcal{N}(0,I_{d \times d})$, we get from~\eqref{align:lawofxt} that
\begin{align*}
\mathrm{W}_1(\vect{p}{\underline{T}}, p^\star)&\leq \mathbb{E}[\|\vect{X}{0}-(e^{-\underline{T}}\vect{X}{0}+\sigma_{\underline{T}}Z)\|]\\
&\leq (1-e^{-\underline{T}})\mathbb{E}[\|\vect{X}{0}\|]+\sigma_{\underline{T}}\mathbb{E}[\|Z\|]\\
&\leq C\underline{T}^{1/2}
\\
&=
Cn^{-\frac{\beta+1}{2\beta+d}},
\end{align*}
which concludes the bound on the the early stopping error. Now, to bound the first term of~\eqref{align:firstboundw1}, let us recall that $p_{\overline{T}-t}=\vect{p}{t}$ is the distribution of the process $X_t$ solution to the equation
\begin{align*}
    \begin{cases}
    \mathrm{d} X_t
=
a_t(X_t) \mathrm{d} t 
+
\sqrt{2} b_t \mathrm{d} B_t, \quad X_0 \sim \vect{p}{\overline{T}}\\
a_t(x) :=  x + \bigl( \sigma^2  + b_t^2 \bigr) s^\star(\overline{T}-t,x)
    \end{cases}
\end{align*}
 and recall that $\widehat{p}_t(x) \mathrm{d}x$ is the distribution of the process $\widehat{X}_t$ solution at time $t$ of the differential equation
\begin{align*}
    \begin{cases}
\mathrm{d} \widehat{X}_t
=
\widehat{a}_t(\widehat{X}_t) \mathrm{d} t 
+
\sqrt{2} b_t \mathrm{d} B_t, \quad \widehat{X}_0 \sim \vect{p}{\infty}\\
\widehat{a}_t(x) := x + \bigl( \sigma^2  + b_t^2 \bigr) \widehat{s}(\overline{T}-t,x).
    \end{cases}
\end{align*}
To bound $\mathrm{W}_1(\widehat{p}_{\overline{T}-\underline{T}}, \vect{p}{\underline{T}})=\mathrm{W}_1(\widehat{p}_{\overline{T}-\underline{T}}, p_{\overline{T}-\underline{T}})$, we use the stability results from Section~\ref{sec:stability-of-sdes}. 

\textbf{Initial condition.}
We first bound the error arising from the initial condition. 
To this aim, let $\widetilde{\mu}_{t}$ denote the distribution of the solution $\widetilde{X}_t$ to the equation
\begin{align*}
    \mathrm{d} \widetilde{X}_t
=
- \widehat{a}_t(\widetilde{X}_t) \mathrm{d} t 
+
\sqrt{2} b_t \mathrm{d} B_t, \quad \widetilde{X}_0 \sim \vect{p}{\overline{T}}.
\end{align*}
over $[0,\overline{T}-\underline{T}]$.
Then, triangle inequality yields
\begin{align}\label{align:firstboundwbis}
    \mathrm{W}_1(\widehat{p}_{\overline{T}-\underline{T}}, p_{\overline{T}-\underline{T}}) & \leq \mathrm{W}_1(\widehat{p}_{\overline{T}-\underline{T}}, \widetilde{\mu}_{\overline{T}-\underline{T}})+ \mathrm{W}_1(\widetilde{\mu}_{\overline{T}-\underline{T}},p_{\overline{T}-\underline{T}}),
\end{align}
By definition of $\mathcal{S}$ (see~\eqref{eq:networkshape} and~\eqref{eq:networkshape2}), we have $\|\widehat{s}\|_\infty\leq C\log(n)^{C_2} n^{\frac{\beta+1}{2\beta+d}},$ so using the stability of SDEs with respect to the initial condition (Lemma~\ref{lem:SDEstability-varying-drift}) with $\varepsilon=n^{-(K\vee \frac{1}{\sigma^2})}$ we have
\begin{align*}
    \mathrm{W}_1(\widehat{p}_{\overline{T}-\underline{T}}, \widetilde{\mu}_{\overline{T}-\underline{T}}) & \leq C\log(n)^{C_2} n^{\frac{\beta+1}{2\beta+d}}\left(\|\vect{p}{\infty}-\vect{p}{\overline{T}}\|_{L^1}+n^{-1}\right).
\end{align*}
 Then, as $\overline{T}\geq C \log( n$),~\cite[Theorem~5.2.1]{bakry2013analysis} ensures that
\begin{align*}
    \|\vect{p}{\infty}-\vect{p}{\overline{T}}\|_{L^1} \leq \sqrt{2 \text{KL}(\vect{p}{\overline{T}},\vect{p}{\infty})}\leq 2e^{-\overline{T}/\sigma^2}\sqrt{\text{KL}(p^\star,\vect{p}{\infty})}\leq Ce^{-\overline{T}/\sigma^2} \leq Cn^{-1}.
\end{align*}

\textbf{Score error.}
Let us now bound the second term of~\eqref{align:firstboundwbis} which corresponds to the score approximation error. As the score estimation is done independently on each intervals $[\tau_{k,j},\tau_{k,j+1}]$, we first decompose the distance
$\mathrm{W}_1(\widetilde{\mu}_{\overline{T}-\underline{T}},p_{\overline{T}-\underline{T}})$ into contributions that involve only the dynamics within each time sub-interval $[\tau_{k,j},\tau_{k,j+1}]$. 
More precisely, let $p_t^{k,j}$ denote the distribution of the solution $X^{k,j}_t$ to the equation
$$\mathrm{d} X^{k,j}_{t}
=
- \widehat{a}_t^{k,j}(X^{k,j}_{t}) \mathrm{d} t 
+
\sqrt{2} b_t \mathrm{d} B_t, \quad X^{k,j}_{0}\sim \vect{p}{\overline{T}},$$
over $[0,\overline{T}-\underline{T}]$,
with
\begin{align*}
\widehat{a}_t^{k,j}
:= \left\{\begin{array}{ll} a_t & \text{if } t\leq \overline{T}-\tau_{k,j},
\\ \widehat{a}_t & \text{if } t > \overline{T}-\tau_{k,j}.
    \end{array}\right.
\end{align*}
From triangle inequality again, 
\begin{align*}
   \mathrm{W}_1(\widetilde{\mu}_{\overline{T}-\underline{T}},p_{\overline{T}-\underline{T}})\leq \sum_{k=0}^{m-1}\sum_{j=0}^{\Upsilon_k-1} \mathrm{W}_1(p_{\overline{T}-\underline{T}}^{k,j+1}, p_{\overline{T}-\underline{T}}^{k,j}).
\end{align*}
Now, using the stability of SDEs with respect to the drift (Theorem~\ref{thm:stabilitysde}) on each pair of stochastic processes $\bigl((X^{k,j+1}_t)_t,(X^{k,j}_t)_t\bigr)$, we obtain the bound
\begin{align*}
\mathrm{W}_1(\widetilde{\mu}_{\overline{T}-\underline{T}},p_{\overline{T}-\underline{T}}) 
&\leq
\sum_{k=0}^{m-1}\sum_{j=0}^{\Upsilon_k-1} \int_{\overline{T}-\tau_{k,j+1}}^{\overline{T}-\tau_{k,j}}e^{\int_t^{\overline{T}-\underline{T}} \|\lambda_{\max}(\nabla \widehat{a}_u))\|_\infty \mathrm{d}u} \int  \|\widehat{s}(\overline{T}-t,x) - s^\star(\overline{T}-t,x)\|p_{t}(x)\mathrm{d}x \mathrm{d}t
\\
&
=
\sum_{k=0}^{m-1}\sum_{j=0}^{\Upsilon_k-1} \int_{\tau_{k,j}}^{\tau_{k,j+1}}e^{\int_{\overline{T}-t}^{\overline{T}-\underline{T}} \|\lambda_{\max}(\nabla \widehat{a}_{u}))\|_\infty \mathrm{d}u} \int  \|\widehat{s}(t,x) - s^\star(t,x)\|\vect{p}{t}(x)\mathrm{d}x \mathrm{d}t.
\end{align*}
As the class of estimators $\mathcal{S}$ is composed of neural networks $\phi$ satisfying 
$$\lambda_{\max}\left(\frac{1}{\sigma^2}I_{d \times d}+\nabla \phi(t,\cdot)\right)\leq Ce^{-t}(1+t^{-1+\frac{\beta\wedge 1}{d}}),$$
we deduce that 
$$\int_{0}^{\overline{T}-\underline{T}} \|\lambda_{\max}(\nabla \widehat{a}_u)\|_\infty \mathrm{d}u\leq C,
$$ 
and hence
\begin{align*}
\mathrm{W}_1(\widetilde{\mu}_{\overline{T}-\underline{T}},p_{\overline{T}-\underline{T}}) 
&\leq  
C
\sum_{k=0}^{m-1}\sum_{j=0}^{\Upsilon_k-1} \int_{\tau_{k,j}}^{\tau_{k,j+1}}\int  \|\widehat{s}(t,x) - s^\star(t,x)\|\vect{p}{t}(x)\mathrm{d}x \mathrm{d}t
.
\end{align*}
Then, applying Cauchy-Schwarz inequality we get
\begin{align}\label{align:finalboundw1}
\mathrm{W}_1(\widetilde{\mu}_{\overline{T}-\underline{T}},p_{\overline{T}-\underline{T}}) & \leq C  \sum_{k=0}^{m-1}\sum_{j=0}^{\Upsilon_k-1} \sqrt{\tau_{k,j+1}-\tau_{k,j}}\left(\int_{\tau_{k,j}}^{\tau_{k,j+1}} \int  \|\widehat{s}(t,x) - s^\star(t,x)\|^2\vect{p}{t}(x)\mathrm{d}x \mathrm{d}t \right)^{1/2}.
\end{align}
We now use the bias-variance bound from Proposition~\ref{prop:tradeoffbv} and the neural network approximation of Proposition~\ref{prop:aproxsuzuki} to control the score approximation. Recall that in Proposition~\ref{prop:aproxsuzuki}, the approximation properties hold uniformly over $t \in [\tau_{k,j}, \tau_{k,j+1}]$ and $x \in A_{\tau_k}$, where $\int_{A_{\tau_k}} \vect{p}{t}(x)\mathrm{d}x \geq 1-1/n$.

For all $s\in \mathcal{S}$ and $t\in [\tau_k,\tau_{k+1}]$, we have by definition of the class $\mathcal{S}$ that $\|s(t,x)\|\leq C\log(n)^{C_2}(1+ \tau_k^{-1/2})$ and from~\eqref{eq:estimatenormscore} we have $\|s^\star(t,x)\|\leq C(1+t^{-1})(1+\|x\|)$.  Then,  
\begin{align}\label{eq:control_L2norm_subinterval_outside}
    \int_{(A_{\tau_k})^c} \|s(t, x) - s^\star(t,x)\|^2 \vect{p}{t}(x)\mathrm{d}x \notag 
    & \leq  C\log(n)^{C_2} \int_{(A_{\tau_k})^c} (1+\|x\|^2)(1+\tau_k^{-2}) \vect{p}{t}(x)\mathrm{d}x \notag\\
   & \leq  C\log(n)^{C_2}\biggl( \int_{(A_{\tau_k})^c\cap \B(0,C\log(n))} (1+\tau_k^{-2}) \vect{p}{t}(x)\mathrm{d}x\notag\\
   &\quad + \int_{\B(0,C\log(n))^c} (1+\|x\|^2)(1+\tau_k^{-2}) \vect{p}{t}(x)\mathrm{d}x\biggl) \notag \\
   &\leq C(1+\tau_k^{-2})\log(n)^{C_2}n^{-1}\notag,
\end{align}
where we used that $\vect{p}{t}$ is $C(K^2+\sigma^2)$-subGaussian for the last line. For the score approximation, using Proposition~\ref{prop:tradeoffbv}  we may thus write
\begin{align}\label{align:findeborneoracle}
&\mathbb{E}_{\{X^{(i)}\}_{i=1}^n} \Bigg[
     \int_{\tau_{k,j}}^{\tau_{k,j+1}}\int \| \widehat{s}(t, x) - s^\star(t,x) \|^2 \vect{p}{t}(x)\mathrm{d}x \mathrm{d}t
    \Bigg]\nonumber\\
    & \leq C \inf_{s \in \mathcal{S}} \int_{\tau_{k,j}}^{\tau_{k,j+1}}\Big(\|s(t, \cdot) - s^\star(t,\cdot)\|^2_{\mathcal{H}^0(A_{\tau_k})} + \int_{(A_{\tau_k})^c} \|s(t, x) - s^\star(t,x)\|^2 \vect{p}{t}(x)\mathrm{d}x\Big) \mathrm{d}t\nonumber\\
    & \quad + \frac{C}{n}\log(n)^{C_2}(\tau_{k,j+1}-\tau_{k,j})\frac{1}{\tau_k}
 \log \mathcal{N}\bigl(\restr{\mathcal{S}}{[\tau_{k,j},\tau_{k,j+1}]}, \|\cdot\|_{\infty}, n^{-1}\bigr)\nonumber \\
  & \leq C\log(n)^{C_2}(\tau_{k,j+1}-\tau_{k,j})\left(\inf_{s \in \mathcal{S}}\|s- s^\star\|^2_{\mathcal{H}^0([\tau_{k,j},\tau_{k,j+1}]\times A_{\tau_k})}+(1+\tau_k^{-2})n^{-1}\right)\nonumber\\
 & \quad +\frac{C}{n}\log(n)^{C_2}(\tau_{k,j+1}-\tau_{k,j})\frac{1}{\tau_k}
 \log \mathcal{N}\bigl(\restr{\mathcal{S}}{[\tau_{k,j},\tau_{k,j+1}]}, \|\cdot\|_{\infty}, n^{-1}\bigr).
\end{align}
Let us apply Proposition~\ref{prop:aproxsuzuki} to obtain the bound on the covering number and approximation error of the class $\mathcal{S}$.
We first treat the case $\tau_k< n^{-\frac{2}{2\beta+d}}$. In this case $\Upsilon_k=1$, so that we have to bound
\begin{align*}
    &\mathbb{E}_{\{X^{(i)}\}_{i=1}^n} \Bigg[(\tau_{k+1} - \tau_{k})
     \int_{\tau_{k}}^{\tau_{k+1}}\int \| \widehat{s}(t, x) - s^\star(t,x) \|^2 \vect{p}t(x)  \mathrm{d}x \mathrm{d}t
    \Bigg] \\ 
    &\quad = \mathbb{E}_{\{X^{(i)}\}_{i=1}^n} \Bigg[\tau_k
     \int_{\tau_{k}}^{\tau_{k+1}}\int \| \widehat{s}(t, x) - s^\star(t,x) \|^2 \vect{p}t(x) \mathrm{d}x \mathrm{d}t
    \Bigg], 
\end{align*}
since $\tau_{k+1} = 2 \tau_k$.
Combining~\eqref{align:findeborneoracle} and the approximation results of Proposition~\ref{prop:aproxsuzuki} leads to
\begin{align*}
&\mathbb{E}_{\{X^{(i)}\}_{i=1}^n} \Bigg[\tau_k
     \int_{\tau_{k}}^{\tau_{k+1}}\int \| \widehat{s}(t, x) - s^\star(t,x) \|^2 \vect{p}t(x)  \mathrm{d}x \mathrm{d}t
    \Bigg]\\
& \leq C  \tau_k\left(\tau_k \left(\frac{1}{\tau_k}\log(n)^{C_2}n^{-\frac{\beta+1}{2\beta+d}}\right)^2+\frac{1}{n}\log(n)^{C_2}\left(\frac{\tau_k}{\tau_k}\frac{1}{\tau_k}n^{\frac{d-2}{2\beta+d}}\right)+\tau_k(1+\tau_k^{-2})\log (n)^{C_2}n^{-1}\right) \\
& \leq C \log (n)^{C_2} n^{-\frac{2(\beta+1)}{2\beta+d}}.
\end{align*}
Now assume that $\tau_k> n^{-\frac{2}{2\beta+d}}$. In this case, we have $\tau_{k,j+1}-\tau_{k,j}=\tau_k(\tau_kn^{\frac{2}{2\beta+d}})^{-\frac{d-2}{d}}$ so using~\eqref{align:findeborneoracle} and Proposition~\ref{prop:aproxsuzuki} again yields
\begin{align*}
    &\mathbb{E}_{\{X^{(i)}\}_{i=1}^n} \Bigg[\tau_k(\tau_kn^{\frac{2}{2\beta+d}})^{-\frac{d-2}{d}}
     \int_{\tau_{k,j}}^{\tau_{k,j+1}}\int \| \widehat{s}(t, x) - s^\star(t,x) \|^2 \vect{p}{t}(x)  \mathrm{d}x \mathrm{d}t
    \Bigg]\\
             & \leq  C\tau_k^2(\tau_kn^{\frac{2}{2\beta+d}})^{-2\frac{d-2}{d}}\left( \frac{1}{\tau_k^2}\log(n)^{C_2}n^{-\frac{2(\beta+1)}{2\beta+d}}+\frac{1}{n}\log(n)^{C_2}\left(\frac{1}{\tau_k^2}n^{\frac{d-2}{2\beta+d}}\right)+\frac{C}{n}(1+\tau_k^{-2}) \log (n)^{C_2} \right)\\
& \leq C\log (n)^{C_2} (\tau_kn^{\frac{2}{2\beta+d}})^{-2\frac{d-2}{d}} n^{-\frac{2(\beta+1)}{2\beta+d}}.
\end{align*}

Finally coming back to~\eqref{align:finalboundw1}, taking ${k^\star}:=\max \{k|\tau_k\leq n^{-\frac{2}{2\beta+d}}\}$, we conclude the proof by the bound
\begin{align*}
    \mathrm{W}_1(\widehat{p}, p^\star) 
    &\leq
    Cn^{-1}+C\log (n)^{C_2}\sum_{k=0}^{k^\star}\left( n^{-\frac{2(\beta+1)}{2\beta+d}}\right)^{1/2}+ C\log(n)^{C_2}\sum_{k=k^\star}^{m-1}\sum_{j=0}^{\Upsilon_k-1} \left( (\tau_k n^{\frac{2}{2\beta+d}})^{-2\frac{d-2}{d}} n^{-\frac{2(\beta+1)}{2\beta+d}} \right)^{1/2}\\
    &\leq C \log (n)^{C_2} n^{-\frac{\beta+1}{2\beta+d}},
\end{align*}
where we used that $k^* \leq m \leq C \log n$, together with 
\begin{align*}
\Upsilon_k (\tau_k n^{\frac{2}{2\beta+d}})^{-\frac{d-2}{d}} 
&\leq 
\left ( (\tau_k n)^\frac{2}{2\beta+d} +1\right )^{\frac{d-2}{d}} (\tau_k n^{\frac{2}{2\beta+d}})^{-\frac{d-2}{d}}
\leq 
2^{\frac{d-2}{2}}
\end{align*}
for all $k \geq k^*$.
\end{proof}

\section{Final rates and outlook}\label{sec:concluandpersp}

\subsection{Minimax optimality of SGMs}
This section presents the final convergence rates achieved under our statistical model. 
Assuming that the data distribution \( p^\star \) satisfies Assumption~\ref{assum:1}, the associated score function enjoys both the one-sided Lipschitz regularity property~\eqref{eq:onesidedlipscoreassum} and the higher-order regularity properties~\eqref{eq:reguscoretime},~\eqref{eq:reguscorespace}. 
By applying Lemma~\ref{lemma:minimaxforsmoothscore}, we then directly conclude that the score-based estimator achieves the minimax optimal rate.

\begin{maintheorem}[Main minimax upper bound]
\label{thm:cv_rates}
    Let $n\in \mathbb{N}_{>0}$, $\beta> 0$, $K\geq 1$, $p^\star\in \mathcal{A}^\beta_K$ (see~\eqref{eq:finalstatisticalmodel}) and $(X^{(1)},\ldots,X^{(n)})$ an i.i.d. sample drawn from $p^\star$.
    Then when using the score estimator~\eqref{eq:scoreestimatorheuristic} with the neural network class~\eqref{eq:classofnn},  the distribution of $\widehat{p}$ of the score-based generator $\widehat{X}_{\overline{T}-\underline{T}}$ obtained from the backward dynamic~\eqref{eq:final-generator} with $K^{-1} \leq \sigma \leq K$ and $ \|b\|_{L^\infty([0,\overline{T}-\underline{T}])} \leq K$, satisfies 
    \begin{align*}
    \sup_{p^\star \in \mathcal{A}^\beta_K}\ \mathbb{E}_{X^{(i)}\sim p^\star}[\mathrm{W}_1(p^\star,\widehat{p})]  \underset{{\mathrm{polylog}(n)}}{\lesssim}
    n^{-\frac{\beta+1}{2\beta + d}}
    .
\end{align*}
\end{maintheorem}  
\begin{proof}
From Lemma~\ref{lem:AbetaimpliesHbeta} we have that if $p^\star \in \mathcal{A}^\beta_K$, then  the score function $s^\star(t,x)$ satisfies Properties~\hyperref[{assump:scoreregu}]{$(\mathcal{P}_\beta)$}.
 Applying Lemma~\ref{lemma:minimaxforsmoothscore}, we obtain the minimax upper bound.
\end{proof} The rate $n^{-\frac{\beta+1}{2\beta + d}}$ naturally arises in the estimation of a $\beta$-smooth density in $\mathbb{R}^d$ in transportation distances.
It is known to be the minimax rate for estimating densities in $ \mathcal{H}_K^\beta$ with respect to the Wasserstein-1 distance~\cite{Uppal19}. 
In fact, a close look at the proof of~\cite[Theorem~4]{Uppal19} shows that this minimax rate still holds when densities are further restricted to have support contained in a fixed bounded set. 
What remains is to verify that our additional shape constraint does not simplify the estimation problem. This is the focus of the next result.

\begin{maintheorem}[Main minimax lower bound]
\label{thm:borneinf} 
    Let $n\in \mathbb{N}_{>0}$, $\beta> 0$, $K\geq 1$, $p^\star\in \mathcal{A}^\beta_K$ (see~\eqref{eq:finalstatisticalmodel}) and $(X^{(1)},\ldots,X^{(n)})$ be an i.i.d. sample drawn from $p^\star$. Then,
\begin{align*}
    \inf_{\widehat{p}_n} \sup_{p^\star \in \mathcal{A}^\beta_K}\ \mathbb{E}_{X^{(i)}\sim p^\star}[\mathrm{W}_1(p^\star,\widehat{p}_n)] 
      \gtrsim
     n^{-\frac{\beta+1}{2\beta + d}},
\end{align*}
where  $\widehat{p}_n$ ranges among all the estimators of $p^\star$ based on $(X^{(1)},\ldots,X^{(n)})$.
\end{maintheorem}

The proof of Theorem~\ref{thm:borneinf} builds on the argument developed in~\cite[Theorem~4]{Uppal19}, with minor adaptations to account precisely for the model $\mathcal{A}^\beta_K$.
For completeness, these modifications are detailed in Section~\ref{sec:proof_of_proposition-borneinf}.

\subsection{Perspectives}\label{sec:perspective}
Overall, our results contribute to the theoretical understanding of score-based generative models by providing a streamlined proof strategy and opening new directions for generalization. In particular, we demonstrate that regularity of the score function is sufficient to achieve minimax convergence rates, and that such regularity naturally arises under broad conditions on the data distribution. Remarkably, it suffices for the score to exhibit high-order Hölder regularity on high-probability sets, up to logarithmic factors. This observation suggests that Hölder smoothness may not be the most intrinsic notion of regularity to control; instead, the score function might possess Sobolev-type regularity with respect to the law $\vect{p}{t}$. 

This perspective gives hope that regularity results for score functions could extend to settings where the data distribution lies near lower-dimensional structures, such as manifolds. In particular future work, could be dedicated to study the following conjecture.

\begin{conjecture}
Let $P^\star$ be a probability measure with density $p^\star \in \mathcal{H}^\beta_K$ with respect to the volume measure of a $\mathcal{C}^{\beta+1}$ compact submanifold with reach bounded below. If the density of $p^\star$ is bounded away from zero, the score function associated to $p^\star$ satisfies, for all $t>0$,
$$ s^\star(t,\cdot) \in W^{\beta+\gamma,2}_{Ct^{-\frac{1+\gamma}{2}}}(\mathbb{R}^d,\vect{p}{t}),$$
where $W^{\beta+\gamma,2}(\mathbb{R}^d,\vect{p}{t})$ stands for the $L^2(p_t)$-Sobolev space of regularity $\beta+\gamma$.
\end{conjecture}

The intuition behind the conjecture is that as $t \to 0$, the measure $\vect{p}{t}$ concentrates around the support of $p^\star$, and the score function $s^\star(t, x)$ becomes increasingly regular as $x$
is closer to the manifold.

If the regularity results from~\cite{stephanovitch2025regularity} could be extended to the manifold setting, they would significantly simplify existing proofs of statistical optimality (see Section~\ref{sec:priro-works}) by leveraging the fact that the score adapts not only to the intrinsic dimension of the data distribution but also to  its regularity. We believe that pushing further the analysis of space-time regularity of the score function could lead to a better theoretical understanding of the impressive performance of score-based generative models.

\begingroup
\small

\bibliography{bib}

\begin{thebibliography}{43}
\providecommand{\natexlab}[1]{#1}
\providecommand{\url}[1]{\texttt{#1}}
\expandafter\ifx\csname urlstyle\endcsname\relax
  \providecommand{\doi}[1]{doi: #1}\else
  \providecommand{\doi}{doi: \begingroup \urlstyle{rm}\Url}\fi

\bibitem[Anderson(1982)]{anderson1982reverse}
Brian~DO Anderson.
\newblock Reverse-time diffusion equation models.
\newblock \emph{Stochastic Processes and their Applications}, 12\penalty0
  (3):\penalty0 313--326, 1982.

\bibitem[Azangulov et~al.(2024)Azangulov, Deligiannidis, and
  Rousseau]{azangulov2024convergence}
Iskander Azangulov, George Deligiannidis, and Judith Rousseau.
\newblock Convergence of diffusion models under the manifold hypothesis in
  high-dimensions.
\newblock \emph{arXiv preprint arXiv:2409.18804}, 2024.

\bibitem[Backhoff-Veraguas et~al.(2022)Backhoff-Veraguas, K{\"a}llblad, and
  Robinson]{backhoff2022adapted}
Julio Backhoff-Veraguas, Sigrid K{\"a}llblad, and Benjamin~A Robinson.
\newblock Adapted wasserstein distance between the laws of sdes.
\newblock \emph{arXiv preprint arXiv:2209.03243}, 2022.

\bibitem[Bakry et~al.(2013)Bakry, Gentil, and Ledoux]{bakry2013analysis}
Dominique Bakry, Ivan Gentil, and Michel Ledoux.
\newblock \emph{Analysis and geometry of Markov diffusion operators}, volume
  348.
\newblock Springer Science \& Business Media, 2013.

\bibitem[Benton et~al.(2023)Benton, Deligiannidis, and Doucet]{benton2023error}
Joe Benton, George Deligiannidis, and Arnaud Doucet.
\newblock Error bounds for flow matching methods.
\newblock \emph{arXiv preprint arXiv:2305.16860}, 2023.

\bibitem[Boucheron et~al.(2013)Boucheron, Lugosi, and Massart]{Boucheron13}
St\'ephane Boucheron, G\'abor Lugosi, and Pascal Massart.
\newblock \emph{Concentration inequalities}.
\newblock Oxford University Press, Oxford, 2013.
\newblock ISBN 978-0-19-953525-5.
\newblock \doi{10.1093/acprof:oso/9780199535255.001.0001}.
\newblock URL \url{https://doi.org/10.1093/acprof:oso/9780199535255.001.0001}.
\newblock A nonasymptotic theory of independence, With a foreword by Michel
  Ledoux.

\bibitem[Cai and Li(2025)]{cai2025minimax}
Changxiao Cai and Gen Li.
\newblock Minimax optimality of the probability flow ode for diffusion models.
\newblock \emph{arXiv preprint arXiv:2503.09583}, 2025.

\bibitem[Chen et~al.(2023{\natexlab{a}})Chen, Lee, and Lu]{pmlr-v202-chen23q}
Hongrui Chen, Holden Lee, and Jianfeng Lu.
\newblock Improved analysis of score-based generative modeling: User-friendly
  bounds under minimal smoothness assumptions.
\newblock In \emph{Proceedings of the 40th International Conference on Machine
  Learning}, volume 202, pages 4735--4763. PMLR, 23--29 Jul 2023{\natexlab{a}}.

\bibitem[Chen et~al.(2022)Chen, Chewi, Li, Li, Salim, and
  Zhang]{chen2022sampling}
Sitan Chen, Sinho Chewi, Jerry Li, Yuanzhi Li, Adil Salim, and Anru~R Zhang.
\newblock Sampling is as easy as learning the score: theory for diffusion
  models with minimal data assumptions.
\newblock \emph{arXiv preprint arXiv:2209.11215}, 2022.

\bibitem[Chen et~al.(2023{\natexlab{b}})Chen, Chewi, Lee, Li, Lu, and
  Salim]{chen2023probability}
Sitan Chen, Sinho Chewi, Holden Lee, Yuanzhi Li, Jianfeng Lu, and Adil Salim.
\newblock The probability flow ode is provably fast.
\newblock \emph{Advances in Neural Information Processing Systems},
  36:\penalty0 68552--68575, 2023{\natexlab{b}}.

\bibitem[De~Ryck et~al.(2021)De~Ryck, Lanthaler, and
  Mishra]{de2021approximation}
Tim De~Ryck, Samuel Lanthaler, and Siddhartha Mishra.
\newblock On the approximation of functions by tanh neural networks.
\newblock \emph{Neural Networks}, 143:\penalty0 732--750, 2021.

\bibitem[Dhariwal and Nichol(2021)]{dhariwal2021diffusion}
Prafulla Dhariwal and Alexander Nichol.
\newblock Diffusion models beat gans on image synthesis.
\newblock \emph{Advances in neural information processing systems},
  34:\penalty0 8780--8794, 2021.

\bibitem[Donoho et~al.(1996)Donoho, Johnstone, Kerkyacharian, and
  Picard]{Donoho96}
David~L. Donoho, Iain~M. Johnstone, G{\'e}rard Kerkyacharian, and Dominique
  Picard.
\newblock Density estimation by wavelet thresholding.
\newblock \emph{Ann. Stat.}, 24\penalty0 (2):\penalty0 508--539, 1996.
\newblock ISSN 0090-5364.
\newblock \doi{10.1214/aos/1032894451}.

\bibitem[Dou et~al.(2024)Dou, Kotekal, Xu, and Zhou]{dou2024optimal}
Zehao Dou, Subhodh Kotekal, Zhehao Xu, and Harrison~H Zhou.
\newblock From optimal score matching to optimal sampling.
\newblock \emph{arXiv preprint arXiv:2409.07032}, 2024.

\bibitem[Ducotterd et~al.(2024)Ducotterd, Goujon, Bohra, Perdios, Neumayer, and
  Unser]{ducotterd2024improving}
Stanislas Ducotterd, Alexis Goujon, Pakshal Bohra, Dimitris Perdios, Sebastian
  Neumayer, and Michael Unser.
\newblock Improving lipschitz-constrained neural networks by learning
  activation functions.
\newblock \emph{Journal of Machine Learning Research}, 25\penalty0
  (65):\penalty0 1--30, 2024.

\bibitem[Fazlyab et~al.(2019)Fazlyab, Robey, Hassani, Morari, and
  Pappas]{fazlyab2019efficient}
Mahyar Fazlyab, Alexander Robey, Hamed Hassani, Manfred Morari, and George
  Pappas.
\newblock Efficient and accurate estimation of lipschitz constants for deep
  neural networks.
\newblock \emph{Advances in neural information processing systems}, 32, 2019.

\bibitem[Fukumizu et~al.(2024)Fukumizu, Suzuki, Isobe, Oko, and
  Koyama]{fukumizu2024flow}
Kenji Fukumizu, Taiji Suzuki, Noboru Isobe, Kazusato Oko, and Masanori Koyama.
\newblock Flow matching achieves almost minimax optimal convergence.
\newblock \emph{arXiv preprint arXiv:2405.20879}, 2024.

\bibitem[Gentiloni-Silveri and Ocello(2025)]{gentiloni2025beyond}
Marta Gentiloni-Silveri and Antonio Ocello.
\newblock Beyond log-concavity and score regularity: Improved convergence
  bounds for score-based generative models in w2-distance.
\newblock \emph{arXiv preprint arXiv:2501.02298}, 2025.

\bibitem[Gouk et~al.(2021)Gouk, Frank, Pfahringer, and
  Cree]{gouk2021regularisation}
Henry Gouk, Eibe Frank, Bernhard Pfahringer, and Michael~J Cree.
\newblock Regularisation of neural networks by enforcing lipschitz continuity.
\newblock \emph{Machine Learning}, 110:\penalty0 393--416, 2021.

\bibitem[Hoogeboom et~al.(2022)Hoogeboom, Satorras, Vignac, and
  Welling]{hoogeboom2022equivariant}
Emiel Hoogeboom, V{\i}ctor~Garcia Satorras, Cl{\'e}ment Vignac, and Max
  Welling.
\newblock Equivariant diffusion for molecule generation in 3d.
\newblock In \emph{International conference on machine learning}, pages
  8867--8887. PMLR, 2022.

\bibitem[Huang et~al.(2024)Huang, Huang, and Lin]{huang2024convergence}
Daniel~Zhengyu Huang, Jiaoyang Huang, and Zhengjiang Lin.
\newblock Convergence analysis of probability flow ode for score-based
  generative models.
\newblock \emph{arXiv preprint arXiv:2404.09730}, 2024.

\bibitem[Hudde et~al.(2024)Hudde, Hutzenthaler, Jentzen, and
  Mazzonetto]{hudde2024ito}
Anselm Hudde, Martin Hutzenthaler, Arnulf Jentzen, and Sara Mazzonetto.
\newblock On the it{\^o}--alekseev--gr{\"o}bner formula for stochastic
  differential equations.
\newblock In \emph{Annales de l'Institut Henri Poincare (B) Probabilites et
  statistiques}, volume~60, pages 904--922. Institut Henri Poincar{\'e}, 2024.

\bibitem[Karras et~al.(2022)Karras, Aittala, Aila, and
  Laine]{karras2022elucidating}
Tero Karras, Miika Aittala, Timo Aila, and Samuli Laine.
\newblock Elucidating the design space of diffusion-based generative models.
\newblock \emph{Advances in neural information processing systems},
  35:\penalty0 26565--26577, 2022.

\bibitem[Kong et~al.(2020)Kong, Ping, Huang, Zhao, and
  Catanzaro]{kong2020diffwave}
Zhifeng Kong, Wei Ping, Jiaji Huang, Kexin Zhao, and Bryan Catanzaro.
\newblock Diffwave: A versatile diffusion model for audio synthesis.
\newblock \emph{arXiv preprint arXiv:2009.09761}, 2020.

\bibitem[Kunkel and Trabs(2025)]{kunkel2025minimax}
Lea Kunkel and Mathias Trabs.
\newblock On the minimax optimality of flow matching through the connection to
  kernel density estimation.
\newblock \emph{arXiv preprint arXiv:2504.13336}, 2025.

\bibitem[Kwon et~al.(2025)Kwon, Kim, Ohn, and Chae]{kwon2025nonparametric}
Hyeok~Kyu Kwon, Dongha Kim, Ilsang Ohn, and Minwoo Chae.
\newblock Nonparametric estimation of a factorizable density using diffusion
  models.
\newblock \emph{arXiv preprint arXiv:2501.01783}, 2025.

\bibitem[Le~Gall(2016)]{le2016brownian}
Jean-Fran{\c{c}}ois Le~Gall.
\newblock \emph{Brownian motion, martingales, and stochastic calculus}.
\newblock Springer, 2016.

\bibitem[Lee et~al.(2023)Lee, Lu, and Tan]{lee2023convergence}
Holden Lee, Jianfeng Lu, and Yixin Tan.
\newblock Convergence of score-based generative modeling for general data
  distributions.
\newblock In \emph{International Conference on Algorithmic Learning Theory},
  pages 946--985. PMLR, 2023.

\bibitem[Li et~al.(2024)Li, Wei, Chi, and Chen]{li2024sharp}
Gen Li, Yuting Wei, Yuejie Chi, and Yuxin Chen.
\newblock A sharp convergence theory for the probability flow odes of diffusion
  models.
\newblock \emph{arXiv preprint arXiv:2408.02320}, 2024.

\bibitem[Oko et~al.(2023)Oko, Akiyama, and Suzuki]{oko2023diffusion}
Kazusato Oko, Shunta Akiyama, and Taiji Suzuki.
\newblock Diffusion models are minimax optimal distribution estimators.
\newblock In \emph{International Conference on Machine Learning}, pages
  26517--26582. PMLR, 2023.

\bibitem[Pauli et~al.(2021)Pauli, Koch, Berberich, Kohler, and
  Allg{\"o}wer]{pauli2021training}
Patricia Pauli, Anne Koch, Julian Berberich, Paul Kohler, and Frank
  Allg{\"o}wer.
\newblock Training robust neural networks using lipschitz bounds.
\newblock \emph{IEEE Control Systems Letters}, 6:\penalty0 121--126, 2021.

\bibitem[Schmidt-Hieber(2020)]{schmidt2020nonparametric}
Johannes Schmidt-Hieber.
\newblock Nonparametric regression using deep neural networks with relu
  activation function.
\newblock \emph{The Annals of Statistics}, 48\penalty0 (4):\penalty0 1875,
  2020.

\bibitem[Song et~al.(2020)Song, Sohl-Dickstein, Kingma, Kumar, Ermon, and
  Poole]{song2020score}
Yang Song, Jascha Sohl-Dickstein, Diederik~P Kingma, Abhishek Kumar, Stefano
  Ermon, and Ben Poole.
\newblock Score-based generative modeling through stochastic differential
  equations.
\newblock \emph{arXiv preprint arXiv:2011.13456}, 2020.

\bibitem[St{\'e}phanovitch(2024)]{stephanovitch2024smooth}
Arthur St{\'e}phanovitch.
\newblock Smooth transport map via diffusion process.
\newblock \emph{arXiv preprint arXiv:2411.10235}, 2024.

\bibitem[St{\'e}phanovitch(2025)]{stephanovitch2025regularity}
Arthur St{\'e}phanovitch.
\newblock Regularity of the score function in generative models.
\newblock \emph{arXiv preprint arXiv:2506.19559}, 2025.

\bibitem[Tang and Yang(2024)]{pmlr-v238-tang24a}
Rong Tang and Yun Yang.
\newblock Adaptivity of diffusion models to manifold structures.
\newblock In \emph{Proceedings of The 27th International Conference on
  Artificial Intelligence and Statistics}, volume 238, pages 1648--1656. PMLR,
  02--04 May 2024.

\bibitem[Tsybakov(2008)]{Tsybakov08}
Alexandre~B. Tsybakov.
\newblock \emph{Introduction to Nonparametric Estimation}.
\newblock Springer Publishing Company, Incorporated, 1st edition, 2008.
\newblock ISBN 0387790519.

\bibitem[Uppal et~al.(2019)Uppal, Singh, and P{\'o}czos]{Uppal19}
Ananya Uppal, Shashank Singh, and Barnab{\'a}s P{\'o}czos.
\newblock Nonparametric density estimation \& convergence rates for gans under
  besov ipm losses.
\newblock \emph{Advances in neural information processing systems}, 32, 2019.

\bibitem[Vincent(2011)]{vincent2011connection}
Pascal Vincent.
\newblock A connection between score matching and denoising autoencoders.
\newblock \emph{Neural computation}, 23\penalty0 (7):\penalty0 1661--1674,
  2011.

\bibitem[Wibisono et~al.(2024)Wibisono, Wu, and Yang]{wibisono2024optimal}
Andre Wibisono, Yihong Wu, and Kaylee~Yingxi Yang.
\newblock Optimal score estimation via empirical bayes smoothing.
\newblock In \emph{The Thirty Seventh Annual Conference on Learning Theory},
  pages 4958--4991. PMLR, 2024.

\bibitem[Yakovlev and Puchkin(2025)]{yakovlev2025generalization}
Konstantin Yakovlev and Nikita Puchkin.
\newblock Generalization error bound for denoising score matching under relaxed
  manifold assumption.
\newblock \emph{arXiv preprint arXiv:2502.13662}, 2025.

\bibitem[Zhang et~al.(2024)Zhang, Yin, Liang, and Liu]{Zhang24}
Kaihong Zhang, Heqi Yin, Feng Liang, and Jingbo Liu.
\newblock Minimax optimality of score-based diffusion models: Beyond the
  density lower bound assumptions.
\newblock In \emph{Proceedings of the 41st International Conference on Machine
  Learning}, volume 235, pages 60134--60178. PMLR, 21--27 Jul 2024.

\bibitem[Zhang and Chen(2023)]{zhang2023fast}
Qinsheng Zhang and Yongxin Chen.
\newblock Fast sampling of diffusion models with exponential integrator.
\newblock In \emph{The Eleventh International Conference on Learning
  Representations}, 2023.
\newblock URL \url{https://openreview.net/forum?id=Loek7hfb46P}.

\end{thebibliography}

\endgroup

\appendix
\section{Details for the proof of Theorem~\ref{thm:cv_rates}}
\subsection{Proof of Proposition~\ref{prop:mixtures-in-model}}\label{sec:prop:mixtures-in-model}
\begin{proof}[Proof of Proposition~\ref{prop:mixtures-in-model}]
By assumption, there exists $C_A>0$ such that for all $x\notin \B(0,C_A)$ we have that $-\nabla^2 \log(q(x))\succeq (2A)^{-1} I_{d \times d}$, $-\langle \nabla \log(q(x)),x\rangle \geq (2A)^{-1}\|x\|^2 $ and $-\log(q(x))\geq (4A)^{-1}\|x\|^2$. Let $l:[-1,C_A+1]\rightarrow [0,1]$ be a smooth non-decreasing function such that $l([-1,0])=0$ and $l([C_A,C_A+1])=1$. Defining,
\begin{align*}
u(x)=\left\{\begin{array}{ll} (4A)^{-1}\|x\|^2 & \text{ if } \|x\|\leq C_A\\
l(\|x\|-C_A)(-\log(q(x)))+(1-l(\|x\|-C_A))(4A)^{-1}\|x\|^2 & \text{ if } C_A \leq \|x\|\leq 2C_A\\
-\log(q(x)) & \text{ if } 2C_A \leq \|x\|,
\end{array}\right.
\end{align*}
one can check that $u$ satisfies $ (4A)^{-1} I_{d \times d}\preceq \nabla^2 u(x)$ for all $x\in \mathbb{R}^d$ straightforwardly.
Then, defining
\begin{align*}
b(x):=\log(q(x))+u(x)+a(x),
\end{align*}
we have that $b$ satisfies Assumption~\ref{assum:1}-1. Therefore, we conclude that $p(x)=q(x)e^{a(x)}=e^{-u(x)+b(x)}$ satisfies Assumption~\ref{assum:1}.

\end{proof}

\subsection{Proof of Proposition~\ref{prop:estimatenormscore}}\label{sec:prop:estimatenormscore}

\begin{proof}[Proof of Proposition~\ref{prop:estimatenormscore}]
Let us write $x^\star := \argmin_{x} u(x)$. From the mean value theorem, we have
\begin{align}\label{eq:fghjiuygtfghj}
    \|s^\star(t,x)\| 
    &\leq \|\nabla s^\star(t,\cdot)\|_{\infty}\|x-x^\star\|+\|s^\star(t,x^\star)\|.
\end{align}
Now, writing $\varphi^{t,x}$ for the probability density of $\mathcal{N}(e^{-t}x,(1-e^{-2t})\sigma^2I_{d \times d})$, and taking $r(y):=p^\star(y)/\varphi^{\infty,0}(y)$, define the probability density
$$p^{t,x}(y) := \frac{r(y)\varphi^{t,x}(y)}{\int r(z) \varphi^{t,x}(z)\mathrm{d}z}.$$
From~\cite[Proposition~4]{stephanovitch2025regularity}, we have for all $x^{'}\in \mathbb{R}^d$,
$$s^\star(t,x^{'})=-x^{'}+\frac{e^{-t}}{1-e^{-2t}}\int (y-e^{-t}x^{'}) p^{t,x^{'}}(\mathrm{d} y)$$
and
\begin{equation}\label{eq:graidentoofscore}
\nabla s^\star(t,x^{'})=-I_{d \times d}+\frac{e^{-2t}}{(1-e^{-2t})^2}\int \left(y-\int z p^{t,x^{'}}(\mathrm{d} z)\right)^{\otimes 2} p^{t,x^{'}}(\mathrm{d} y)-\frac{e^{-2t}}{1-e^{-2t}}I_{d \times d}.
\end{equation}
As,
\begin{align*}
    \int r(z) \varphi^{t,x^\star}(z)\mathrm{d}z\geq \int_{B(x^\star,\sqrt{1-e^{-2t}})} r(z) \varphi^{t,x^\star}(z)\mathrm{d}z\geq C^{-1},
\end{align*}
and $\|x^\star\|\leq C$, we deduce that
\begin{align*}
    \|s^\star(t,x^\star)\| & =\| -x^\star+\frac{e^{-t}}{1-e^{-2t}}\int (y-e^{-t}x^\star)p^{t,x^\star}( \mathrm{d} y)\|\\
    & \leq C \left(1+t^{-1}\int \|y-e^{-t}x^\star\|r(y)\varphi^{t,x^\star}(y)\mathrm{d} y\right)
    \\
    & \leq C(1+t^{-1}).
\end{align*}
On the other hand, we have from~\eqref{eq:graidentoofscore} that
$$\inf_{x^{'}\in \mathbb{R}^d} \lambda_{\min}(\nabla s^\star(t,x^{'})) \geq -1-C e^{-2t}(1 + t^{-1}) .
$$
Putting this bound together with the estimate of Proposition~\ref{prop:lipregu}, we deduce that
$$\| \nabla s^\star(t,\cdot)\|_{\infty} \leq C (1 + t^{-1}).$$
Therefore, we obtain from~\eqref{eq:fghjiuygtfghj} that
\begin{align*}
    \|s^\star(t,x)\|& \leq C(1+(1+t^{-1})\|x-x^\star\|)\\
    & \leq C(1+t^{-1})(1+\|x\|),
\end{align*}
which concludes the proof.
\end{proof}

\subsection{Proof of Lemma~\ref{lem:AbetaimpliesHbeta}}\label{sec:lem:AbetaimpliesHbeta}

\begin{proof}[Proof of Lemma~\ref{lem:AbetaimpliesHbeta}]

Under the assumption that $p^\star\in \mathcal{A}_K^\beta$, the score $s^\star(t,x) = \nabla \log \vect{p}{t}(x)$ is regular, in the sense that it satisfies~\eqref{eq:estimatenormscore} from Proposition~\ref{prop:estimatenormscore},~\eqref{eq:onesidedlipscoreassum} from Proposition~\ref{prop:lipregu}, as well as~\eqref{eq:reguscoretime} and~\eqref{eq:reguscorespace} from Theorem~\ref{coro:regularityp}. Let us now show that it satisfies points 1. and 2. of Properties~\hyperref[{assump:scoreregu}]{$(\mathcal{P}_\beta)$}. 

As given in~\cite[Section~3.1.1]{stephanovitch2025regularity}, we have
\begin{align*}
A_t^\varepsilon:= 
\begin{cases}
A_0^\varepsilon 
= \{ y \mid u(y) \leq C (\log(\varepsilon^{-1}) + K +1 )^K\} 
& \text{if } t\leq C_\star^{-1}\log(\varepsilon^{-1})^{-C_\star}
\\
A_\infty^\varepsilon 
 = \{ y \mid \|y -y^\star\| \leq C \log(\varepsilon^{-1})(1+K)\} & \text{if }  t> C_\star^{-1}\log(\varepsilon^{-1})^{-C_\star},
    \end{cases}
\end{align*}
corresponding to the set  given by Theorem~\ref{coro:regularityp}. Now by Assumption~\ref{assum:1}.2 we have that $\mathbb{P}_{X \sim p^\star} (A_\infty^\varepsilon)\geq 1-\varepsilon$, so we only need to check Property~\hyperref[{assump:scoreregu}]{$(\mathcal{P}_\beta)$}.2.

First of all, the case $t> C_\star^{-1}\log(\varepsilon^{-1})^{-C_\star}$ is immediate.
Now, in the case $t\leq C_\star^{-1}\log(\varepsilon^{-1})^{-C_\star}$, for $x\in A_{t}^{\varepsilon}$ and $y\in \mathbb{R}^d$ such that
$$\|x-y\|\leq (KC (\log(\varepsilon^{-1}) + K +1 )^K)^{-K},$$
we have from Assumption~\ref{assum:1}.4 that
for all $r\in [0,1]$,
$$\| \nabla u(x+r(y-x))\|\leq K(1+(\log(\varepsilon^{-1}) + K +1 )^{K^2}).$$
Then, from Taylor expansion we deduce that
\begin{align*}
\|u(y)\|\leq C (\log(\varepsilon^{-1}) + K +1 )^K+ C_3\log(\varepsilon^{-1})^{C_4}\|x-y\|.
\end{align*}
Therefore, there exists $C_1^\star,C_2^\star>0$ such that for any , if $\|x-y\|\leq  (C_1^\star)^{-1}\log\left(\varepsilon^{-1}\right)^{-C_2^\star},$ we have
\begin{align*}
\|u(y)\|
&\leq C (\log(\varepsilon^{-2}) + K +1 )^K
\end{align*}
so $y\in A_{t}^{\varepsilon^{2}}$.
We can then conclude that for $p^\star\in \mathcal{A}_K^\beta$, the score function satisfies Properties~\hyperref[{assump:scoreregu}]{$(\mathcal{P}_\beta)$}.
\end{proof}

\subsection{Proof of Lemma~\ref{lem:SDEstability-varying-drift}}
\label{sec:lem:SDEstability-varying-drift}

\begin{proof}[Proof of Lemma~\ref{lem:SDEstability-varying-drift}]
Given an initial condition $x\in \mathbb{R}^d$, we have
\begin{align*}
    \mathbb{E}\Big[\|X_{\overline{\tau}}\| \mid X_{\underline{\tau}}= x\Big]
    &=
    \mathbb{E}\Big[ \|x+\int_{\underline{\tau}}^{\overline{\tau}} a_t(X_t) \mathrm{d}t+\int_{\underline{\tau}}^{\overline{\tau}} 
    r_t \mathrm{d}B_{t}\|\mid X_{\underline{\tau}}= x\Big]
    \\
    &\leq 
    \|x\|+(\overline{\tau}-\underline{\tau})V+\sqrt{d}M.
\end{align*}
On the other hand, by the $L$-subGaussian property of $p_{\underline{\tau}}$, we have that for all $0 < \varepsilon \leq 1$, 
\begin{align*}
    \int_{\B\left(0,\sqrt{L \log(1/\varepsilon)}\right)^c}
    \|x\|
    p_{\underline{\tau}}(x) \mathrm{d}x
    =
    \int_{\sqrt{L \log(1/\varepsilon)}}^\infty
    \delta
    \mathbb{P}
    (
    \| X_{\underline{\tau}} \| \geq \delta
    )
    \mathrm{d}\delta
    &\leq
    \int_{\sqrt{L \log(1/\varepsilon)}}^\infty
    \delta e^{-\delta^2/L}
    \mathrm{d} \delta
    = L\varepsilon/2
    ,
\end{align*}
with same expression for $\widetilde{p}_{\underline{\tau}}$.
As a result, we obtain
\begin{align*}
\mathrm{W}_1(p_{\overline{\tau}},\widetilde{p}_{\overline{\tau}})
&=
\sup_{f\in \text{Lip}_1}\int \mathbb{E}\Big[f(X_{\overline{\tau}}) \mid X_{\underline{\tau}}=x\Big]p_{\underline{\tau}}(x) \mathrm{d}x-\int \mathbb{E}\Big[f(\widetilde{X}_{\overline{\tau}})|\widetilde{X}_{\underline{\tau}}=x\Big]\widetilde{p}_{\underline{\tau}}(x) \mathrm{d}x\\
&=\sup_{f\in \text{Lip}_1,f(0)=0}\int \mathbb{E}\Big[f(X_{\overline{\tau}}) \mid X_{\underline{\tau}}=x\Big](p_{\underline{\tau}}(x)-\widetilde{p}_{\underline{\tau}}(x)) \mathrm{d}x\\
& \leq
\int_{\R^d} \mathbb{E}\Big[\|X_{\overline{\tau}}\| \mid X_{\underline{\tau}}=x\Big]\big|p_{\underline{\tau}}(x)-\widetilde{p}_{\underline{\tau}}(x)| \mathrm{d}x
\\
& \leq
\int_{\R^d}(\|x\|+(\overline{\tau}-\underline{\tau})V+\sqrt{d}M)|p_{\underline{\tau}}(x)-\widetilde{p}_{\underline{\tau}}(x)|
(
\mathds{1}_{x\in \B(0,\sqrt{L\log(\varepsilon^{-1})})}
+
\mathds{1}_{x\notin \B(0,\sqrt{L\log(\varepsilon^{-1})})}
) 
\mathrm{d}x
\\
&\leq
\bigl(
    (\overline{\tau}-\underline{\tau})V+\sqrt{d}M
\bigr)
\|p_{\underline{\tau}}-\widetilde{p}_{\underline{\tau}}\|_{L^1}
+
\bigl(
    \sqrt{L\log(\varepsilon^{-1})}
    \|p_{\underline{\tau}}-\widetilde{p}_{\underline{\tau}}\|_{L^1}
    +
    L\varepsilon/2
\bigr)
,
\end{align*}
which concludes the proof.
\end{proof}

\subsection{Proof of Theorem~\ref{thm:stabilitysde}}
\label{sec:thm:stabilitysde}

\begin{proof}[Proof of Theorem~\ref{thm:stabilitysde}]
Consider the synchronous coupling given by initial condition $ X_{\underline{\tau}} =  \bar{X}_{\underline{\tau}} \sim  p_{\underline{\tau}}$.
Then for all $x\in \mathbb{R}^d$, we have
\begin{align}\label{align:boundWdiffdrift}
\mathrm{W}_1(p_{\overline{\tau}}, \bar{p}_{\overline{\tau}})&=\sup_{f\in \text{Lip}_1}\int f(x)\mathrm{d}p_{\overline{\tau}}(x)-\int f(x)\bar{p}_{\overline{\tau}}(x) \mathrm{d}x\nonumber
\\
&=\sup_{f\in \text{Lip}_1}\int \mathbb{E}\Big[f(X_{\overline{\tau}})-f(\bar{X}_{\overline{\tau}})\big|X_{\underline{\tau}}=x,\bar{X}_{\underline{\tau}}=x\Big]p_{\underline{\tau}}(x)\mathrm{d}x\nonumber\\
& \leq\int \mathbb{E}\Big[\|X_{\overline{\tau}}-\bar{X}_{\overline{\tau}}\|\big|X_{\underline{\tau}}=x,\bar{X}_{\underline{\tau}}=x\Big]p_{\underline{\tau}}(x) \mathrm{d}x.
\end{align}
Now, using the stochastic Itô-Alekseev-Gröbner formula~\cite[Theorem~1.1]{hudde2024ito}, we have
\begin{align*}
\bar{X}_{\overline{\tau}}- X_{\overline{\tau}} =  \int_{\underline{\tau}}^{\overline{\tau}}\nabla \bar{X}_{s,\overline{\tau}}(X_s)(\bar{a}_s(X_s)-a_s(X_s))\mathrm{d}s,
\end{align*}
where for all $y \in \R^d$, the process $(\bar{X}_{s,\overline{\tau}}(y))_s$ is solution to the equation
$$\bar{X}_{s,\overline{\tau}}(y)
=
y
+
\int_s^{\overline{\tau}} \bar{a}_t(\bar{X}_{t,\overline{\tau}}(y))\mathrm{d}t+ \int_s^{\overline{\tau}} r_t \mathrm{d}B_t.
$$
Because $\bar{a}$ is locally Lipschitz, $\bar{X}_{s,\overline{\tau}}$ is differentiable and 
\begin{align*}
\nabla \bar{X}_{s,\overline{\tau}}(y)=
I_{d \times d}
+
\int_s^{\overline{\tau}} 
\nabla 
\bar{a}_t(\bar{X}_{t,\overline{\tau}}(y))
\bar{X}_{t,\overline{\tau}}(y)
\mathrm{d}t
.
\end{align*}
Given a vector $v \in \R^d$, consider $R(s) := \|\nabla \bar{X}_{\overline{\tau}-s,\overline{\tau}}(y)v\|^2$. Then, we have
\begin{align*}
    \partial_s R(s) 
    &= 
    2\left\langle \nabla \bar{X}_{\overline{\tau}-s,\overline{\tau}}(y)v
    , 
    - \partial_s \nabla \bar{X}_{\overline{\tau}-s,\overline{\tau}}(y)v \right\rangle
    \\
    &=
    2\left\langle 
    \nabla \bar{X}_{\overline{\tau}-s,\overline{\tau}}(y)
	v    
	, 
    \nabla \bar{a}_{\overline{\tau}-s}(\bar{X}_{\overline{\tau}-s
	,
    \overline{\tau}}) 
    \nabla \bar{X}_{\overline{\tau}-s,\overline{\tau}}(y) v
    \right\rangle
    \\
    &\leq  2 \lambda_{\max}(\nabla\bar{a}_{\overline{\tau}-s}(\bar{X}_{\overline{\tau}-s,\overline{\tau}}))R(s)
    \\
    &\leq  2 L_{\overline{\tau}-s}R(s)
    .
\end{align*}
Since $R(0) = \|\nabla \bar{X}_{\overline{\tau},\overline{\tau}}(y)v\|^2 = \| v \|^2$ Gronwall's inequality yields
$$R(s)\leq \| v \|^2 e^{2 \int_{0}^s \bar{L}_{\overline{\tau}-u}\mathrm{d}u}
,
$$
or equivalently,
$$\|\nabla \bar{X}_{s,\overline{\tau}}(y)v\|\leq
e^{\int_s^{\overline{\tau}} \bar{L}_u\mathrm{d}u}
\| v \|
.
$$
As a result, we obtain
\begin{align*}
 \|X_{\overline{\tau}}-\bar{X}_{\overline{\tau}}\| & \leq \int_{\underline{\tau}}^{\overline{\tau}}\|\nabla \bar{X}_{s,\overline{\tau}}(X_s)\|\|(a_s(X_s)-\bar{a}_s(X_s)\|\mathrm{d}s
 \\
 & \leq
\int_{\underline{\tau}}^{\overline{\tau}}e^{\int_s^{\overline{\tau}} \bar{L}_u\mathrm{d}u}\|(a_s(X_s)-\bar{a}_s(X_s)\|\mathrm{d}s,
\end{align*}
which combined with~\eqref{align:boundWdiffdrift}
yields
\begin{align*}
    \mathrm{W}_1(p_{\overline{\tau}}, \bar{p}_{\overline{\tau}})
& \leq\int \mathbb{E}\Big[\|X_{\overline{\tau}}-\bar{X}_{\overline{\tau}}\|\big|X_{\underline{\tau}}=x,\bar{X}_{\underline{\tau}}=x\Big]p_{\underline{\tau}}(x)\mathrm{d}x\\
&\leq
\int \mathbb{E}\Big[\int_{\underline{\tau}}^{\overline{\tau}}e^{\int_s^{\overline{\tau}} \bar{L}_u\mathrm{d}u}\|a_s(X_s)-\bar{a}_s(X_s)\|\mathrm{d}s\big|X_{\underline{\tau}}=x,\bar{X}_{\underline{\tau}}=x\Big]p_{\underline{\tau}}(x)\mathrm{d}x\\
&=
\int_{\underline{\tau}}^{\overline{\tau}} \int e^{\int_s^{\overline{\tau}} \bar{L}_u\mathrm{d}u}\|a_s(X_s)-\bar{a}_s(X_s)\|p_{s}(x)
\mathrm{d}x
\mathrm{d}s.
\end{align*}
\end{proof}

\subsection{Proof of Proposition~\ref{prop:tradeoffbv}}\label{sec:prop:tradeoffbv}
The proof of Proposition~\ref{prop:tradeoffbv} is based on the following deviation result.
\begin{proposition}[A clipped deviation bound for empirical risk minimizers]
\label{prop:Bernstein_union}
Let $p^\star$ be a distribution on $\R^d$, and let $A_n$ be such that
\begin{align*}
    \P_{p^\star}(A_n) \geq 1-n^{-4}.
\end{align*}
For a class of function $\mathcal{F}$ on $\R^d$, we let $c_n$ and $v$ be such that
\begin{align*}
    \sup_{x \in A_n,f \in \mathcal{F}} |f(x)| \leq c_n
    ,
    \text{\vspace{2em} and \vspace{2em}}
    \sup_{f \in \mathcal{F}} \E_{X \sim p^\star} f^2(X) \leq v
    .
\end{align*}
For all $f \in \mathcal{F}$, we denote by $f_c$ its clipped version, that is
\begin{align*}
    f_c(x) := \mathrm{sg}(f(x))(|f(x)| \wedge c_n).
\end{align*}
For $\varepsilon >0$, recall that $\mathcal{N}(\mathcal{F},\|\cdot\|_\infty, \varepsilon)$ denote the covering number of $\mathcal{F}$ at scale $\varepsilon$, in terms of sup-norm. Let $X^{(n)}, \hdots, X^{(n)}$ be an i.i.d. sample drawn from $p^\star$, and let 
\begin{align*}
    \bar{f} & \in \arg\min_{f \in \mathcal{F}} \E_{X \sim p^\star}(f(X)), \\
    \widehat{f} & \in \arg\min_{f \in \mathcal{F}} \E_n(f),
\end{align*}
where $\E_n(f)$ denotes expectation of $f$ with respect to sample distribution.
Then, for any $t$, $\varepsilon$ and $\delta >0$, with probability larger than $1-n^{-3} - 2e^{-t}$, it holds
\begin{align*}
     \E_{X\sim p^\star} \widehat{f}(X) 
    & \leq \E_{X\sim p^\star} \bar{f}(X) +  \delta(\mathrm{Var}_{X \sim p^\star} (\bar{f}_c) + \mathrm{Var}_{X \sim p^\star} (\widehat{f}_c)) + 2\sqrt{v}n^{-2} \\
    & \quad +  2 \left (c_n + \frac{1}{\delta} \right ) \frac{t+\log \left (\mathcal{N}(\mathcal{F},\|\cdot\|_\infty, \varepsilon))\right )}{n} + 4 \varepsilon + 16 \varepsilon^2.  
\end{align*}
Moreover, it holds
\begin{multline*}
     \E_{\mathbb{X}^{(n)}} \left [ \E_{X\sim p^\star} \widehat{f}(X) - \E_{X\sim p^\star} \bar{f}(X) - \delta \mathrm{Var}_{X \sim p^\star}(\widehat{f}_c)) \right ]  \leq \delta \mathrm{Var}_{X \sim p^\star} (\bar{f}_c) \\ 
     + 6 \left (c_n + \frac{1}{\delta} \right ) \frac{\log \left (\mathcal{N}(\mathcal{F},\|\cdot\|_\infty, \varepsilon)) \vee n\right )}{n} 
     \quad + 5 \sqrt{v}n^{-2}  + 4 \varepsilon + 16 \varepsilon^2,
\end{multline*}
where $\E_{\mathbb{X}^{(n)}}$ denotes the expectation with respect to the randomness of data $\mathbb{X}^{(n)} = \{X^{(1)},\ldots,X^{(n)}\}$.
\end{proposition}

The proof of Proposition~\ref{prop:Bernstein_union} is postponed to Section~\ref{sec:proof_prop_Bernstein_union}.
As in~\cite{schmidt2020nonparametric} and~\cite{oko2023diffusion}, the proof of Proposition~\ref{prop:Bernstein_union} relies on Bernstein's deviation inequality, combined with a chaining argument.

In the particular context of denoising score matching with $p^\star \in \mathcal{A}_K^\beta$, Proposition~\ref{prop:tradeoffbv} will use the following setup.
Take $\mathcal{S} \subset \mathcal{H}_V^0$, and let $s\in \mathcal{S}$.
Fix time horizons $n^{-1} \leq \underline{\tau} < \overline{\tau}$. 
Recall that the denoising score matching contrast $\gamma_{\underline{\tau},\overline{\tau}}(s,x)$ is defined by~\eqref{eq:contrast}.
Define the excess contrast by
\begin{align}
\label{eq:excess-contrast}
    \nu(s,x) 
    := \gamma_{\underline{\tau},\overline{\tau}}(s,x) - \gamma_{\underline{\tau},\overline{\tau}}(s^\star,x).
\end{align}
According to Theorem~\ref{thm:Vincent}, we have
\begin{align*}
    \E_{X \sim p^\star} [ \nu(s,X) ] = \int_{\underline{\tau}}^{\overline{\tau}} \int_{\R^d} \|s(t,x) - s^\star(t,x)\|^2\vect{p}{t}(x)\mathrm{d} x \mathrm{d} t,
\end{align*}
so that $s^\star$ minimizes the expected excess contrast over all measurable functions $s : [\underline{\tau},\overline{\tau}] \times \R^d \to \R^d$. 
To prove Proposition~\ref{prop:tradeoffbv}, a relevant strategy is then to apply  Proposition~\ref{prop:Bernstein_union} to the family functions $\mathcal{F}_{\mathcal{S}} = \{\nu(s,.)\}_{s \in \mathcal{S}}$. 
This leads us to upper bound their variances and sup-norms — the objective of the next lemma.

\begin{lemma}\label{lem:supnorm_var_contrast}
Assume that $p^\star$ is $K$-subGaussian, and that $s^\star$ satisfies Properties~\hyperref[{assump:scoreregu}]{$(\mathcal{P}_\beta)$}. 
Let $\mathcal{S} \subset \mathcal{H}_V^0$, $ \sigma >0$, and $\nu(s,x)$ be defined as in~\eqref{eq:excess-contrast}. Then there exists $A_n \subset \B(0, C \log(n)^C) $ such that $\P_{p^\star}(A_n) \geq 1-n^{-4}$ and, for any $\underline{\tau} \geq n^{-1}$, 
\begin{align*}
    \sup_{s \in \mathcal{S},x \in A_n} | \nu(s,x) | \leq C(\overline{\tau}-\underline{\tau})\log(n)^C g(\underline{\tau},\sigma,V),
\end{align*}
where $C$ depends on $K$ and $d$ only and  $g(\underline{\tau},\sigma,V) = 1 + V^2 + (1+\sigma+{\sigma^{-1}})^4 (1+\underline{\tau}^{-1}) + \sigma_{\underline{\tau}}^{-2} $. In addition,
\begin{align*}
    \sup_{s \in \mathcal{S}} \E_{X \sim p^\star} \nu^2(s,X) \leq C (\overline{\tau}-\underline{\tau})^2 g^2(\underline{\tau}, \sigma,V),
\end{align*}
where $C$ depends only on $K$. Moreover, it holds
\begin{align*}
    \E_{X \sim p^\star} \nu_c^2(s,X) \leq c_n \E_{X \sim p^\star} \nu(s,X) + c_n^2n^{-4},
\end{align*}
where $c_n = C(\overline{\tau}-\underline{\tau}) \log(n)^C g(\underline{\tau},\sigma,V)$, for $C$ large enough depending on $K$.
\end{lemma}

The proof of Lemma~\ref{lem:supnorm_var_contrast} is postponed to Section~\ref{sec:proof_lem_supnor_var_contrast}. Equipped with Proposition~\ref{prop:Bernstein_union} and Lemma~\ref{lem:supnorm_var_contrast}, we are in position to prove Proposition~\ref{prop:tradeoffbv}. 

\begin{proof}[Proof of Proposition~\ref{prop:tradeoffbv}]

Recall that the centered contrast function $\nu$ is defined by~\eqref{eq:excess-contrast}. Let us denote by $c_n :=C(\overline{\tau}-\underline{\tau})\log(n)^C g(\underline{\tau},\sigma,V)$, for $C$ large enough depending on $K$. Applying Proposition~\ref{prop:Bernstein_union} to $\mathcal{F}_{\mathcal{S}} := \{\nu(s,.)\}_{s \in \mathcal{S}}$ with $\delta = 1/2$ leads to
\begin{multline*}
     \E_{\mathbb{X}^{(n)}}  \left [  \E_{X \sim p^\star} (\nu(\widehat{s},X)) - \nu(\bar{s},X)) - (1/2)\mathrm{Var}_{X \sim p^\star} (\nu_c(\widehat{s},X)) \right ] \\
     \leq (1/2)\mathrm{Var}_{X \sim p^\star}(\nu_c(\bar{s},X)) + c_n \frac{\log\left ( \mathcal{N}(\mathcal{F}_{\mathcal{S}}, \|\cdot\|_\infty,\varepsilon) \vee n \right )}{n}  
     + 4 \varepsilon + 16 \varepsilon^2.
\end{multline*}

Using the connection between $\mathrm{Var}_{X \sim p^\star} \nu(s,X)$ and $\E_{X \sim p^\star} \nu(s,X)$ provided by Lemma~\ref{lem:supnorm_var_contrast}, it follows that
\begin{align*}
    \E_{\mathbb{X}^{(n)}, X \sim p^\star} \nu(\widehat{s},X) & \leq (1/2) \E_{\mathbb{X}^{(n)}, X \sim p^\star} \nu(\widehat{s},X) + (3/2) \E_{X \sim p^\star} \nu(\bar{s},X) + c_0c_n \frac{\log\left ( \mathcal{N}(\mathcal{F}_{\mathcal{S}}, \|\cdot\|_\infty,\varepsilon) \vee n \right )}{n}  \\
    & \quad + 4 \varepsilon + 16 \varepsilon^2, 
\end{align*}
for some constant $c_0$. It follows that
\begin{align}\label{eq:oracle_1}
    \E_{\mathbb{X}^{(n)}, X \sim p^\star} \nu(\widehat{s},X) \leq 3 \E_{X \sim p^\star} \nu(\bar{s},X) + c_0c_n \frac{\log\left ( \mathcal{N}(\mathcal{F}_{\mathcal{S}}, \|\cdot\|_\infty,\varepsilon) \vee n \right )}{n} + 8 \varepsilon + 32 \varepsilon^2. 
\end{align}
It remains to connect the covering numbers of $\mathcal{F}_{\mathcal{S}}$ with that of $\mathcal{S}$.
To do so, let $Z \sim \mathcal{N}(0,I_{d \times d})$. We note that for all $s,s' \in \mathcal{S}$ and $x \in \R^d$, definition~\eqref{eq:excess-contrast} and a direct development of~\eqref{eq:contrast} yield
\begin{align*}
    | \nu(s,x) - \nu(s',x) |
    &= | \gamma(s,x) - \gamma(s',x) |
    \\
    & = 
    \left|
    \int_{\underline{\tau}}^{\overline{\tau}} (\E_Z \| s(t,e^{-\lambda t}x + \sigma_t Z) - Z/\sigma_t \|^2 - \E_Z \| s'(t,e^{-\lambda t}x + \sigma_t Z) - Z/\sigma_t \|^2)\mathrm{d}t
    \right|
    \\
    & = 
    \left|
    \int_{\underline{\tau}}^{\overline{\tau}} \E_Z\left \langle (s-s')(t,e^{-\lambda t}x + \sigma_t Z), (s+s')(t,e^{-\lambda t}x + \sigma_t Z) + (2Z/\sigma_t) \right\rangle \mathrm{d}t 
    \right|
    \\
    & \leq \int_{\underline{\tau}}^{\overline{\tau}} \|s-s'\|_\infty (2 V + 2 \sqrt{d}/\sigma_t) \mathrm{d}t \\
    & \leq C_d (V+\sigma_{\underline{\tau}}^{-1}) (\overline{\tau}-\underline{\tau}) \|s-s'\|_\infty,
\end{align*}
where we used that $(\|s\|_{\infty} \vee \|s'\|_\infty) \leq V$ and $t\mapsto \sigma_t$ is non-increasing. Therefore, 
\begin{align*}
    \mathcal{N}(\mathcal{F}_\mathcal{S},\|\cdot\|_\infty,g(\underline{\tau}, \sigma,V)(\overline{\tau}-\underline{\tau})\varepsilon) \leq \mathcal{N}(\mathcal{S},\|\cdot\|_\infty,\varepsilon).
\end{align*}
Choosing $\varepsilon := g(\underline{\tau}, \sigma,V)(\overline{\tau}-\underline{\tau})/n$ in~\eqref{eq:oracle_1} yields, for $n$ large enough,
\begin{align*}
 \E_{\mathbb{X}^{(n)}, X \sim p^\star} \nu(\widehat{s},X) \leq 3 \E_{X \sim p^\star} \nu(\bar{s},X) + C_K(\overline{\tau}-\underline{\tau})\log(n)^{C_K} g(\underline{\tau}, \sigma,V) \frac{\log\left ( \mathcal{N}(\mathcal{S},\|\cdot\|_\infty,n^{-1}) \vee n \right )}{n}, 
\end{align*}
and hence the result.
\end{proof}

\subsubsection{Proof of Proposition~\ref{prop:Bernstein_union}}\label{sec:proof_prop_Bernstein_union}
To begin with the proof of Proposition~\ref{prop:Bernstein_union}, a lemma connecting clipped and non-clipped expectation is needed.
\begin{lemma}\label{lem:connection_clip_unclip}
Let $p^\star$ be a distribution on $\R^d$, and $A_n \subset \R^d$ be such that $\P_{p^\star}(A_n^c) \leq n^{-4}$, and 
\begin{center}
$\sup_{x \in A_n,f \in \mathcal{F}} |f(x)| \leq c_n$, $\sup_{f \in \mathcal{F}} \E_{X \sim p^\star} f^2(X) \leq v$. 
\end{center}
Then, for any $f \in \mathcal{F}$,
\begin{align*}
    \left | \E_{X \sim p^\star} f(X) - \E_{X \sim p^\star} f_c(X) \right | \leq \sqrt{v}n^{-2}.
\end{align*}
\end{lemma}
\begin{proof}[Proof of Lemma~\ref{lem:connection_clip_unclip}]
Let $f \in \mathcal{F}$. We decompose $\left | \E_{X \sim p^\star} (f(X) -  f_c(X)) \right |$ as follows.
\begin{align*}
    \left | \E_{X \sim p^\star} (f(X) -  f_c(X)) \right | & = \left | \E_{X \sim p^\star} (f -  f_c)\1_{A_n}(X) \right | + \left | \E_{X \sim p^\star} (f -  f_c)\1_{A_n^c}(X) \right |.
\end{align*}
According to the definition of $c_n$, $f(X)\1_{X \in A_n} = f_c(X)\1_{X \in A_n}$, so that 
\begin{align*}
\left | \E_{X \sim p^\star} (f(X) -  f_c(X)) \right | & = \left | \E_{X \sim p^\star} (f -  f_c)\1_{A_n^c}(X) \right | \\
& \leq \E_{X \sim p^\star} ( \left | f \1_{A_n^c}(X) \right | ) \\
& \leq \sqrt{\E_{X \sim p^\star} f^2(X)} \sqrt{\P_{p^\star} (A_n^c)} \\
& \leq \sqrt{\E_{X \sim p^\star} f^2(X)} n^{-2},
\end{align*}
where we used $|f(x) - f_c(x)| \leq |f(x)|$ and Cauchy-Schwarz inequality.
\end{proof}

We are now in position to prove Proposition~\ref{prop:Bernstein_union}.
\begin{proof}[Proof of Proposition~\ref{prop:Bernstein_union}]
Let $\widehat{f} \in \arg\min_{f \in \mathcal{F}} \E_n(f)$ and $\bar{f} \in \arg\min_{f \in \mathcal{F}} \E_{X \sim p^\star} f(X)$. Then it holds
\begin{align}\label{eq:oracle_0}
        \E_{X \sim p^\star} \widehat{f}(X) & = (\E_{X \sim p^\star} - \E_n) \widehat{f}(X) + \E_n \widehat{f}(X) \notag \\
        & \leq (\E_{X \sim p^\star} - \E_n) \widehat{f}(X) + \E_n \bar{f}(X) \notag \\
        & = \E_{X \sim p^\star} \bar{f}(X) + (\E_{X \sim p^\star} - \E_n) \widehat{f}(X) + (\E_n-\E_{X \sim p^\star}) \bar{f}(X).
\end{align}
We let $\mathcal{C}_{\infty}(\mathcal{F},\varepsilon) \subset \mathcal{F}$ be a minimal $\varepsilon$-covering of $\mathcal{F}$ in sup norm, with cardinal $|\mathcal{C}_{\infty}(\mathcal{F},\varepsilon)| = \mathcal{N}(\mathcal{F},\|\cdot\|_\infty,\varepsilon)$. 
For $f \in \mathcal{F}$, let $f_\varepsilon(f) \in \mathcal{C}_{\infty}(\mathcal{F},\varepsilon)$ denote any function such that $\|f - f_\varepsilon(f)\|_\infty \leq \varepsilon$.

Let us bound the first deviation term of~\eqref{eq:oracle_0}. Straightforward computation entails that
\begin{align}\label{eq:oracle_2}
    (\E_{X \sim p^\star} - \E_n) \widehat{f} & \leq  (\E_{X \sim p^\star} - \E_n) f_\varepsilon(\widehat{f}) + 2 \varepsilon. 
\end{align}

Next, denote by $A$ the event $\left \{ \forall i,\; X^{(i)} \in A_n \right \}$, that has probability larger than $1-n^{-3}$. For a fixed $f_j \in \mathcal{C}_{\infty}(\mathcal{F},\varepsilon)$, we let $f_{j,c}$ denote its clipped version.
On $A$ it holds $\E_n(f_j) = \E_n(f_{j,c})$. Since Lemma~\ref{lem:connection_clip_unclip} entails that 
\begin{align*}
    \E_{X \sim p^\star}f_{j}(X) 
    & \leq \E_{X\sim p^\star} f_{j,c}(X) + \sqrt{v}n^{-2},
\end{align*}
on $A$ it holds
\begin{align}\label{eq:oracle_3}
    (\E_{X \sim p^\star} - \E_n) f_j(X) \leq (\E_{X \sim p^\star} - \E_n) f_{j,c}(X) + \sqrt{v}n^{-2}. 
\end{align}
Applying a standard Bernstein inequality (see, e.g,~\cite[Corollary~2.11]{Boucheron13}) $f_{j,c}$ leads to
\begin{align*}
    |(\E_{X \sim p^\star} - \E_n) f_{j,c}| & \leq \sqrt{2 \mathrm{Var}_{X \sim p^\star}(f_{j,c}) t/n} + c_nt/n \\
    & \leq \delta \mathrm{Var}_{X \sim p^\star} (f_{j,c}) + \left (c_n + \frac{1}{2\delta} \right ) \frac{t}{n},
\end{align*}
 with probability larger than $1-2e^{-t}$.

Combining the above calculation with~\eqref{eq:oracle_3} leads to
\begin{align*}
    \P \left [A \cap \left \{ |(\E_{X \sim p^\star} - \E_n) f_j(X)| \geq \delta \mathrm{Var}_{X \sim p^\star} (f_{j,c}) + \sqrt{v}n^{-2} + \left (c_n + \frac{1}{2\delta} \right ) \frac{t}{n}  \right \} \right ] \leq 2 e^{-t}.
\end{align*}
Recalling that $|\mathcal{C}_{\infty}(\mathcal{F},\varepsilon)| = \mathcal{N}(\mathcal{F},\|\cdot\|_\infty,\varepsilon)$, a standard union bound yields that
\begin{align}\label{eq:oracle_4}
    \P \left [A \cap \left \{ \exists j \mid |(\E_{X \sim p^\star} - \E_n) f_j| \geq \delta \mathrm{Var}_{X \sim p^\star} (f_{j,c}) + \sqrt{v}n^{-2} + \left (c_n + \frac{1}{2 \delta} \right ) \frac{t+\log\left ( \mathcal{N}(\mathcal{F},\|\cdot\|_\infty,\varepsilon) \right )}{n} \right \} \right ] \leq 2 e^{-t}.
\end{align}
Denoting by $B_{t,n}$ the event 
\begin{align*}
    B_{t,n} = \left \{ \forall j \quad |(\E_{X \sim p^\star} - \E_n) f_j| \leq \delta \mathrm{Var}_{X \sim p^\star} (f_{j,c}) + \sqrt{v}n^{-2} + \left (c_n + \frac{1}{2 \delta} \right ) \frac{t+\log\left ( \mathcal{N}(\mathcal{F},\|\cdot\|_\infty,\varepsilon) \right )}{n} \right \}, 
\end{align*}
it comes $\P(B_{t,n}^c) \leq \P(B_{t,n}^c \cap A) + \P(A^c) \leq n^{-3} + 2 e^{-t}$. Furthermore, on $B_{t,n}$,~\eqref{eq:oracle_2} entails
\begin{align}\label{eq:oracle_5}
      & (\E_{X \sim p^\star} - \E_n) \widehat{f}  \notag \\
     &  \quad  \leq  (\E_{X \sim p^\star} - \E_n) f_\varepsilon(\widehat{f})  + 2 \varepsilon \notag \\
    &  \quad  \leq \delta \mathrm{Var}_{X \sim p^\star} (f_{\varepsilon}(\widehat{f})_c) + \sqrt{v}n^{-2} + \left (c_n + \frac{1}{2 \delta} \right ) \frac{t+\log\left ( \mathcal{N}(\mathcal{F},\|\cdot\|_\infty,\varepsilon) \right )}{n} +2 \varepsilon \notag  \\
    & \quad \leq   2 \delta \mathrm{Var}_{X \sim p^\star} (\widehat{f}_c) + \sqrt{v}n^{-2} + \left (c_n + \frac{1}{2 \delta} \right ) \frac{t+\log\left ( \mathcal{N}(\mathcal{F},\|\cdot\|_\infty,\varepsilon) \right )}{n} + 2 \varepsilon + 8 \varepsilon^2,
\end{align}
where we used that, for any $f \in \mathcal{F}$,
\begin{align*}
    \mathrm{Var}_{X \sim p^\star} (f_c) & \leq 2 \mathrm{Var}_{X \sim p^\star} ((f_\varepsilon(f))_c) + 8 \| f_c - (f_\varepsilon(f))_c\|_\infty^2 \\
    & \leq 2 \mathrm{Var}_{X \sim p^\star} ((f_\varepsilon(f))_c) + 8 \varepsilon^2.
\end{align*}
Using the same computation as above for $(\E_n-\E_{X \sim p^\star}) \bar{f}$ leads to
\begin{align}\label{eq:oracle_6}
    (\E_n-\E_{X \sim p^\star}) \bar{f} \leq 2 \delta \mathrm{Var}_{X \sim p^\star} (\bar{f}_c) + \sqrt{v}n^{-2} + \left (c_n + \frac{1}{2\delta} \right ) \frac{t+\log\left ( \mathcal{N}(\mathcal{F},\|\cdot\|_\infty,\varepsilon) \right )}{n} + 2 \varepsilon + 8 \varepsilon^2,
\end{align}
on $B_{t,n}$. Plugging~\eqref{eq:oracle_5} and~\eqref{eq:oracle_6} into~\eqref{eq:oracle_0} yields that, still on $B_{t,n}$, 
\begin{align*}
     \E_{X\sim p^\star} \widehat{f}(X) 
    & \leq \E_{X\sim p^\star} \bar{f}(X) + 2 \delta(\mathrm{Var}_{X \sim p^\star} (\bar{f}_c) + \mathrm{Var}_{X \sim p^\star}) (\widehat{f}_c) + 2\sqrt{v}n^{-2} \\
    & \quad +  2 \left (c_n + \frac{1}{2\delta} \right ) \frac{t+\log\left ( \mathcal{N}(\mathcal{F},\|\cdot\|_\infty,\varepsilon) \right )}{n} + 4 \varepsilon + 16 \varepsilon^2,  
\end{align*}
that gives the deviation bound of Proposition~\ref{prop:Bernstein_union}, replacing $2\delta$ by $\delta$.

For the expectation bound, Jensen inequality entails $\E_{X \sim p^\star} f(X) \leq \sqrt{v}$, for all $f \in \mathcal{F}$. Using the deviation bound from above, we may write
\begin{align*}
    & \E_{\mathbb{X}^{(n)}} \left [ \E_{X\sim p^\star} \widehat{f}(X) - \E_{X\sim p^\star} \bar{f}(X) - \delta \mathrm{Var}_{X \sim p^\star}(\widehat{f})) \right ] \\
    & \quad \leq  \E_{\mathbb{X}^{(n)}} \left ( \left [\E_{X\sim p^\star} ( \widehat{f}(X))  - \E_{X\sim p^\star} \bar{f}(X)- \delta \mathrm{Var}_{X \sim p^\star}(\widehat{f}) \right ]\1_{B_{t,n}} \right ) + \sqrt{v}(n^{-3} + 2e^{-t}) \\
    & \quad \leq  \delta \mathrm{Var}_{X \sim p^\star} (\bar{f}_c) + 2\sqrt{v}n^{-2} + 2 \left (c_n + \frac{1}{\delta} \right ) \frac{t+\log\left ( \mathcal{N}(\mathcal{F},\|\cdot\|_\infty,\varepsilon) \right )}{n} + 4 \varepsilon + 16 \varepsilon^2  \\
    & \qquad + \sqrt{v}(n^{-3} + 2e^{-t}). 
\end{align*}
Choosing $t = 2 \log(n)$ yields
\begin{align*}
     \E_{\mathbb{X}^{(n)}} \left [ \E_{X\sim p^\star} \widehat{f}(X) - \E_{X\sim p^\star} \bar{f}(X) - \delta \mathrm{Var}_{X \sim p^\star}(\widehat{f})) \right ] &  \leq \delta \mathrm{Var}_{X \sim p^\star} (\bar{f}_c) + 6 \left (c_n + \frac{1}{\delta} \right ) \frac{\log\left ( \mathcal{N}(\mathcal{F},\|\cdot\|_\infty,\varepsilon) \vee n \right )}{n} \\
    & \quad + 5 \sqrt{v}n^{-2}  + 4 \varepsilon + 16 \varepsilon^2,
\end{align*}
that is the expectation bound of Proposition~\ref{prop:Bernstein_union}.
\end{proof}

\subsubsection{Proof of Lemma~\ref{lem:supnorm_var_contrast}}
\label{sec:proof_lem_supnor_var_contrast}

\begin{proof}[Proof of Lemma~\ref{lem:supnorm_var_contrast}]
Let $x \in \R^d$ and $\|s\|_\infty \leq V$. Then
\begin{align*}
    \gamma_{\underline{\tau},\overline{\tau}}(s,x) & \leq 2 \int_{\underline{\tau}}^{\overline{\tau}} \E_Z (\|s(t,e^{-\lambda t}x + \sigma_t Z)\|^2 + \sigma_t^{-2}\E_Z \|Z\|^2 \mathrm{d}t \\
    & \leq 2 V^2 (\overline{\tau}-\underline{\tau}) + 2d \int_{\underline{\tau}}^{\overline{\tau}} \sigma_{t}^{-2} \mathrm{d}t
    \\
    & \leq C(\overline{\tau}-\underline{\tau})( V^2+\sigma_{\underline{\tau}}^{-2}) ),
\end{align*}
since $t \mapsto \sigma_t^{-2}$ is non-increasing.

Since $s^\star$ satisfies Properties~\hyperref[{assump:scoreregu}]{$(\mathcal{P}_\beta)$}, it holds that  
\begin{align*}
    \P_{X \sim p^\star} \left ( A_0^{\varepsilon_n} \cap A_{\infty}^{\varepsilon_n} \right ) \geq 1 - n^{-4},
\end{align*}
for $\varepsilon_n=n^{-C}$ and $A_0^{\varepsilon_n}$, $A_\infty^{\varepsilon_n}$ onto which Theorem~\ref{coro:regularityp} is valid.
For short, write
$$
t_n := C(1+\sigma)^{-2} \log(n)^{-C_+},
$$
for a constant $C_+$ that will be chosen large enough.
We will show that, for any $0<t \leq t_n$, $x \in A_0^{\varepsilon_n} \cap A_{\infty}^{\varepsilon_n}$ and $\|z\| \leq 4 \log(n)$, $xe^{-t} + \sigma_t z \in A_0^{\varepsilon^2_n} \cap A_\infty^{\varepsilon^2_n} $, for $n$ large enough.

A standard triangle inequality yields that
\begin{align*}
    \|x - e^{-t}x - \sigma_t z\| & \leq (1-e^{-t}) \|x\| + \sigma_t \|z\| \\
    & \leq C \log(n) t + C \sigma \sqrt{t} \log n \\
    & \leq C(1+\sigma) \sqrt{t} \log n ,
\end{align*}
for $t \leq 1$ and $n$ large enough. According to Properties~\hyperref[{assump:scoreregu}]{$(\mathcal{P}_\beta)$}, 
\begin{align*}
    C(1+\sigma) \sqrt{t} \log n \leq C^{-1} \log(\varepsilon^{-1}_n)^{-C} = C \log(n)^{-C}
\end{align*}
is thus enough to guarantee that $x - e^{-t}x - \sigma_t z \in A_0^{\varepsilon^2_n} \cap A_\infty^{\varepsilon^2_n}$. Choosing $C_+$ large enough in $t_n = C(1+\sigma)^{-2} \log(n)^{-C_+}$ ensures that the above inequality is satisfied, for any $t \leq t_n$, and $n$ large enough.

We are now in position to control $\gamma(s^\star,x)$, for $x \in A_0^{\varepsilon_n} \cap A_\infty^{\varepsilon_n}$. 
According to Proposition~\ref{prop:estimatenormscore} and Theorem~\ref{coro:regularityp}, it holds
\begin{itemize}
    \item if $x \in A_0^\varepsilon \cap A_\infty^\varepsilon$, $\|s^\star(t,x)\| \leq C \log(\varepsilon^{-1})^C (1+ t^{-1/2} + {\sigma^{-2}}) $;
    \item if $x \notin A_0^\varepsilon \cap A_\infty^\varepsilon$, $\|s^\star(t,x)\| \leq C \|x\|(1+t^{-1})$. 
\end{itemize}
For short we denote by $A_n = A_0^{\varepsilon_n} \cap A_\infty^{\varepsilon_n}$, and $A'_n=A_0^{\varepsilon^2_n} \cap A_\infty^{\varepsilon^2_n}$, for $\varepsilon_n$ defined above. Note that, according to Theorem~\ref{coro:regularityp}, $A_n \subset \B(0,C \log(n)^C)$.

If $n^{-1} \leq t\leq t_n$,  then
\begin{align*}
     & \E_Z (\|s^\star(t,e^{- t}x + \sigma_t Z\|^2) \\
     & \quad  \leq C \log(n)^C (t^{-1}+{ \sigma^{-4}}) \E_Z (  \1_{e^{- t}x + \sigma_t Z \in A'_n}) + Ct^{-2} \E_Z(\|e^{-t} x + \sigma_t Z\|^2 \1_{e^{- t}x + \sigma_t Z \notin A'_n}) \\
    & \quad \leq C \log(n)^C (t^{-1}+{ \sigma^{-4}}) + C t^{-2} \log(n)^C \P(\|Z\| > 4 \log(n)) + Ct^{-2} \sigma^2  \E_{Z}(\|Z\|^2 \1_{\|Z\| \geq 4 \log(n)}) \\
    & \quad \leq C (1 + \sigma^2+{ \sigma^{-4}})(\log(n)^C t^{-1} + n^{-2} t^{-2}) \\
    &  \quad \leq C (1 + \sigma^2+{ \sigma^{-4}}) \log(n)^C t^{-1},
\end{align*}
using Cauchy-Schwarz inequality and $tn \geq 1$ for the last lines.

If $t \geq t_n = c \log(n)^{-c}(1+\sigma)^{-2}$, then
\begin{align*}
    \E_Z (\|s^\star(t,e^{- t}x + \sigma_t Z\|^2) & \leq C(1+t^{-2}) \E_Z (\|e^{-t} x + \sigma_t Z\|^2) \\
    & \leq C(1+\sigma)^{2} \log(n)^C (1+t^{-2}) \\
    & \leq C(1+\sigma)^{4} \log(n)^C (1+t^{-1}).
\end{align*}

Thus, 
\begin{align*}
    \gamma(s^\star,x) & \leq 2 \int_{\underline{\tau}}^{\overline{\tau}} \E_Z (\|s^\star(t,e^{- t}x + \sigma_t Z\|^2) \mathrm{d}t + 2d \int_{\underline{\tau}}^{\overline{\tau}} \sigma_t^{-2} \mathrm{d}t \\
    & \leq C(1+\sigma+{ \sigma^{-1}})^4 \log(n)^C (\overline{\tau} - \underline{\tau})(1+ \underline{\tau}^{-1}) + 2d (\overline{\tau} - \underline{\tau}) \sigma_{\underline{\tau}}^{-2}. 
\end{align*}
It immediatly follows that
\begin{align*}
    |\nu(s,x)| \leq C(\overline{\tau}-\underline{\tau})\log(n)^C g(\underline{\tau},\sigma,V),
\end{align*}
with $g(\underline{\tau},\sigma,V) = 1 + V^2 + (1+\sigma+{ \sigma^{-1}})^4 (1+\underline{\tau}^{-1}) + \sigma_{\underline{\tau}}^{-2}$.

Now we turn to the bound on $\E_{X \sim p^\star} \nu^2(s,X)$. For all $x \in \R^d$ and $z \in \R^d$, the above calculation yields that
\begin{align*}
    \|s^\star(t,e^{-t}x + \sigma_t z)\|^4 \leq C\left ( \log(n)^C(1+t^{-2} + { \sigma^{-8}}) + (\|x\|^4 + \sigma^4\|z\|^4)(1+t^{-4}) \right ).
\end{align*}
Thus, for $t \geq n^{-1}$,
\begin{align*}
  & \E_{X \sim p^\star} \E_{Z} \|s^\star(t,e^{-t}X + \sigma_t Z)\|^4 \\  
  &   = \E_{X \sim p^\star} \E_{Z} \left [ \|s^\star(t,e^{-t}X + \sigma_t Z)\|^4 \1_{A_n}(X) \right ] + \E_{X \sim p^\star} \E_{Z} \left [ \|s^\star(t,e^{-t}X + \sigma_t Z)\|^4 \1_{A_n^c}(X) \right ] \\  
    & \leq C(1+\sigma+{ \sigma^{-1}})^8 \log(n)^C (1+t^{-2}) \\
    & \quad + C \log(n)^C(1+t^{-2}+{ \sigma^{-8}})n^{-4} + C (1+t^{-4}) \E_{X \sim p^\star} (\|X\|^4 \1_{A_n^c}(X)) + C \sigma^4 (1+t^{-4}) n^{-4} \\
    & \leq C(1+\sigma+\sigma^{-1})^8\log(n)^C(1+t^{-2}),
\end{align*}
using that $p^\star$ is $K$-subGaussian, Cauchy-Schwarz inequality and $tn \geq 1$.

We deduce that
\begin{align*}
    \E_{X \sim p^\star} \gamma^2(s^\star,X) & = \E_{X \sim p^\star} \left ( \int_{\underline{\tau}}^{\overline{\tau}} \E_Z \|s^\star(t,e^{-t} X + \sigma_t Z) - Z/\sigma_t \|^2 \mathrm{d}t \right )^2 \\
    & \leq (\overline{\tau}-\underline{\tau}) \int_{\underline{\tau}}^{\overline{\tau}} \E_{X\sim p^\star,Z} \|s^\star(t,e^{-t} X + \sigma_t Z) - Z/\sigma_t \|^4 \mathrm{d}t \\
    & \leq C(\overline{\tau}-\underline{\tau})^2 \sigma_{\underline{\tau}}^{-4} + C(\overline{\tau}-\underline{\tau})^2 C(1+\sigma+{ \sigma^{-1}})^8\log(n)^C(1+\underline{\tau}^{-2}).
\end{align*}
At last, 
\begin{align*}
    \E_{X \sim p^\star} \nu^2(s,X) & \leq 2 (\E_{X\sim p^\star} \gamma^2(s,X) + \E_{X \sim p^\star}\gamma^2(s^\star,X)) \\
    & \leq C(\overline{\tau}-\underline{\tau})^2 \log(n)^C g^2(\underline{\tau},\sigma,V).
\end{align*}

Let us turn to the last inequality of Lemma~\ref{lem:supnorm_var_contrast}.
For all $x,t,z$, it holds 
\begin{align*}
   &  \left | \|s(t,e^{ t} x + \sigma_t z) - z/\sigma_t \|^2 - \|s^\star(t,e^{- t} x + \sigma_t z) - z/\sigma_t \|^2\right | \\
   & \quad \leq \|(s-s^\star)((t,e^{- t} x + \sigma_t z)\|\left ( \|s(t,e^{- t} x + \sigma_t z) - z/\sigma_t \| + \|s^\star(t,e^{- t} x + \sigma_t z) - z/\sigma_t \| \right ). \\
\end{align*}
Using Cauchy-Schwaz inequality entails that 
\begin{align*}
    | \nu(s,x) | & = \left | \E_Z \int_{\underline{\tau}}^{\overline{\tau}} \| s(t,e^{- t} x + \sigma_t Z) - Z/\sigma_t \|^2 \mathrm{d}t - \E_Z \int_{\underline{\tau}}^{\overline{\tau}} \| s^\star(t,e^{- t} x + \sigma_t Z) - Z/\sigma_t \|^2 \mathrm{d}t \right | \\
    & \leq \E_Z \int_{\underline{\tau}}^{\overline{\tau}} \|(s-s^\star)((t,e^{- t} x + \sigma_t Z)\|\left ( \|s(t,e^{- t} x + \sigma_t Z) - Z/\sigma_t \| + \|s^\star(t,e^{- t} x + \sigma_t Z) - Z/\sigma_t \|  \right ) \mathrm{d}t \\
    & \leq \sqrt{\E_Z \int_{\underline{\tau}}^{\overline{\tau}} \|(s-s^\star)((t,e^{- t} x + \sigma_t Z)\|^2 \mathrm{d}t} \\
    & \quad \times \sqrt{2 \E_Z \int_{\underline{\tau}}^{\overline{\tau}}  (\|s(t,e^{- t} x + \sigma_t Z) - Z/\sigma_t \|^2 + \|s^\star(t,e^{- t} x + \sigma_t Z) - Z/\sigma_t \|^2) \mathrm{d}t} \\
    & \leq \sqrt{\E_Z \int_{\underline{\tau}}^{\overline{\tau}} \|(s-s^\star)((t,e^{- t} x + \sigma_t Z)\|^2 \mathrm{d}t} \sqrt{2(\gamma(s,x) + \gamma(s^\star,x))}.
\end{align*}
Let us decompose $\E_{X \sim p^\star} \nu_c(s,X)^2$ as follows.
\begin{align*}
     \E_{X \sim p^\star} \nu_c(s,X)^2 &  =  \E_{X \sim p^\star} (\nu_c(s,X)^2 \1_{A_n}(X)) + \E_{X\sim p^\star} ( \nu_c(s,X)^2 \1_{A_n^c}(X)).
\end{align*}
Since $\nu_c(s,X) \1_{\B(0,A_n)}(X) = \nu(s,X) \1_{\B(0,A_n)}(X)$, the above calculation yields the following bound for the first term:
\begin{align*}
    \E_{X \sim p^\star} (\nu_c(s,X)^2 \1_{A_n}(X))  & \leq  c_n \E_{X \sim p^\star} \E_Z \int_{\underline{\tau}}^{\overline{\tau}} \|(s-s^\star)((t,e^{- t} X + \sigma_t Z)\|^2 \mathrm{d}t \\
    & \quad =  c_n \E_{X \sim p^\star} (\nu(s,X)),
\end{align*}
for $\left | 2(\gamma(s,x) + \gamma(s^\star,x))\1_{\B(0,c_0K\log(n)}(x) \right | \leq c_n = C(\overline{\tau}-\underline{\tau})\log(n)^C g(\underline{\tau},\sigma,V) $, using the bounds on $\gamma(s,x)$ and $\gamma(s^\star,x)$ provided above and Theorem~\ref{thm:Vincent} for the last line. In turn, the second term may be bounded by
\begin{align*}
    \E_{X\sim p^\star} ( \nu_c(s,X)^2 \1_{A_n^c}(X)) 
    &  \leq c_n^2 \P_{X \sim p^\star}(A_n^c) \\
    & \leq c_n^2/n^4.
\end{align*}

\end{proof}

\subsection{Proof of Proposition~\ref{prop:aproxsuzuki}}\label{sec:coro:prop:aproxsuzuki}
\subsubsection{Preliminaries and plan for the proof}\label{sec:planoftheproofapprox}
Let $C_1^\star>0$ and define for all $k\in \{0,1,\ldots,m-1\}$ the set
\begin{equation}\label{eq:form-of-Bk}
    \mathcal{B}_k:= \{x\in \mathbb{R}^d|\ \mathrm{dist}(x,A_{\tau_k}^{1/n})\leq \log(n)^{-C_1^\star}\}.
\end{equation}
By Property~\hyperref[{assump:scoreregu}]{$(\mathcal{P}_\beta)$}.2, taking $C_1^\star>0$ large enough we have that \begin{equation}\label{eq:prop-of-Bk}
\mathcal{B}_k\subset A_{\tau_k}^{1/n^2}.
\end{equation}

In order to prove Proposition~\ref{prop:aproxsuzuki}, let us define a smooth approximation of the projection on the subset $\mathcal{B}_k$. This is possible as the boundary of $\mathcal{B}_k$ is $C^\infty$ as the set $A_t^\varepsilon$ is convex.
Let $\pi_k\in C^\infty(\mathbb{R}^d,\mathbb{R}^d)$ such that for all $x\in A_{\tau_k}^{1/n}$ we have $\pi_k(x)=x$ and for all $y\notin \mathcal{B}_k$ we have $\pi_k(y)=\pi_{\mathcal{B}_k}(y)$ the orthogonal projection on $\mathcal{B}_k$. In particular we can choose $\pi_k$ such that for all $i\in \mathbb{N}_{>0}$, we have
\begin{equation}
    \|\nabla^i \pi_k(x)\|\leq C_i \log(n)^{-C_\star^1(i-1)}.
\end{equation}
To avoid the bad behavior of $s^\star(\tau_k,\cdot)$ outside of $\mathcal{B}_k$, we are going to approximate the function
\begin{equation}\label{eq:sstark}
x\mapsto s^\star(\tau_k,\pi_k(x))+\pi_k(x)
\end{equation}
on the set
\begin{equation}\label{eq:lagrandeboule}
    \mathcal{B}:= \{x\in \mathbb{R}^d|\ \mathrm{dist}(x,\mathcal{B}_k)\leq \log(n)^{C_2^\star}\},
\end{equation}
with $C_2^\star>0$ to be fixed later such that for all $t > 0$ we have $A_t^{1/n^2}\subset \mathcal{B}.$ For lightness of notation, we define
\begin{equation}\label{eq:splusid}
    S^\star(t,x) := s^\star(t,x)+\frac{1}{\sigma^2}x,
\end{equation}
which is the function satisfying the regularity properties of Section~\ref{sec:model-properties}.

Let us  fix $k\in \{0,1,\ldots,m-1\}$ and take $t \in  [\tau_k,\tau_{k+1}].$ For $l\geq 0$, we define $P^{\lfloor l\rfloor}$ the $\lfloor l\rfloor$-th Taylor approximation which given an input $(a_q)_{q=0}^{\lfloor l\rfloor}$, outputs the polynomial $$P^{\lfloor l\rfloor}\big((a_q)_{q=0}^{\lfloor l\rfloor}\big)(t) = \sum_{q=0}^{\lfloor l\rfloor} a_q\frac{(t-\tau_k)^q}{q!}.$$ 
In particular, for a function $f\in \mathcal{H}^{l}_B$ we have
$$|f(t)-P^{\lfloor l\rfloor}\big((f^{(q)}(\tau_k))_{q=0}^{\lfloor l\rfloor}\big)(t)|\leq B|t-\tau_k|^{l}.$$

Using these Taylor approximations, we are going to build a polynomial function $s_N:[\tau_{k,j},\tau_{k,j+1}]\times \mathbb{R}^d\rightarrow \mathbb{R}^d$ such that
 $$\sup_{t\in [\tau_{k,j},\tau_{k,j+1}]}\sup_{x\in A_{\tau_k}^{1/n}}\|S^\star(t,x)-s_N(t,x)\|\leq \frac{C}{\tau_k}\log(n)^{C_2}n^{-\frac{\beta+1}{2\beta+d}},$$
 and  then show that there exists a neural network $\bar{s} \in \Psi(L_k, W_k,B_k,V_k,V_k^{'})_d^{d+1}$ such that
  $$
  \sup_{t\in [\tau_{k,j},\tau_{k,j+1}]}\sup_{x\in A_{\tau_k}^{1/n}}\|s_N(t,x)-\bar{s}(t,x)\|\leq n^{-1}.
  $$
We treat separately the cases $\tau_k\leq n^{-\frac{2}{2\beta+d}}$ and $n^{-\frac{2}{2\beta+d}}< \tau_k$ to exploit the increasing regularity of $s^\star(t,\cdot)$  as $t$ becomes larger.

\subsubsection{Proof of Proposition~\ref{prop:aproxsuzuki} in the case $t\leq n^{-\frac{2}{2\beta+d}}$}\label{sec:approxcasetpetit}
\paragraph{Approximation with polynomials of neural nets} Recalling the definition of $S^\star$ in~\eqref{eq:splusid} and using the regularity property~\eqref{eq:reguscorespace} we have that for all $q\in \mathbb{N}_{\geq 0}$
$$\|\partial_t^q S^\star(t,\cdot)\|_{\mathcal{H}^{(\beta+1-2q)\vee 1}(A_{t}^{1/n^2})}\leq Ce^{-t}\left(1+t^{-1-((q-\beta/2)\vee 0)}(1+\log(n)^{C_2})\right)$$
so from~\eqref{eq:prop-of-Bk} we deduce that
$$\|\partial_t^q S^\star(t,\pi_k(\cdot))\|_{\mathcal{H}^{(\beta+1-2q)\vee 1}(\mathcal{B})}\leq Ce^{-t}\left(1+t^{-1-((q-\beta/2)\vee 0)}(1+\log(n)^{C_2})\right).$$
Then, taking $\varepsilon:=\left(\tau_k^{-1/d}n^{\frac{d-2}{d(2\beta+d)}}\right)^{-1}$ and $q\in \{0,\ldots,\lfloor \frac{\beta+1}{2}\rfloor \}$, from Proposition~\ref{prop:approxrelu} we have that there exists a $\tanh$ neural network 
$$  
N_k^q\in \Psi\big(2,C(\log(n)^{C_2+C_2^\star}\varepsilon^{-1})^d ,n^{C_2},C\log(n)^{C_2}\tau_k^{-\frac{1}{2}-(q-\frac{\beta}{2})\vee 0},\infty\big)^d_d
$$ 
such  that  
\begin{equation}\label{eq:boundpartialtnn}
\sup_{x\in \mathcal{B}}\|\partial_t^q S^\star(\tau_k,\pi_k(x))-N_k^q(x)\|\leq Ce^{-t}\log(n)^{C_2}\tau_k^{-1-((q-\beta/2)\vee 0)} \varepsilon^{(\beta+1-2q)\vee 1} 
\end{equation}
and 
\begin{align}\label{eq:boundgradientnnx}
     \sup_{x\in \mathcal{B}}\|\nabla \partial_t^q S^\star(\tau_k,\pi_k(x))-\nabla N_k^q(x)\|
    & \leq Ce^{-t} \log(n)^{C_2}\tau_k^{-1-((q-\beta/2)\vee 0)} \varepsilon^{(\beta+1-2q)\vee 1-1}.
\end{align}
For $t\in [\tau_k,\tau_{k+1}]$ and $x\in \mathbb{R}^d$, we define the Taylor approximation of $S^\star(\cdot,\pi_k(x))$ with respect to time by
\begin{align}
    s_N(t,x) = P^{\lfloor \frac{\beta+1}{2}\rfloor}\big((N_k^q(x))_{q=0}^{\lfloor \frac{\beta+1}{2}\rfloor} \big)(t).
\end{align}
Then, for $x\in \mathcal{B}$ we have
\begin{align*}
        \|S^\star(t,\pi_k(x))-s_N(t,x)\|  & \leq \|S^\star(t,\pi_k(x))-P^{\lfloor \frac{\beta+1}{2}\rfloor}\big(\partial_t^q S^\star(\tau_k,\pi_k(x)))_{q=0}^{\lfloor \frac{\beta+1}{2}\rfloor }\big)(t)\|\\
        &
        \hspace{3em} 
        +
        \|P^{\lfloor \frac{\beta+1}{2}\rfloor}\big(\partial_t^q S^\star(\tau_k,\pi_k(x)))_{q=0}^{\lfloor \frac{\beta+1}{2}\rfloor }\big)(t)-s_N(t,x)\|.
\end{align*}
Using the time-regularity property~\eqref{eq:reguscoretime} we have
\begin{align*}
    \|S^\star(\cdot,\pi_k(x))\|_{ \mathcal{H}^{(\beta+1)/2}([t,\infty))}\leq C\left(1+t^{-1}(1+\log(n)^{C_2})\right),
\end{align*}
which gives for $x\in \mathcal{B}$,
\begin{align*}
    \|S^\star(t,\pi_k(x))-P^{\lfloor \frac{\beta+1}{2}\rfloor}\big((\partial_t^q S^\star(\tau_k,\pi_k(x)))_{q=0}^{\lfloor \frac{\beta+1}{2}\rfloor }\big)(t)\|
    & \leq C\log(n)^{C_2} \tau_k^{-1}(t-\tau_k)^{(\beta+1)/2}\\
    & \leq C\log(n)^{C_2} \tau_k^{-1} n^{-\frac{\beta+1}{2\beta+d}}.
\end{align*}
On the other hand, using~\eqref{eq:boundpartialtnn} we have
\begin{align*}
    &\|P^{\lfloor \frac{\beta+1}{2}\rfloor}\big(\partial_t^q S^\star(\tau_k,\pi_k(x)))_{q=0}^{\lfloor \frac{\beta+1}{2}\rfloor }\big)(t)-s_N(t,x)\|\\
    & \leq C \|\sum_{q=0}^{\lfloor \frac{\beta+1}{2}\rfloor} (\partial_t^q S^\star(\tau_k,\pi_k(x))-N_k^q(x))(t-\tau_k)^q\|\\
    &\leq C\log(n)^{C_2} \sum_{q=0}^{\lfloor \frac{\beta+1}{2}\rfloor}\tau_k^{-1-((q-\beta/2)\vee 0)}\left(\tau_k^{1/d}n^{-\frac{d-2}{d(2\beta+d)}}\right)^{(\beta+1-2q)\vee 1}(\tau_{k+1}-\tau_k)^{q}.
\end{align*}
Now, if $q\leq \beta/2$, recalling that $\tau_k^{1/2}\leq \tau_k^{1/d}n^{-\frac{d-2}{d(2\beta+d)}}\leq n^{-\frac{1}{2\beta+d}}$  we have
\begin{align*}
\tau_k^{-1}\left(\tau_k^{1/d}n^{-\frac{d-2}{d(2\beta+d)}}\right)^{\beta+1-2q}(\tau_{k+1}-\tau_k)^{q}&\leq \tau_k^{-1}\left(\tau_k^{1/d}n^{-\frac{d-2}{d(2\beta+d)}}\right)^{\beta+1}\left(\tau_k^{1/d}n^{-\frac{d-2}{d(2\beta+d)}}\right)^{-2q}(2\tau_k)^{q}\\
    &\leq C \tau_k^{-1} n^{-\frac{\beta+1}{2\beta+d}}.
\end{align*}
On the other hand, if $q> \beta/2$ we have
\begin{align*}
\tau_k^{-1-q+\beta/2}\tau_k^{1/d}n^{-\frac{d-2}{d(2\beta+d)}}(\tau_{k+1}-\tau_k)^{q}& \leq \tau_k^{-1+\beta/2}\tau_k^{1/d}n^{-\frac{d-2}{d(2\beta+d)}}\tau_k^{-q}(2\tau_k)^{q}\\
    &\leq C \tau_k^{-1} n^{-\frac{\beta+1}{2\beta+d}}.
\end{align*}
We can then conclude that for $x\in \mathcal{B}$,
\begin{align}\label{align:bckjdkshsksjdhdhd}
    \|S^\star(t,\pi_k(x))-s_N(t,x)\|  & \leq C\log(n)^{C_2} \tau_k^{-1} n^{-\frac{\beta+1}{2\beta+d}}.
\end{align}
\paragraph{Construction of the final neural net on $\mathcal{B}_k$}
Let us now show that there exists a neural network $\bar{s} \in \Psi(L_k, W_k,B_k,V_k,V_k^{'})^{d+1}_d$ such that
  $$
  \sup_{t\in [\tau_{k,j},\tau_{k,j+1}]}\sup_{x\in \mathcal{B}}\|s_N(t,x)-\bar{s}(t,x)\|\leq n^{-1}.
  $$
We know from Corollary
3.7 in~\cite{de2021approximation} that there exists a $\tanh$ neural network $P_N:\mathbb{R}^{d\lfloor \frac{\beta+1}{2}\rfloor+1}\rightarrow\mathbb{R}$ with depth  $O(1)$, width $O(1)$ and weights bounded by $O(n^{C_2})$ such that for all $i\in \{0,1\}$,
\begin{equation}\label{eq:nnaproxpoly}
\sup_{y\in [-n,n]^{d\times\lfloor \frac{\beta+1}{2}\rfloor},t \in [0,n]}\|\nabla^i P_N(y,t)-\nabla^i P^{\lfloor \frac{\beta+1}{2}\rfloor}\big((y_k)_{q=0}^{\lfloor \frac{\beta+1}{2}\rfloor} \big)(t)\|\leq n^{-1}.
\end{equation}
Using this approximation of polynomials, we define the $\tanh$ neural network $\bar{s}_1:\mathbb{R}^{d+1}$ by
\begin{equation}\label{eq:nnaproxpoly2}
\bar{s}_1(t,x):= P_N\Big(\big(N_k^q(x)\big)_{q=0}^{\lfloor \frac{\beta+1}{2}\rfloor},t\Big).
\end{equation}
In particular, 
the width of $\bar{s}_1$ is bounded by $O\left(\log(n)^{C_2}\left( n^{-\frac{d-2}{d(2\beta+d)}}\right)^d\right)$, depth by $O(1)$ and weights bounded by $O(n^{C_2})$, so $\bar{s}_1$
belongs to $\Psi(L_k, W_k,B_k, \infty,\infty)^{d+1}_d$ and satisfies from~\eqref{eq:nnaproxpoly} that
\begin{equation}\label{eq:ghudosjsbfhydjskzs}
  \sup_{t\in [\tau_{k,j},\tau_{k,j+1}]}\sup_{x\in \mathcal{B}}\|s_N(t,x)-\bar{s}_1(t,x)\|\leq n^{-1}.
\end{equation}
  
\paragraph{Construction of the final neural net on $\mathbb{R}^d$} Now, in order to have a neural net with finite $\lambda_{\max}$ on the whole $\mathbb{R}^d$ space, let us define the cut-off function
\begin{equation}\label{eq:deltaghjklh}
    \delta(x) := \left\{\begin{array}{ll} 1 & \text{if } x\in \mathcal{B}_k\\
     \mathrm{dist}(x,\mathcal{B}^c)\Big(\mathrm{dist}(x,\mathcal{B}_k)+\mathrm{dist}(x,\mathcal{B}^c)\Big)^{-1}& \text{if } x\in \mathcal{B}\backslash \mathcal{B}_k\\
     0 & \text{if } x\notin \mathcal{B}.
    \end{array}\right.
\end{equation}
Let \(\bar{\delta}\) and \(\bar{\times}\) denote the tanh neural network approximations of \(\delta\) and the multiplication operation respectively, both constructed to approximate their targets with precision \(n^{-1}\) in the \(\mathcal{H}^1\)-norm. Then, the final neural network \(\bar{s}\), defined on \([\tau_k, \tau_{k+1}] \times \mathbb{R}^d\), is given by:
\begin{equation}\label{eq:finalneuralnet}
    \bar{s}(t, x) := 
    \bar{\times}(\bar{\delta}(x), \bar{s}_1(t, x)).
\end{equation}

In particular we have that $\bar{s}$
belongs to $\Psi(L_k, W_k,B_k, \infty,\infty)^{d+1}_d$ and from~\eqref{align:bckjdkshsksjdhdhd} and~\eqref{eq:ghudosjsbfhydjskzs},
\begin{equation*}
  \sup_{t\in [\tau_{k,j},\tau_{k,j+1}]}\sup_{x\in A_{\tau_k}^{1/n}}\|S^\star(t,x)-\bar{s}(t,x)\|\leq C\log(n)^{C_2} \tau_k^{-1} n^{-\frac{\beta+1}{2\beta+d}}.
\end{equation*}
\paragraph{Bound on the sup-norm}
By definition, for $x\notin \mathcal{B}$, we have $\|\bar{s}(t,x)\|\leq n^{-1}$. Now, for  $x\in \mathcal{B}$ we have
\begin{align*}
    \|\bar{s}(t,x)\| &\leq \|\bar{s}_1(t,x)\|  \\
    &\leq \|\bar{s}_1(t,x)-P^{\lfloor \frac{\beta+1}{2}\rfloor} ((N_k^q(\cdot))_{q=0}^{\lfloor \frac{\beta+1}{2}\rfloor},t))\| +\|P^{\lfloor \frac{\beta+1}{2}\rfloor} ((N_k^q(\cdot))_{q=0}^{\lfloor \frac{\beta+1}{2}\rfloor},t))\| \\
    & \leq n^{-1}+C\sum_{q=0}^{\lfloor \frac{\beta+1}{2}\rfloor} \|N_k^q(x) (t-\tau_k)^q\|\\
    & \leq n^{-1}+\sum_{q=0}^{\lfloor \frac{\beta+1}{2}\rfloor} C\log(n)^{C_2}\tau_k^{-\frac{1}{2}-(q-\frac{\beta}{2})\vee 0}\tau_k^q\\
    &\leq C\log(n)^{C_2}\tau_k^{-1/2}
    \\
    &\leq V_k.
\end{align*}

\paragraph{Bound on $\lambda_{\max}$ for $x\in \mathcal{B}_k$}
For $x\in \mathcal{B}_k$ we have
\begin{align*}
    \lambda_{\max}(\nabla \bar{s}(t,x)) &\leq \lambda_{\max}(\nabla \bar{s}_1(t,x)-\nabla s_N(t,x)))+\lambda_{\max}(\nabla s_N(t,x)-\nabla S^\star(t,\pi_k(\cdot))(x))
    \\
    &\hspace{1em}
    +\lambda_{\max}(\nabla S^\star(t,\pi_k(\cdot))(x))
    \\
    & \leq  Cn^{-1}\sum_{q=0}^{\lfloor \frac{\beta+1}{2}\rfloor} \|\nabla N_k^q(x) \|+\lambda_{\max}(\nabla s_N(t,x)-\nabla S^\star(t,\pi_k(\cdot))(x))+Ce^{- t}t^{-1+\frac{\beta\wedge 1}{2}}.
\end{align*} 
Now, for any $x\in \mathcal{B}$, using~\eqref{eq:boundgradientnnx} we get
\begin{align*}
    n^{-1}\sum_{q=0}^{\lfloor \frac{\beta+1}{2}\rfloor} \|\nabla N_k^q(x)\|& \leq n^{-1}\sum_{q=0}^{\lfloor \frac{\beta+1}{2}\rfloor}\|\nabla \partial_t^q S^\star(\tau_k,\pi_k(x))-\nabla N_k^q(x)\|+\|\nabla \partial_t^q S^\star(\tau_k,\pi_k(x))\|\nonumber\\
     & \leq C \log(n)^{C_2}n^{-1}\sum_{q=0}^{\lfloor \frac{\beta+1}{2}\rfloor}\tau_k^{-1-((q-\beta/2)\vee 0)} \varepsilon^{(\beta+1-2q)\vee 1-1} +\tau_k^{-\frac{1}{2}-(q+\frac{1-\beta}{2})\vee 0}\\
    &\leq 
    C \log(n)^{C_2}n^{-1}\sum_{q=0}^{\lfloor \frac{\beta+1}{2}\rfloor}(\tau_k^{-1}\varepsilon^{\beta-2q}+\tau_k^{-\frac{1}{2}-(q+\frac{1-\beta}{2})\vee 0})\mathds{1}_{2q < \beta}+\tau_k^{-1-(q-\beta/2)}\mathds{1}_{2q \geq \beta}\\
    &\leq
    n^{-1}C \log(n)^{C_2} \sum_{q=0}^{\lfloor \frac{\beta+1}{2}\rfloor}(\tau_k^{-1/2}+\tau_k^{-1}(\varepsilon^{\beta-2q}+\tau_k^{\beta/2-q}))\mathds{1}_{2q < \beta}+\tau_k^{-1-(q-\beta/2)}\mathds{1}_{2q \geq \beta}\\
    &\leq
    C\tau_k^{-1/2}
    \\
    &\leq
    V_k^{'}
\end{align*}
and 
\begin{align*}
\lambda_{\max}(\nabla s_N(t,x)-\nabla S^\star(t,\pi_k(\cdot))(x))
&\leq \|\nabla S^\star(t,\pi_k(\cdot))(x)-\nabla P^{\lfloor \frac{\beta+1}{2}\rfloor}\big(\partial_t^q S^\star(\tau_k,\pi_k(\cdot)))_{q=0}^{\lfloor \frac{\beta+1}{2}\rfloor }\big)(t)(x)\|\\
&\hspace{2em} + \|\nabla P^{\lfloor \frac{\beta+1}{2}\rfloor}\big(\partial_t^q S^\star(\tau_k,\pi_k(\cdot)))_{q=0}^{\lfloor \frac{\beta+1}{2}\rfloor }\big)(t)(x)-\nabla s_N(t,x)\|
\\
& \leq  \|\nabla S^\star(t,\pi_k(\cdot))(x)-\sum_{q=0}^{\lfloor \frac{\beta+1}{2}\rfloor}\nabla \partial_t^q S^\star(\tau_k,\pi_k(\cdot))(x)\frac{(t-\tau_k)^q}{q!}\|
\\
&\hspace{2em} 
+ C\sum_{q=0}^{\lfloor \frac{\beta+1}{2}\rfloor}(t-\tau_k)^q \|\nabla \partial_t^q S^\star(\tau_k,\pi_k(\cdot))(x)-\nabla N_k^q(x)\|
    \\
    & \leq \|\nabla \partial_t^{\lfloor \frac{\beta+1}{2}\rfloor} S^\star(\cdot,\pi_k(x))\circ \nabla \pi_k(x)\|_{\mathcal{H}^{\frac{\beta+1}{2}-\lfloor \frac{\beta+1}{2}\rfloor}([\tau_k,\infty))}\tau_k^{\frac{\beta+1}{2}} \\
    &\hspace{2em}
    +C\log(n)^{C_2}\sum_{q=0}^{\lfloor \frac{\beta+1}{2}\rfloor}\tau_k^q \big(\tau_k^{-1-((q-\beta/2)\vee 0)} \varepsilon^{(\beta+1-2q)\vee 1-1}\big)
    \\
    & \leq C\log(n)^{C_2}\left(\tau_k^{\frac{\beta+1}{2}}\tau_k^{-1-(( \frac{\beta+1}{2}-\beta/2)\vee 0)}+\tau_k^{-1+\beta/2}\right)\\
    & \leq V_k^{'}.
\end{align*}

\paragraph{Bound on $\lambda_{\max}$ for $x\notin \mathcal{B}_k$} By construction, we have $\|\nabla \bar{s}(t,x)\|\leq n^{-1}$ for all 
    $x\notin \mathcal{B}$. Now, for $x\in \mathcal{B} \setminus \mathcal{B}_k $ we have 
\begin{align*}
    \lambda_{\max}(\nabla \bar{s}(t,x))
    &\leq \lambda_{\max}\big(\nabla \bar{s}(t,x)-\nabla (y\mapsto \bar{\delta}(y) \bar{s}_1(t, y))(x)\big)+\lambda_{\max}\big(\nabla (y\mapsto \bar{\delta}(y) \bar{s}_1(t, y))(x)\big) \\
    &\leq n^{-1}\big(\|\nabla \bar{s}_1(x)\|+\|\nabla\bar{\delta}(x)\|\big)+ \bar{\delta}(x)\lambda_{\max}(\nabla \bar{s}_1(x))+ \|\bar{s}_1(x)\|\|\nabla \bar{\delta}(x)\|.
\end{align*}
On one hand, from the precedent paragraph and property~\eqref{eq:reguscorespace} we have 
\begin{align*}
    n^{-1}\|\nabla \bar{s}_1(x)\|&\leq n^{-1}\big(\|\nabla \bar{s}_1(x)-\nabla S^\star(t,  \pi_k(\cdot))(x) \| +\|\nabla S^\star(t,  \pi_k(\cdot))(x)\|\big)\\
    &\leq n^{-1}\big(V_k^{'}+C\log(n)^{C_2}e^{-t}t^{-\big(1-\frac{\beta}{2})\vee 0\big)} \big)\\
    &\leq V_k^{'}.
\end{align*}
On the other hand, recalling the definitions of $\delta$ and $\mathcal{B}$ in~\eqref{eq:deltaghjklh} and~\eqref{eq:lagrandeboule} respectively, we deduce that
\begin{align*}
   n^{-1}\|\nabla\bar{\delta}(x)\| \leq C n^{-1}\log(n)^{-C_2^\star}\leq V_k^{'}.
\end{align*}
Finally, we have
\begin{align*}
\bar{\delta}(x)\lambda_{\max}(\nabla \bar{s}_1(x))+ \|\bar{s}_1(x)\|\|\nabla \bar{\delta}(x)\|
 &\leq  \lambda_{\max}\big(\nabla \bar{s}_1-\nabla  S^\star(t,  \pi_k(\cdot))(x)\big)+\lambda_{\max}\big(\nabla S^\star(t,  \pi_k(\cdot))(x)\big)
 \\
 &
 \hspace{2em} + Ce^{-t}\log(n)^{C_2}t^{-1/2}\|\nabla \bar{\delta}(x)\|
 \\
   & \leq  V_k^{'}+Ce^{-t}\log(n)^{C_2}t^{-1/2}\log(n)^{-C_2^\star}\\
   & \leq   V_k^{'},
\end{align*}
taking $C_2^\star>0$ large enough in~\eqref{eq:lagrandeboule}.

\subsubsection{Proof of Proposition~\ref{prop:aproxsuzuki} in the case \( t > n^{-\frac{2}{2\beta + d}} \)}  
The main difference from the case \( t \leq n^{-\frac{2}{2\beta + d}} \) lies in the polynomial approximation of neural networks (first paragraph of Section~\ref{sec:approxcasetpetit}). We therefore focus on detailing this aspect, as the remainder of the proof closely follows the arguments of Section~\ref{sec:approxcasetpetit}.

To obtain a more refined approximation of $S^\star$, this case is based on a similar construction as above, but requires an additional refinement of intervals $[\tau_k,\tau_{k+1}]$ into $\Upsilon_k$ sub-intervals, as described in Section~\ref{sec:scoreapp}. Furthermore, we use the additional regularity of $S^\star(t,\cdot)$ gained from the fact that $t$ is large. 
Let us set $\varepsilon := \left(\tau_k^{-1/d}n^{\frac{d-2}{d(2\beta+d)}}\right)^{-1}$, $\gamma :=( 2\beta+d)/(d-2)$ and $\gamma_2 := (\beta+1)(d-2)$. Recalling that 
$
\tau_{k,j}
:=
\left(1 + j/\Upsilon_k \right) \tau_k
$, let $j\in \{0,\ldots,\Upsilon_k \}$ and $q\in \{0,\ldots,\lfloor \frac{\beta+1}{2} +\gamma_2\rfloor\}$.
We let
$$
N_{k,j}^q\in  \Psi\big(2,C(\log(n)^{C_2+C_2^\star}\varepsilon^{-1})^d ,n^{C_2},C\log(n)^{C_2}\tau_k^{-\frac{1}{2}-(q-\frac{\beta}{2})\vee 0},\infty\big)^d_d
$$ 
be a $\tanh$ neural network given by Proposition~\ref{prop:approxrelu}, such  that 
$$\sup_{x\in \mathcal{B}}\|\partial_t^q S^\star(\tau_{k,j},\pi_k(x))-N_{k,j}^q(x)\|\leq C\log(n)^{C_2} \tau_k^{-q-\frac{1+\gamma}{2}} \varepsilon^{\beta+\gamma} $$
and
$$\sup_{x\in \mathcal{B}}\|\nabla \partial_t^q S^\star(\tau_{k,j},\pi_k(x))-\nabla N_{k,j}^q(x)\|\leq C\log(n)^{C_2} \tau_k^{-q-\frac{1+\gamma}{2}} \varepsilon^{\beta+\gamma-1}.$$
For $t\in [\tau_{k,j},\tau_{k,j+1}]$, define
\begin{align*}
    s_N(t,x) 
    := 
    P^{\lfloor \frac{\beta+1}{2} +\gamma_2\rfloor}\big((N_{k,j}^q(x))_{q=0}^{\lfloor \frac{\beta+1}{2}+\gamma_2\rfloor} \big)(t)
    .
\end{align*}
Then, we have 
\begin{align*}
    \|S^\star(t,\pi_k(x))-s_N(t,x)\|  & \leq \|S^\star(t,\pi_k(x))-P^{\lfloor\frac{\beta+1}{2} +\gamma_2\rfloor}\big((\partial_t^q S^\star(\tau_{k,j},\pi_k(x)))_{q=0}^{\lfloor\frac{\beta+1}{2} +\gamma_2\rfloor}\big)(t)\| \\
    &\hspace{3em} 
    + \|P^{\lfloor\frac{\beta+1}{2} +\gamma_2\rfloor}\big((\partial_t^q S^\star(\tau_{k,j},\pi_k(x)))_{q=0}^{\lfloor\frac{\beta+1}{2} +\gamma_2\rfloor}\big)(t)-s_N(t,x)\|\\
    & \leq C\log(n)^{C_2} \tau_k^{-(1+\gamma_2)}(\tau_{k,j+1}-\tau_{k,j})^{(\beta+1)/2+\gamma_2}\\
    &\hspace{3em} 
    +C\|\sum_{q=0}^{\lfloor\frac{\beta+1}{2} +\gamma_2\rfloor} (\partial_t^q S^\star(\tau_{k,j},\pi_k(x))-N_{k,j}^q(x))(t-\tau_{k,j})^q\|.
\end{align*}
Now, for the first term we have
\begin{align*}
    \tau_k^{-(1+\gamma_2)}(\tau_{k,j+1}-\tau_{k,j})^{(\beta+1)/2+\gamma_2}& \leq C \tau_k^{-(1+\gamma_2)}(\tau_k(\tau_k^{-1}n^{\frac{-2}{2\beta+d}})^{\frac{d-2}{d}})^{(\beta+1)/2+\gamma_2}\\
    & \leq C \tau_k^{-(1+\gamma_2)+\frac{2}{d}((\beta+1)/2+\gamma_2)}n^{\frac{-(d-2)(\beta+1+2\gamma_2)}{d(2\beta+d)}}\\
    & \leq C\tau_k^{-1}n^{-\frac{\beta+1}{2\beta+d}}.
\end{align*}
On the other hand,
\begin{align*}
&
\|\sum_{q=0}^{\lfloor\frac{\beta+1}{2} +\gamma_2\rfloor} (\partial_t^q S^\star(\tau_{k,j},\pi_k(x))-N_{k,j}^q(x))(t-\tau_{k,j})^q\|
\\
&\hspace{3em}
\leq C\sum_{q=0}^{\lfloor\frac{\beta+1}{2} +\gamma_2\rfloor}\tau_k^{-q-\frac{1+\gamma}{2}}\left(\tau_k^{1/d}n^{-\frac{d-2}{d(2\beta+d)}}\right)^{\beta+\gamma}(\tau_k(\tau_k^{-1}n^{\frac{-2}{2\beta+d}})^{\frac{d-2}{d}})^{q}\\
&\hspace{3em}
\leq C\log(n)^{C_2}\sum_{q=0}^{\lfloor\frac{\beta+1}{2} +\gamma_2\rfloor} \tau_k^{-1}n^{-\frac{\beta+1}{2\beta+d}}\\
&\hspace{3em}
\leq
C\log(n)^{C_2}\tau_k^{-1}n^{-\frac{\beta+1}{2\beta+d}}.
\end{align*}

Finally, as in the case $t\leq n^{-\frac{2}{2\beta+d}}$, we have that $s_N$ can be approximated by a tanh neural network $\bar{s} \in  \Psi(L_k, W_k,B_k, V_k,V_k^{'})^{d+1}_d$ up to an error $n^{-1}$ on  $A_{\tau_k}^{1/n}$, which concludes the proof.

\section{Proof of Theorem~\ref{thm:borneinf}}
\label{sec:proof_of_proposition-borneinf}

We slightly adapt the proof of~\cite[Theorem~2]{Donoho96} and~\cite[Theorem~4]{Uppal19}, using standard bayesian-like lower bound arguments comparing perturbed distributions.

\begin{proof}[Proof of Theorem~\ref{thm:borneinf}]

As a base density, let us choose the standard gaussian 
\begin{align*}
    g_0(x) := (2 \pi)^{-d/2} \exp(-\|x\|^2/2).
\end{align*}
To perturb it, let $\psi:\R^d \to \R$ denote a kernel-like function such that
\begin{itemize}
    \item $\psi \in \mathcal{C}^{\infty}(\R^d,\R)$;
    \item $\mathrm{Supp}(\Psi) \subset (-1,1)^d$;
    \item $\int_{\R^d} \Psi(u) du = 0$, and $\int_{\R^d} \Psi^2(u) du = 1$.
\end{itemize}
Let us now fix a resolution $j \in \mathbb{N}^*$ to be chosen later. Set $m := 2^{dj}$.
Given a multi-index $\mathbf{k}$ of the grid $\{0,\ldots,2^j-1 \}^d \subset \R^d$, and $x \in \R^d$, we write
\begin{align*}
    \psi_{\mathbf{k}}(x) := \sqrt{m} \psi(2^j x -(2^j-1) \mathbbm{1}_d + 2 \mathbf{k}),
\end{align*}
where $\1_d:=(1, \hdots, 1)$. Functions $\psi_{\mathbf{k}}$ correspond to translations and scalings of the base perturbation~$\Psi$. They cover a grid of the hypercube $[-1,1]^d$ at scale $2^{-j}$, and are designed so as to have disjoint supports and be normalized, in the sense that $\int_{\R^d} \psi_{\mathbf{k}}^2 = 1$.

Given $\varepsilon \in \{-1,1 \}^{m}$, consider the function
$$
g_{\varepsilon} :=  g_0 + c_1 \sum_{\mathbf{k} \in \{0,\ldots,2^j-1\}^d} \varepsilon_{\mathbf{k}} \psi_{\mathbf{k}}
,
$$
for some constant $c_1>0$ to be chosen later. Its integral is trivially equal to one. 
Furthermore, the function
\begin{align*}
    \frac{g_{\varepsilon}}{g_0} - 1 = \frac{c_1}{g_0}\sum_{\mathbf{k} \in \{0,\ldots,2^j-1\}^d} \varepsilon_{\mathbf{k}} \psi_{\mathbf{k}},
\end{align*}
has support included in $[-1,1]^d$, and $g_0 \geq c >0$ on this support. As a result, it holds
\begin{align*}
    \left | \frac{g_{\varepsilon}}{g_0} - 1 \right | \leq 1/2
\end{align*}
as soon as  $c_1 \leq  c \|\psi\|_\infty^{-1} 2^{-dj/2 + 1}$. Since $\log \in \mathcal{H}_{C_\beta}^{\lceil \beta \rceil}((1/2;3/2))$, for some constant $C_\beta$, and $g_\varepsilon/g_0 \in \mathcal{H}_{C c_1 2^{j(\beta+d/2)}}^\beta(\R^d)$, the Faa di Bruno formula ensures that $\log(g_\varepsilon/g_0) \in \mathcal{H}_{C_\beta c_1 2^{j(\beta+d/2)}}^\beta(\R^d)$.  Thus choosing $c_1 := c_K 2^{-j(\beta + d/2)}$ for $c_K$ small enough yields that  $\log(g_{\varepsilon}/g_0) \in \mathcal{H}_{C_K}^\beta$. For such a $c_1$, it holds $g_\varepsilon = g_0 e^{a_\varepsilon}$, for $a_\varepsilon \in \mathcal{H}_{C_K}^\beta$. According to Proposition~\ref{prop:mixtures-in-model}, we deduce that $g_\varepsilon \in \mathcal{A}_K^\beta$.

Therefore, the class
$\Omega^{(j)} := \left \{ g_\varepsilon \right\}_{\varepsilon \in \{-1,1 \}^{m}}
$
is composed of densities, and we have $\Omega^{(j)} \subset \mathcal{A}_K^\beta$.

For $\varepsilon, \varepsilon' \in \{-1,1\}^{2^{dj}}$, let $\omega(\varepsilon, \varepsilon')$ denote the Hamming distance 
\begin{align*}
    \omega(\varepsilon, \varepsilon') = \sum_{\mathbf{k} \in \{0,\ldots,2^J-1\}^d} \1_{\varepsilon_\mathbf{k} \neq \varepsilon'_\mathbf{k}}
    .
\end{align*} 
We consider the test function
\begin{align*}
    f_{\varepsilon} :=
    \sum_{\mathbf{k} \in \{0,\ldots,2^J-1\}^d} \varepsilon_{\mathbf{k}} 2^{-j(d/2+1)} \frac{\Psi_{\mathbf{k}}}{\|\Psi\|_{\mathcal{H}^1}}
    .
\end{align*}
By construction, we have $\|f_\varepsilon\|_{\mathcal{H}^1} \leq 1$, and therefore $f_\varepsilon$ is $1$-Lipschitz. 
Consequently, the Wasserstein distance between $g_\varepsilon$ and $g_{\varepsilon'}$ is lower-bounded by
\begin{align*}
    \mathrm{W}_1(g_\varepsilon,g_{\varepsilon'}) & \geq \int_{\R^d} f_{\varepsilon}(u)(g_\varepsilon(u) - g_{\varepsilon'}(u))\mathrm{d}u \\
    & = \sum_{\mathbf{k} \in \{0,\ldots,2^J-1\}^d} \int_{\R^d} 2^{-j(d/2+1)} (\|\Psi\|_{\mathcal{H}^1})^{-1} \varepsilon_\mathbf{k} \psi_{\mathbf{k}}(u) c_1 (\varepsilon_{\mathbf{k}} -  \varepsilon'_{\mathbf{k}}) \psi_{\mathbf{k}}(u)\mathrm{d}u \\
    & = 2^{-j(d/2+1)} (\|\Psi\|_{\mathcal{H}^1})^{-1} c_1 \sum_{\mathbf{k} \in \{0,\ldots,2^J-1\}^d} 2 \1_{\varepsilon_{\mathbf{k}} \neq \varepsilon'_{\mathbf{k}}} \int_{\R^d} \psi^2_{\mathbf{k}}(u) \mathrm{d}u \\
    & = 2 c_K 2^{-j(\beta + d +1)} (\|\Psi\|_{\mathcal{H}^1})^{-1} \omega(\varepsilon,\varepsilon').
\end{align*}
Using the Varshamov-Gilbert Lemma~\cite[Lemma 2.9, p. 104]{Tsybakov08} provides the existence of a subset $\Omega^{(j)}_\ast \subset \Omega^{(j)}$ such that
\begin{itemize}
    \item $M_\ast:=|\Omega^{(j)}_\ast| \geq 2^{m/8}$;
    \item for all distinct $\varepsilon,\varepsilon' \in \Omega^{(j)}_\ast$, $\omega(\varepsilon,\varepsilon') \geq m/8$. 
\end{itemize}
Next, according to~\cite[Lemma 24 - Supplementary Material]{Uppal19}, it holds
\begin{align*}
    \mathrm{KL}(g_\varepsilon^{\otimes n} \mid g_{0}^{\otimes n}) 
    &=
    n
    \mathrm{KL}(g_\varepsilon \mid g_{0}) 
    \\
    &\leq 
    c n \| g_0 - g_\varepsilon\|^2_{L^2}/2 
    \\
    &\leq
    n c c_1^2  2^{dj} 
    \\
    &\leq 
    n c_K 2^{-2j \beta}
    ,
\end{align*}
where $g_\varepsilon^{\otimes n}$ denotes the $n$-times tensor product of $g_\varepsilon$ with itself. 
Choosing $j := \lfloor \log_2(n)/(2\beta +d) \rfloor$ thus yields
\begin{align*}
    \mathrm{KL}(g_\varepsilon^{\otimes n} \mid g_{0}^{\otimes n}) 
    &\leq
    c_K n^{\frac{d}{2\beta +d}} 
    \\
    &\leq
    \frac{1}{10} \log M_\ast,
\end{align*}
provided $c_K$ is small enough. 
A direct application of Fano-Birgé Lemma~\cite[Theorem~2.5]{Tsybakov08} hence entails 
\begin{align*}
    \inf_{\widehat{p}_n} \sup_{\varepsilon \in \Omega^{(j)}_\ast} 
    \E_{g_\varepsilon^{\otimes n}} \mathrm{W}_1(g_\varepsilon,\widehat{p}_n) 
    &\geq 
    (1/4) (2^{dj}/8)  (c_K 2^{-j(\beta + d +1)} / \|\Psi\|_{\mathcal{H}^1}) 
    \\
    &\geq 
    c'_K n^{- \frac{\beta+1}{2 \beta + d}}
    ,
\end{align*}
which yields the announced result. 
\end{proof}

\end{document}